\documentclass[leqno,12pt]{amsart}
\usepackage[utf8]{inputenc}
\usepackage{comment}
\usepackage{enumitem,amssymb,url}
\usepackage{amsmath,amsxtra,amssymb,amsfonts,latexsym, mathrsfs, dsfont,bm,enumitem, amsthm, mathtools}
\usepackage{enumitem}
\usepackage{color,soul}
\usepackage{cite}
\usepackage{geometry}[margin=1in]
\usepackage[colorlinks=true, pdfstartview=FitV, linkcolor=black, citecolor=blue, urlcolor=blue]{hyperref}
\usepackage{ulem}

\usepackage{enumitem,amssymb,url}
\newlist{todolist}{itemize}{2}
\setlist[todolist]{label=$\square$}
\usepackage{pifont}
\newtheorem{theorem}{Theorem}[section]
\newtheorem{proposition}[theorem]{Proposition}

\newtheorem*{acknowledgment}{Acknowledgment}
\newtheorem{problem}[theorem]{Problem}

\newtheorem{lemma}[theorem]{Lemma}

\theoremstyle{definition}
\newtheorem{definition}{Definition}[section]

\newtheorem{remark}[definition]{Remark}

\newtheorem{claim}[definition]{Claim}

\theoremstyle{definition}

\newcommand{\N}{\mathbb{N}}

\newcommand{\R}{\mathbb{R}}

\newcommand{\Rn}{\mathbb{R}^n}
\newcommand{\void}{\varnothing}

\newcommand{\diam}{\text{diam}}
\newcommand{\supp}{\text{supp}}
\newcommand{\dist}{\text{dist}}
\newcommand{\norm}[1]{\|#1\|}
\newcommand{\set}[1]{\left\{#1\right\}}
\newcommand{\abs}[1]{\left|#1\right|}

\newcommand{\p}{\mathcal{P}}
\newcommand{\D}{\partial}

\renewcommand{\d}{\partial}
\newcommand{\da}{\partial^\alpha}
\newcommand{\ma}{{m-\abs{\alpha}}}

\renewcommand{\P}{\mathcal{P}}

\newcommand{\G}{\Gamma}

\newcommand{\ring}{\mathcal{R}}
\newcommand{\pbar}{\overline{\mathcal{P}}}

\newcommand{\dmax}{\delta_{\text{max}}}

\newcommand{\Gbar}{\overline{\Gamma}}
\newcommand{\dm}{\delta_{\text{max}}}
\newcommand{\Gb}{\overline{\Gamma}}

\newcommand{\dbar}{\overline{d}}
\newcommand{\ksh}{{k^\sharp}}

\newcommand{\oP}{\overline{P}}
\newcommand{\cmw}{C^{m,\omega}}

\newcommand{\kb}{{\overline{k}}}
\newcommand{\eps}{\epsilon}

\newcommand{\LA}[1]{\refstepcounter{equation}\text{(\theequation)}\label{#1}}

\numberwithin{equation}{section}

\begin{document}

\title{Smooth Selection for Infinite Sets}
\author{Fushuai Jiang}
	\address{Department of Mathematics, UC Davis
		One Shields Ave
		Davis, CA 95616}
	\email{fsjiang@math.ucdavis.edu}

	\author{Garving K. Luli}
	\address{Department of Mathematics, UC     Davis
		One Shields Ave
		Davis, CA 95616}
	\email{kluli@math.ucdavis.edu}

	\author{Kevin O'Neill}
	\address{Applied Mathematics Program, Yale University
		51 Prospect Street
		New Haven, CT 06511}
	\email{kevin.oneill@yale.edu}
\maketitle

\begin{abstract}

Whitney's extension problem asks the following: Given a compact set $E\subset\mathbb{R}^n$ and a function $f:E\to \mathbb{R}$, how can we tell whether there exists $F\in C^m(\mathbb{R}^n)$ such that $F=f$ on $E$? A 2006 theorem of Charles Fefferman \cite{F06} answers this question in its full generality.

In this paper, we establish a version of this theorem adapted for variants of the Whitney extension problem, including nonnegative extensions and the smooth selection problems. Among other things, we generalize the Finiteness Principle for smooth selection by Fefferman-Israel-Luli \cite{FIL16} to the setting of infinite sets. 

Our main result is stated in terms of the iterated Glaeser refinement of a bundle formed by taking potential Taylor polynomials at each point of $E$. In particular, we show that such bundles (and any bundles with closed, convex fibers) stabilize after a bounded number of Glaeser refinements, thus strengthening the previous results of Glaeser, Bierstone-Milman-Paw{\l}ucki, and Fefferman which only hold for bundles with affine fibers.
\end{abstract}

\section{Introduction}


Let $m,n,d$ be positive integers. We write $C^{m}(\mathbb{R}^{n},\mathbb{R}%
^{d})$ to denote the space of all functions $\vec{F}:=( F_{1},\cdots
,F_{d}) :\mathbb{R}^{n}\rightarrow \mathbb{R}^{d}$ whose derivatives $%
\partial ^{\alpha }\vec{F}=( \partial ^{\alpha }F_{1},\cdots ,\partial
^{\alpha }F_{d}) $ (for all $|\alpha |\leq m$) are continuous and
bounded on $\mathbb{R}^{n}$. We equip $C^{m}(\mathbb{R}^{n},\mathbb{R}^{d})$
with the usual norm
\begin{equation*}
\Vert {\vec{F}}\Vert _{C^{m}(\mathbb{R}^{n},\mathbb{R}^{d})}:=\max
_{\substack{ |\alpha |\leq m,  \\ 1\leq j\leq d}}\sup_{x\in \mathbb{R}%
^{n}}| \partial ^{\alpha }F_{j}(x)|
\end{equation*}%
and seminorm 
\begin{equation*}
\Vert {\vec{F}}\Vert _{\dot{C}^{m}(\mathbb{R}^{n},\mathbb{R}^{d})}:=\max
_{\substack{ |\alpha |= m,  \\ 1\leq j\leq d}}\sup_{x\in \mathbb{R}%
^{n}}| \partial ^{\alpha }F_{j}(x)| .
\end{equation*}%

When $d=1$, we write $C^{m}(\mathbb{R}^{n})$ instead of $C^{m}(\mathbb{R}%
^{n},\mathbb{R}^{d})$.

Whitney's extension problem \cite{W34-2} asks, given a compact set $E\subset\mathbb{R}^n$ and a function $f:E\to \mathbb{R}$, how can we tell whether there exists $F\in C^m(\mathbb{R}^n)$ such that $F|_E=f$? The following nonnegative variation of Whitney's extension problem has attracted some recent attention (see \cite{FIL16,FIL16+,JL20,JL20-Ext,JL20-Alg}).

\begin{problem}[Nonnegative extension]\label{prob:nonnegative interpolation} Given integers $m\geq 0, n \geq 1$, an arbitrary subset $E \subset \mathbb{R}^n$, $f:E\rightarrow [0,\infty)$, how can we tell if there exists $F\in C^m(\mathbb{R}^n)$ such that $F \geq 0$ and $F|_E = f$?\end{problem}

Simple examples, such as $E=\{0\}\cup\{\frac{1}{n}:n\in\N\}$ with $f(x)=x$ and $m\ge1$, demonstrate that the existence of a nonnegative $C^m$ extension does not follow merely from nonnegative $f$ and the existence of a (not necessarily nonnegative) $C^m$ extension. Thus, a solution to Problem \ref{prob:nonnegative interpolation} must address an obstacle absent from the classical Whitney extension problem. Naturally, this obstacle extends to the further generalizations explored in this paper.

Problem \ref{prob:nonnegative interpolation} can be viewed as a special case of the following generalized problem (see \cite{fefferman2021c2} for recent progress on the analogous problem for finite sets):

\begin{problem}[Extension with restricted range]
\label{prob:range restriction} Given extended real numbers $\lambda_1, \lambda_2$ with $%
-\infty \leq \lambda_1 \leq \lambda_2 \leq \infty$, $E \subset \Rn$, and $f : E \to [\lambda_1,\lambda_2] \cap \R$, how can we tell if there exists $F
\in C^m(\Rn)$ such that $\lambda_1 \leq F \leq \lambda_2$ and $F|_E =f$?
\end{problem}


Our solution to Problems \ref{prob:nonnegative interpolation} and \ref{prob:range restriction} will result from the solution to an even more generalized problem. In order to present our solution to this problem, we need to introduce some background definitions and recall some facts.

If $\vec{F}\in C^{m}( \mathbb{R}^{n},\mathbb{R}^{d}) ,$ we write $J_{x}\vec{F}=( J_{x}F_{1},\cdots ,J_{x}F_{d}) $ (the ``$m$-th
jet") to denote the $m$-th degree Taylor polynomial of $\vec{F}$ at $x$. 
Let $\mathcal{P}$ denote the vector space of polynomials in $n$ real variables and of degree at most $M$.
Given $x\in\R^n$, we give $\mathcal{P}$ the structure of a ring, denoted $\mathcal{R}_x$, by defining multiplication $\odot_x$ as $R\odot _{x}R'=J_{x}(RR')$.  We write $
\vec{\mathcal{R}}_{x}$ for the $\mathcal{R}_x$-module of $m$-jets of functions $\vec{F}\in C^{m}(\mathbb{R}^{n},\mathbb{R}^{d}) $ at $x\in\mathbb{R}^{n}$. As a vector space, $\vec{\mathcal{R}}_{x}$ can be identified with $\vec{\mathcal{P}}:=\underbrace{\mathcal{P}\oplus\cdots\oplus\mathcal{P}}_{d\text{ copies}}$, the
space of $d$-tuples of real $m$-th degree polynomials on $\mathbb{R}^{n}$. The $\mathcal{R}_x$-multiplication on $\vec{\mathcal{R}}_x$ is given by $R\odot _{x}\vec{P}=\left(R\odot _{x}P_{1},\cdots ,R\odot _{x}P_{d}\right) $, where $R\in\mathcal{R}_x$ and $\vec{P}=(P_1,...,P_d)$.

We write $c(m,n,d), C(m,n,d),$ etc., to denote constants that depend on $%
m,n,d$; they may denote different constants in different appearances.

If $S$ is a finite set, we write $\#(S)$ to denote the number of elements in
$S$.

A problem that is closely related to Problems \ref{prob:nonnegative interpolation} and \ref{prob:range restriction} is the following 

\begin{problem}[Smooth Selection]
\label{prob:convex selection}Let $E\subset \mathbb{R}^{n}$ be an arbitrary set. For
each $x\in E$, let $\mathcal{K}(x) \subset \mathbb{R}^{d}$ be closed and
convex. How can we tell whether there exists a function $\vec{F}\in C^{m}(\Rn,\R^d) $ such that $\vec{F}(x) \in \mathcal{K}(x) $ for all $x\in E$?
\end{problem}

The pioneer papers \cite{Shv08,FIL16,F05-Sh} studied the analogous version of Problem \ref{prob:convex selection} for finite sets $E$.  In particular, for finite sets $E\subset \mathbb{R}^{n}$, the authors in \cite{FIL16} gave an answer to the finite set version of the problem by means of proving a finiteness principle. See also \cite{FShv18,Shv21-core, Shv21-alg,Shv84,Shv01,Shv02,Shv04} for related Lipschitz selection problems and \cite{FIL16+,JL20,JL20-Alg,JL20-Ext} for related problems of nonnegative interpolation.


The $C^1$ case of nonnegative extension (Problem \ref{prob:nonnegative interpolation}) was addressed in \cite{Black19}, though little other work has directly addressed the infinite set versions of the constrained extension problems addressed here.

In this paper, we modify the approach of \cite{F06} to answer Problem \ref{prob:convex selection} for infinite sets $E\subset \mathbb{R}^n$.

Before we present our solution to the above problems, we would like to point out another closely related problem for which our result can provide an answer:

\begin{problem}[Generalized Brenner-Epstein-Hochster-Koll\'ar Problem]
\label{prob:linear system} 
Suppose we are
given real-valued functions $\phi_1,\cdots,\phi_d$, an $\R^s$-valued
function $\phi$ on $\R^n$, and closed convex subsets $K_1(x),\cdots,K_d(x)
\subset \mathbb{R}^s$ for each $x \in E\subset\R^n$. How can we decide whether there
exist $f_1,\cdots, f_d \in C^m(\R^n,\R^s)$ such that
\begin{equation}
\sum_{i=1}^d \phi_if_i \leq \phi \text{ on } E \label{system_eqs}
\end{equation}
and
\begin{equation}
f_i(x) \in K_i(x), \text{ for } 1\leq i \leq d, x \in E?
\label{system_constraints}
\end{equation}
\end{problem}

The pioneer papers \cite{EH18,FK13,FL14} studied Problem \ref{prob:linear system} where the inequalities in \eqref{system_eqs} are replaced by equations and 
there are no
convex constraints \eqref{system_constraints}. Note that the inequalities in \eqref{system_eqs} may be used to restrict the output of the $f_i$ to convex sets depending on $x$, provided said convex sets are the intersection of a bounded number of half-spaces.

In order to present our solutions to Problems \ref{prob:range restriction} (\textit{a fortiori} \ref{prob:nonnegative interpolation}), \ref{prob:convex selection}, and \ref{prob:linear
system}, we introduce some basic definitions.

Let $E\subset \mathbb{R}^{n}$ be a compact set. A \underline{bundle} over $E$
is a family $\mathcal{H}=(H(x))_{x\in E}$ of (possibly empty) subsets $H(x)\subset \vec{\mathcal{P}}$, parameterized by the points $x\in E$. We refer to each $H(x)$, $x\in E$, as a \underline{fiber}. A \underline{$C^m$ section} (or ``section'' for short) of a bundle $\mathcal{H} = (H(x))_{x \in E}$ is a $C^m$ function $\vec{F}:\R^n\to\R^d$ such that $J_x \vec{F} \in H(x)$ for each $x \in E$. The problem of finding a $C^m$ extension satisfying certain conditions reduces to the problem of finding a section of a given bundle (see below).

We say a bundle $(H(x))_{x\in E}$ is \underline{convex} if $H(x)$ is convex for each $x\in E$. (We note that prior literature has required each fiber of a bundle to be an affine space; we drop the requirement to handle nonlinear problems.)

\label{def.Glaeser} Fix integers $m\ge0, n,d\ge 1$. Let $(H(x))_{x \in E}$
be a bundle and let ${k^\sharp}$ be a positive integer depending only on $m,n,d$.
For each $x_0 \in E$, we define the \underline{Glaeser refinement} of $%
H(x_0) $, denoted by $\tilde{H}(x_0)$, according to the following rule:

\begin{itemize}
\item[\LA{GR}] Let $\vec{P}_{0}\in \vec{\mathcal{P}}$. We say that $\vec{P}%
_{0}\in \tilde{H}(x_{0})$ if and only if given $\epsilon >0$, there exists $%
\delta >0$ such that, for any $x_{1},\cdots ,x_{{k^{\sharp }}}\in E\cap
B(x_{0},\delta )$, there exist $\vec{P}_{1},\cdots ,\vec{P}_{{k^{\sharp }}%
}\in \vec{\mathcal{P}}$, with $\vec{P}_{j}\in H(x_{j})$ for $j=0,1,\cdots ,{%
k^{\sharp }}$ and $\left\vert \partial ^{\alpha }(\vec{P}_{i}-\vec{P}%
_{j})(x_{j})\right\vert \leq \epsilon \left\vert x_{i}-x_{j}\right\vert ^{{%
m-\left\vert \alpha \right\vert }}$ for $\left\vert \alpha \right\vert \leq
m $, $0\leq i,j\leq {k^{\sharp }},1\leq l\leq d$.
\end{itemize}

The Glaeser refinement is used, essentially, to toss polynomials which cannot satisfy the conditions of Taylor's theorem with the polynomials remaining at nearby fibers.

We say that a bundle $(H(x))_{x\in E}$ is \underline{Glaeser stable} if $%
H(x) $ is its own Glaeser refinement for each $x\in E$.

We will often produce repeated Glaeser refinements of the same bundle, in
which case $(H_{0}(x))_{x\in E}$ will be used to denote the original bundle and $%
(H_{k+1}(x))_{x\in E}$ will be used to denote $(\tilde{H_{k}}(x))_{x\in E}$ for $k\in \N$.

The Glaeser refinement is a procedure first suggested in 1958 by G. Glaeser\cite{G58} for determining whether a function $f$ defined on an arbitrary subset $E\subset \mathbb{R}^{n}$ can be extended to a $C^{1}( \mathbb{R}^{n}) $ function. In their studies of extension problems in $C^{m}( \mathbb{R}^{n}) $ for subanalytic sets, Bierstone, Milman and Paw{\l}ucki \cite{BMP03} introduced the notion of iterated paratangent bundles, which is analogous to the Glaeser refinement. Our version of the Glaeser refinement is a vector-valued version of the one used by C. Fefferman for extending a function $f$ defined on an arbitrary subset $E\subset \mathbb{R}^{n}$ to a $C^{m}(\mathbb{R}^{n}) $ function. We also apply our Glaeser refinement to a more general class of bundles than in \cite{F06}.

Here is an outline for the rest of the introduction. First, we briefly
recall Fefferman's solution \cite{F06} to the Whitney's extension problem for $C^{m}$ and then we will point out the major hurdles in adapting the machinery to prove our main theorems (see Theorems \ref{thm:main theorem}, \ref{thm:heart of the matter} below). We will then state our main results and explain our proof in broad strokes.

Given a function $f:E\rightarrow \mathbb{R}^{n}$ on an arbitrary subset $E\subset \mathbb{R}^{n}$, to determine whether $f$ can be extended to a function $F\in C^{m}( \mathbb{R}^{n}) $ -- This is called the \textquotedblleft Whitney Extension Problem for $C^{m}$\textquotedblright -- C. Fefferman \cite{F06}
associated an affine subspace $H(x) \subset \mathcal{P}$
for each $x\in E$, with the following crucial property:

\begin{itemize}
\item If $F\in C^{m}\left( \mathbb{R}^{n}\right) $ and $F=f$ on $E,$ then $%
J_{x}\left( F\right) \in H(x) $ for all $x\in E$.
\end{itemize}

Evidently, if $H(x) =\emptyset $ for any $x\in E$, then $f$
cannot be extended to a $C^{m}$ function $F$. Fefferman's solution \cite{F06} to the
Whitney's extension problem for $C^{m}$ can be summarized as follows

\begin{enumerate}
\item We start with the trivial holding space $H_{0}(x) =\left\{
P\in \mathcal{P}:P(x) =f(x) \right\} $ for all $x\in
E$.

\item We produce a list of holding spaces $H_{0}(x) \supset
H_{1}(x) \supset H_{2}(x) \supset \cdots $ by
repeatedly taking the Glaeser refinement (see Definition \ref{def.Glaeser}).

\item After at most $L=2\dim \mathcal{P}+1$, we have $H_{L}(x)
=H_{l}(x) $ for all $l\geq L$ and all $x\in E$. (In other words, $(H_L(x))_{x\in E}$ is Glaeser stable.)

\item $f$ extends to a $C^{m}$ function on $\mathbb{R}^{n}$ if and only if $%
H_{L}(x) $ is non-empty for all $x\in E$.
\end{enumerate}

Fefferman's proof of the solution relies on the affine space structure of the $H(x)$ (and breaks down for more general sets) in three crucial ways:

\begin{enumerate}
    \item The proof that iterated Glaeser refinements terminate in a Glaeser stable bundle after a finite number of steps (Step 3) in \cite{F06} (adapted from \cite{BMP03}) relies on the affine space structure of $H_{l}$ and an induction on the dimensions. It uses the fact that if $W\subset V$ is a subspace with $\dim W=\dim V$, then $W=V$, which is false for the natural notion of dimension for convex sets.
    \item The proof that the $\delta$ in the definition of Glaeser stability may be chosen uniformly in $x$ and $P$ (with some adjustment) relies on the finite representation of $H(x)$ via a linear basis. General convex sets may not be represented in terms of a finite basis.
    \item The set $E$ is partitioned according to the dimension of the $H(x)$ as affine spaces. If one does this by dimension of a convex set, it must be checked this dimension does not change under natural set operations.
\end{enumerate}

Our circumstances do not guarantee the presence of affine spaces. For instance, considering Problem \ref{prob:nonnegative interpolation} we see that $x^2$ and $2x^2$ are potential Taylor polynomials at $x=0$, while $-x^2$ is not.


To establish the termination of iterated Glaeser refinements of bundles $(H(x))_{x \in E}$ whose fibers $H(x)$ are allowed to be general convex sets, we recognize the following: The Glaeser refinement depends on a local procedure (see Definition \ref{def.Glaeser}). More precisely, in order to check if $\vec{P}_{0}\in \tilde{H}(x_{0})$, we need to consider only polynomials $\vec{P}_{i}$ that are close to $\vec{P}_{0}$. Furthermore, convex subsets of $\mathbb{R}^n$ have a notion of dimension which allows us to view them as affine spaces locally. To formalize this idea, we introduce the \textit{local dimension} in Definition \ref{def:local dimension}.

To find uniform $\delta$ as dictated by (2) above, we pass from the infinite to the finite via compactness of $E$ and the $H(x)$, truncating the latter as necessary. An open cover of neighborhoods in $E\times\p$ must be carefully constructed so that each of these neighborhoods allows for uniform choice of $\delta$ (for given $\epsilon$) and for the necessary quantitative analysis to be conducted. This delicate construction is done in Section \ref{sec:uniform}. Sorting the elements $x$ of $E$ by $\dim H(x)$ is initially easy; any convex set naturally caries a dimension. However, we must show that the truncations described above do not change the dimension and ruin the partition. This is done in Section \ref{sec:convex}.

To state our main theorems, we need to introduce a few terms.

 Let $E\subset \mathbb{R}^{n}$ be an arbitrary set. For each $%
x\in E$ and $M\geq 0$, let $\Gamma \left( x,M\right) \subset \mathcal{\vec{P}%
}$ be a (possibly empty) convex set. We say $\Gamma =(\Gamma \left(
x,M\right) )_{x\in E,M\geq 0}$ is a \underline{shape field} if for all $x\in
E$ and $0\leq M^{\prime }\leq M<\infty $, we have
$\Gamma (x,M^{\prime })\subset \Gamma \left( x,M\right)$.
We say a shape field $\Gamma =(\Gamma \left( x,M\right) )_{x\in E,M\geq 0}$
is \underline{closed} if $\Gamma (x,M)$ is closed for each $x\in E$ and $%
M\geq 0$.

Any shape field $(\Gamma(x,M))_{x \in E, M \geq 0}$ on a compact set $E$
gives rise to a convex bundle $(\Gamma(x))_{x \in E}$ if we set
\begin{equation*}
\Gamma(x) = \bigcup_{M\geq 0}\Gamma(x,M) \text{ for each } x \in E.
\end{equation*}
We may then define the Glaeser refinement of a shape field $(\Gamma(x,M))_{x\in E, M \geq 0}$ to be $(\tilde{\Gamma}(x,M))_{x\in E, M \geq 0}$, where $\tilde{\Gamma}(x,M) = \tilde{\Gamma}(x)\cap \Gamma(x,M)$ with $\tilde{\Gamma}(x)$ being the Glaeser refinement of $\Gamma(x)$.

We will also define what it means for a shape field to be regular (see Definition \ref{def.regular}) or $(C,\dmax)$-convex for $\dmax>0$ (see Definition \ref{def.Cd-convex}), though in the interest of postponing technicalities this will be presented in Section \ref{sec:applications}. For now, we simply note that regularity ensures $\Gamma(x,M)$ ``varies continuously'' in $x$ and $M$, while $(C,\dmax)$ convexity allows us to combine functions via partition of unity as in \cite{FIL16}. Both properties will be satisfied by reasonable examples of shape fields, in particular, the examples arising from Problems \ref{prob:nonnegative interpolation}-\ref{prob:linear system}.

We are now ready to state our main theorems.

\begin{theorem}[Qualitative Main Theorem]
\label{thm:main theorem} Let $m\ge0,n,d\ge1$ be integers. There exists
$k^\sharp=k^\sharp(m,n,d)$ in \eqref{GR} such that the following holds.

Let $E\subset \mathbb{R}^{n}$ be a closed set. Let $%
(\Gamma (x))_{x\in E}$ be a bundle arising from a closed regular $(C,1)$%
-convex shape field. Then the iterated Glaeser refinement of $(\Gamma(x))_{x\in E}$ terminates in a Glaeser stable bundle $(\Gamma^*(x))_{x\in E}$ in at most $2\dim\vec{\mathcal{P}}+1$ steps and $(\Gamma (x))_{x\in E}$ has a section if and if every fiber of $(\Gamma^*(x))_{x\in E}$ is nonempty.
\end{theorem}

We can further strengthen Theorem \ref{thm:main theorem} to have better control
on the derivatives. In fact, this step is required for our proof strategy. We need the following definition for this purpose.

Given a shape field $\Gamma =(\Gamma
(x,M))_{x\in E,M\geq 0}$ with nonempty fibers, we define
\begin{equation}\label{def:norm of a bundle} 
    \norm{\Gamma}:= \inf\set{
    M \geq 0 :\,
  \begin{matrix*}[l]
  \text{For all $x_1, \cdots, x_\ksh \in E$, there exist }\\
  \text{$\vec{P}_i \in \Gamma(x_i,M)$ for $1\leq i\leq \ksh$ such that }\\
  \text{$\abs{\da(\vec{P}_i-\vec{P}_j)(x_i)}\leq M\abs{x_i-x_j}^{m-\abs{\alpha}}$, $1 \leq i,j \leq \ksh$.}
  \end{matrix*}
}\,.
\end{equation}

The finiteness of $\|\Gamma\|$ will be established in Lemma \ref{lemma:finiteness}. (Specifically, we prove this for the scalar-valued case in Section \ref{sec:prelim} and reduce to this case in Section \ref{sec:gradient trick}.)

\begin{theorem}[Quantitative Main Theorem]
\label{thm:heart of the matter} Let $m\ge0,n,d\ge1$ be integers. There exists
$k^\sharp=k^\sharp(m,n,d)$ in \eqref{GR} and \eqref{def:norm of a bundle} such that the following holds:

Fix $C_{w}>0$. Let $Q_{0}\subset \mathcal{\mathbb{R}}^{n}$ be a cube of
length $3$ and $E\subset Q_{0}$ be compact. Let $(\Gamma (x,M))_{x\in
E,M\geq 0}$ be a closed regular $(C_{w},1)$-convex shape field. Write
\begin{equation}
\Gamma (x)=\bigcup_{M\geq 0}\Gamma (x,M).
\end{equation}%
Then the following hold.

\begin{enumerate}
\item The repeated Glaeser refinement of $(\Gamma (x))_{x\in E}$ terminates
in a Glaeser stable bundle $(\Gamma ^{\ast }(x))_{x\in E}$ after at most $%
2D+1$ steps, where $D=D(m,n,d)=\dim \vec{\mathcal{P}}$.

\item If $\Gamma ^{\ast }(x)$ is nonempty for all $x\in E$, then there
exists $\vec{F}\in C^{m}(\mathcal{\mathbb{R}}^{n},\mathcal{\mathbb{R}}^{d})$
such that
\begin{equation}
\Vert \vec{F}\Vert _{\dot{C}^{m}(\mathcal{\mathbb{R}}^{n},\mathcal{\mathbb{R}%
}^{d})}\leq C\Vert \Gamma ^{\ast }\Vert
\end{equation}%
and
\begin{equation}
J_{x}\vec{F}\in \Gamma (x,C\Vert \Gamma ^{\ast }\Vert )\text{ for all }x\in
E.
\end{equation}
\end{enumerate}

Here the constant $C$ depends only on $m,n,d,C_{w}$.
\end{theorem}

In Section \ref{sec:applications}, we will apply Theorem \ref{thm:heart of the matter} to present
our solutions to Problems \ref{prob:convex selection},
\ref{prob:range restriction}, and \ref{prob:linear system}. In Section \ref{sec:termination lemma}, we will prove the first part of Theorem \ref{thm:heart of the matter}, using the aforementioned concept of local dimension. Section \ref{sec:gradient trick} contains a reduction to the case $d=1$, using the gradient trick as in \cite{FL14}.

In Sections \ref{sec:fin princ} through \ref{sec:main proof}, we prove the second half of Theorem \ref{thm:heart of the matter}. We show that for given $\eps>0$, we may find a uniform $\delta$ in the statement that $(\Gamma^*(x))_{x\in E}$ is Glaeser stable and use this to construct a modulus of continuity for which we may apply the $C^{m,\omega}$ shape fields finiteness principle from \cite{FIL16}. The proof that $\delta$ may be taken uniformly will rely on the dimension of $\Gamma(x)$ being the same for all $x\in E$, forcing us to decompose $E$ into smaller sets (strata) for which this property holds. The solutions are then patched together via a partition of unity. The process is formally conducted via induction on the number of strata.


Section \ref{sec:improvement} is not part of the proof of Theorem \ref{thm:heart of the matter}; however, it will provide a drastic decrease in the value of $\ksh$ from that given in the proof of the theorem (greater than $(\dim\vec{\P}+1)^{3\cdot 2^{\dim\vec{\P}}}$) to $\ksh=2^{\dim\vec{P}}$.

This is an overly simplified version of our
long story. The details will be presented in the sections below.

This paper is part of a literature on extension and interpolation, going
back to the seminal works of H. Whitney \cite{W34-1,W34-2,W34-3}. We refer the interested readers
to \cite{FK09-Data-1,FK09-Data-2,F09-Data-3,JL20,FIL16,FIL16+,FI20-book,fefferman2021cm} and references therein for the history and related problems.

\begin{acknowledgment} We are grateful to Charles Fefferman and Pavel Shvartsman for valuable discussions and suggestions. 

The first author is supported by the UC Davis Summer Graduate Student Researcher Award and the Alice Leung Scholarship in Mathematics. The second author
is supported by NSF Grant DMS-1554733 and the UC Davis Chancellor’s Chancellor's Fellowship.
\end{acknowledgment}

\section{Applications}\label{sec:applications}

In this section, we explain how to use Theorem \ref{thm:main theorem} to answer Problems \ref{prob:convex selection}, \ref{prob:range restriction}, and \ref{prob:linear system}.

By specializing to $d = 1$, $E = \Rn$, and
\begin{equation*}
    \mathcal{K}(x) = \begin{cases}\set{f(x)} \text{ for } x \in E_0 \\ [\lambda_1,\lambda_2]\cap \R \text{ for }x \in E\setminus E_0 \end{cases}\,,
\end{equation*}
we see that Problem \ref{prob:convex selection} encompasses Problem \ref{prob:range restriction}. 

By specializing to $D = ds$, $E = E_0$ and 
\begin{equation*}
    \mathcal{K}(x) = \set{(y_1, \cdots, y_d) \in \prod_{i = 1}^d K_i(x) : y_i \in \R^s \text{ for } i = 1, \cdots, d, \text{ and } \sum_{i = 1}^d \phi_i(x)y_i  \leq \phi(x)},
\end{equation*}
we see that Problem \ref{prob:convex selection} encompasses Problem \ref{prob:linear system}.

Thus, to solve Problems \ref{prob:range restriction} and \ref{prob:linear system}, it suffices to solve Problem \ref{prob:convex selection}.

In order to show Theorem \ref{thm:main theorem} applies to Problem \ref{prob:convex selection}, we must provide the definitions missing from the introduction.

\begin{definition}
\label{def.regular} Given a compact set $E\subset \mathbb{R}^{n}$, we say a
shape field $(\Gamma (x,M))_{x\in E,M\geq 0}$, with $\Gamma
(x)=\bigcup_{M\geq 0}\Gamma (x,M)$, is \underline{regular} if the following
hold.

\begin{itemize}
\item[\LA{regularity 1}] Given $M\geq 0, x\in E$,
\begin{equation*}
\Gamma(x,M)\subset \{\vec{P}\in\mathcal{P}:|\partial^\alpha \vec{P}(x)|\le M
\text{ for }|\alpha|\le m\}.
\end{equation*}

\item[\LA{regularity 2}] Given $\epsilon >0,M\geq 0$, there exists $\delta
>0 $ such that for any $x,x^{\prime }\in E$ with $\left\vert x-x^{\prime
}\right\vert \leq \delta $, if $\vec{P}\in \Gamma (x,M)$ and $\vec{P}%
^{\prime }\in \Gamma (x^{\prime })$ with
\begin{equation*}
\abs{\partial^{\alpha }(\vec{P}-\vec{P}')(x)} ,\left\vert \partial ^{\alpha }(\vec{P}-\vec{P}')(x') \right\vert \leq \delta \left\vert x-x'\right\vert ^{m-|\alpha |}\text{ for }|\alpha |\leq m,
\end{equation*}%
then $\vec{P}^{\prime }\in \Gamma (x^{\prime },M+\epsilon )$.

\item[\LA{regularity 3}] Given $\epsilon >0,M\geq 0$, there exists $\delta
>0 $ such that if $\vec{P}\in \Gamma (x,M)$, then for any $\vec{P}^{\prime
}\in \Gamma (x)$ satisfying
\begin{equation*}
|\partial ^{\alpha }(\vec{P}-\vec{P}^{\prime })(x)|<\delta \text{ for }%
|\alpha |\leq m
\end{equation*}%
implies
\begin{equation*}
\vec{P}^{\prime }\in \Gamma (x,M+\epsilon ).
\end{equation*}
\end{itemize}
\end{definition}

Note that the above conditions say nothing about the content of $\Gamma (x)$%
, rather the values of $M$ for which a given $\vec{P}\in \Gamma (x)$ lies in
$\Gamma (x,M)$.

\begin{definition}
\label{def.Cd-convex} Let $\Gamma=\left( \Gamma \left( x,M\right)
\right) _{x\in E,M\geq 0}$ be a shape field. Let $C_{w}>0$ and $\delta
_{\max }\in (0,1]$. We say that $\Gamma$ is \underline{$(C_{w},\delta
_{\max })$-convex}, if the following condition holds.

Let $0<\delta \leq \delta _{\max }$, $x\in E,M\geq 0$, $\vec{P}_{1},\vec{P}%
_{2}\in \mathcal{\vec{P}}$, $Q_{1},Q_{2}\in \mathcal{P}$. Assume that
\begin{align}
    &\vec{P}_{1},\vec{P}_{2}\in \Gamma \left( x,M\right) ;  \label{eq:C-delta 1}\\
    &\left\vert \partial ^{\alpha }\left( \vec{P}_{1}-\vec{P}_{2}\right) \left(
x\right) \right\vert \leq M\delta ^{m-\left\vert \alpha \right\vert }\text{
for }\left\vert \alpha \right\vert \leq m;  \label{eq:C-delta 2}\\
&\left\vert \partial ^{\alpha }Q_{i}(x) \right\vert \leq \delta
^{-\left\vert \alpha \right\vert }\text{ for }\left\vert \alpha \right\vert
\leq m,i=1,2;\text{ and}  \label{eq:C-delta 3}\\
&Q_{1}\odot_x Q_{1}+Q_{2}\odot_x Q_{2}=1.  \label{eq:C-delta 4}
\end{align}
Then
\begin{equation}
\sum_{i=1,2}Q_{i}\odot_x Q_{i}\odot_x \vec{P}_{i}\in \Gamma (x,C_{w}M).
\label{eq:C-delta 5}
\end{equation}
\end{definition}

It is clear that any $(C_{w},\delta _{\max })$-convex shape field is also $%
(C_{w}^{\prime },\delta _{\max }^{\prime })$-convex for any $C_{w}^{\prime
}\geq C_{w}$ and $0<\delta _{\max }^{\prime }\leq \delta _{\max }$.


\subsection{Solution to Problem \ref{prob:convex selection}}

\newcommand{\GK}{\Gamma_{\mathcal{K}}}
Using a standard argument involving partition of unity, we may assume that $E$ is compact.

Consider the bundle $(\Gamma_{\mathcal{K}}(x))_{x \in E}$ with fiber 
\begin{equation}
    \Gamma_{\mathcal{K}}(x) = \bigcup_{M \geq 0}\Gamma_{\mathcal{K}}(x,M),
    \label{eq.GK-def-1}
\end{equation}
where
\begin{equation}
    \Gamma_{\mathcal{K}}(x,M) = \set{\vec{P} \in \vec{\P} :\,\, \begin{matrix*}[l]
    \text{$\vec{P}(x) \in \mathcal{K}(x)$ and $\abs{\da\vec{P}(x)}\leq M$ for $\abs{\alpha}\leq m$.}
    \end{matrix*}}.\label{eq.GK-def-2}
\end{equation}
A $C^m$-section of the bundle $(\GK(x))_{x\in E}$ is precisely a $C^m(\Rn,\R^d)$ function $\vec{F}$ such that $\vec{F}(x)\in \mathcal{K}(x)$ for each $x \in E$.

To find a $C^m$-section of $(\GK(x))_{x\in E}$, we use Theorem \ref{thm:heart of the matter}. In particular, we will prove the following.

\begin{lemma}\label{lemma:application}
$ \left(\Gamma_{\mathcal{K}}(x,M)\right)_{x\in E, M \geq 0}$ is a closed, regular, $(C,1)$-convex shape field.
\end{lemma}

\begin{proof}

We check each condition one at a time.

\textbf{Shape field.} It is clear that $\left(\Gamma_{\mathcal{K}}(x,M)\right)_{x\in E, M \geq 0}$ is a closed shape field.

\textbf{Regularity.} Now we show that $\left(\Gamma_{\mathcal{K}}(x,M)\right)_{x\in E, M \geq 0}$ is regular as in Definition \ref{def.regular}. In particular, we want to show that $\left(\Gamma_{\mathcal{K}}(x,M)\right)_{x\in E, M \geq 0}$ satisfies conditions \eqref{regularity 1}--\eqref{regularity 3}.

It is clear that $\left(\Gamma_{\mathcal{K}}(x,M)\right)_{x\in E, M \geq 0}$ satisfies condition \eqref{regularity 1}. 

Now we check condition \eqref{regularity 2}. Let $\eps > 0$ and $M \geq 0$ be given. Let $\delta= \delta(\eps,M)$ be a small number to be chosen. Let $x,x' \in E$ with $\abs{x-x'}\leq \delta$, $\vec{P} \in \GK(x,M)$, and $\vec{P}' \in \GK(x') = \bigcup_{M\geq 0}\GK(x',M)$ satisfy
\begin{equation}
    \abs{\da(\vec{P}-\vec{P}')(x)}, \abs{\da(\vec{P}-\vec{P}')(x)} \leq \delta\abs{x-x'}^{m-\abs{\alpha}}
    \text{ for }\abs{\alpha}\leq m
    \label{eq.reg-01}
\end{equation}

Since $\vec{P}' \in \GK(x')$, we have 
\begin{equation}
\vec{P}'(x')\in \mathcal{K}(x').
    \label{eq.reg-010}
\end{equation}

By Taylor's theorem and the fact that $\vec{P} \in \GK(x,M)$ is an $m$-jet, we have
\begin{equation}
    \abs{\da\vec{P}(x')} \leq 
    \begin{cases}
    M(1+C\abs{x-x'}^{m-\abs{\alpha}}) \leq M(1+C\delta^{m-\abs{\alpha}}) &\text{ for $\abs{\alpha}< m$}\\
    M &\text{ for $\abs{\alpha}=m$}
    \end{cases},
    \label{eq.reg-02}
\end{equation}
where $C=C(m,n,d)$.

We see from \eqref{eq.reg-01} and \eqref{eq.reg-02} that 
\begin{equation}
\begin{split}
    \abs{\da\vec{P}'(x')} &\leq \abs{\da(\vec{P}-\vec{P'})(x')} + \abs{\da\vec{P}(x')}  \\
    &\leq 
    \begin{cases}
    M(1+C\delta^{m-\abs{\alpha}}) + \delta^{m+1-\abs{\alpha}}
    &\text{ for $\abs{\alpha}<m$}
     \\
    M+\delta &\text{ for $\abs{\alpha}=m$}
    \end{cases}.
    \end{split}
    \label{eq.reg-03}
\end{equation}

In view of \eqref{eq.reg-010} and \eqref{eq.reg-03} while choosing $\delta$ to be sufficiently small in a manner dependent only on $M$ and $C=C(m,n,d)$, we can conclude that $\vec{P}' \in \GK(x',M+\eps)$.

We now turn to condition \eqref{regularity 3}. 

Let $\eps > 0$ and $M\geq 0$. Let $\delta \in (0,\eps)$. Let $x \in E$, $\vec{P} \in \GK(x,M)$, $\vec{P}' \in \GK(x) = \bigcup_{M\geq 0}\GK(x,M)$. Assume that $\abs{\da(\vec{P}-\vec{P}')(x)} \leq \delta < \eps$. We immediately see that $\vec{P}' \in \GK(x,M+\eps)$. Condition \eqref{regularity 3} is satisfied.

Therefore, $\left(\Gamma_{\mathcal{K}}(x,M)\right)_{x\in E, M \geq 0}$ is regular.

\textbf{$(C,1)$-convexity.} Now we show that $(\GK(x,M))_{x\in E, M \geq 0}$ is $(C,1)$-convex. Fix the following:
\begin{align}
    \label{eq:Cd-convex-1}&\delta \in (0,1],\, x \in E, \, M \geq 0;\\
    \label{eq:Cd-convex-2}&\vec{P}_1, \vec{P}_2 \in \Gamma_{\mathcal{K}}(x,M)
    \text{ with } \abs{\da(\vec{P}_1 - \vec{P}_2)(x)} \leq M\delta^{m-\abs{\alpha}}
    \text{ for }\abs{\alpha}\leq m;\\
    \label{eq:Cd-convex-3}&Q_1, Q_2 \in \P \text{ with }
    \abs{\da Q_i(x)}\leq \delta^{-\abs{\alpha}}
    \text{ and }Q_1\odot_x Q_1 + Q_2\odot_x Q_2 = 1.
\end{align}
We want to show that
\begin{equation}
    \vec{P}:= \sum_{i = 1,2}Q_i\odot_x Q_i\odot_x \vec{P}_i \in \Gamma_{\mathcal{K}}(x,CM). 
    \label{eq:Cd-convex-0}
\end{equation}

Thanks to \eqref{eq:Cd-convex-3}, we have $\vec{P}(x)\in \mathcal{K}(x)$. Thanks to \eqref{eq:Cd-convex-2} and \eqref{eq:Cd-convex-3}, we have $\abs{\da\vec{P}(x)} \leq CM$. Therefore, \eqref{eq:Cd-convex-0} holds.

\end{proof}

\section{Termination Lemma}\label{sec:termination lemma}

Let $D=\dim\vec{\P}$. The purpose of this section is to prove the following lemma, which serves as part 1 of Theorem \ref{thm:heart of the matter}.

\begin{lemma}[Termination of Glaser Refinement]\label{lemma:termination of Glaeser refinement}
If $(H(x))_{x\in E}$ is a convex bundle, then $H_l(x)=H_{2D+1}(x)$ for all $l\ge 2D+1$.
\end{lemma}

The origin of Lemma \ref{lemma:termination of Glaeser refinement} goes back to Glaeser \cite{G58}. Bierstone-Milman and Fefferman also adapted it in their works \cite{BMP03} and \cite{F05-L}.

Before beginning the proof of Lemma \ref{lemma:termination of Glaeser refinement}, we will show that the convexity hypothesis on the holding spaces extends to the Glaeser refinements.

\begin{lemma}\label{lemma:convexity for all}
If $(H_0(x))_{x\in E}$ is a convex bundle, then $(H_l(x))_{x\in E}$ is a convex bundle for all $l \ge 1$. 
\end{lemma}

\begin{proof}

\newcommand{\vp}{\vec{P}}
\newcommand{\vq}{\vec{Q}}
\newcommand{\vr}{\vec{R}}

We induct on $l \geq 0$. The case $l = 0$ is trivial by assumption.

Let $l\ge 1$. Suppose $H_{l-1}(x)$ is convex for each $x \in E$. We want to show that $H_{l}(x)$ is convex for each $x \in E$.

Fix $x_0 \in E$. Let $\vp_0, \vq_0 \in H_{l}(x_0)$. Let $\theta \in [0,1]$, and set \begin{equation*}
    \vr_0:= (1-\theta)\vp_0 + \theta \vq_0.
\end{equation*}
We want to show that $\vr_0 \in H_{l}(x_0)$.

Let $\epsilon > 0$. Since $\vp_0, \vq_0 \in H_{l}(x_0)$, there exists $\delta > 0$ such that, for any $x_1, \cdots, x_{\ksh} \in E \cap B(x_0,\delta)$, there exist $\vp_1, ... , \vp_{\ksh}, \vq_1, ... , \vq_\ksh$, with 
\begin{equation}\label{eq:tetris1}
    \vp_j, \vq_j \in H_{l-1}(x_j)\text{ for }j = 1, \cdots, \ksh;
\end{equation}
\begin{equation}
    \abs{\D^\alpha(\vp_i - \vp_j)(x_j)} \leq \epsilon\abs{x_i - x_j}^{m-\abs{\alpha}}\text{ for }\abs{\alpha}\leq m, 0 \leq i, j \leq \ksh;
\end{equation}
and
\begin{equation}
    \abs{\D^\alpha(\vq_i - \vq_j)(x_j)} \leq \epsilon\abs{x_i - x_j}^{m-\abs{\alpha}}\text {for }\abs{\alpha}\leq m, 0 \leq i, j \leq \ksh.
\end{equation}

We set 
\begin{equation}\label{eq:tetris2}
    \vr_j := (1-\theta)\vp_j + \theta \vq_j
    \text{ for } j = 1, \cdots, \ksh.
\end{equation}
By the induction hypothesis that $H_{l-1}(x)$ is convex for each $x \in E$, \eqref{eq:tetris1} and \eqref{eq:tetris2} imply $\vr_j \in H_{l-1}(x_j)$ for each $j = 1, \cdots, \ksh$. Moreover, by \eqref{eq:tetris2} and the triangle inequality,
\begin{align}
    \abs{\D^\alpha(\vr_i - \vr_j)(x_j)} 
        &\leq 
        (1-\theta)\abs{\D^\alpha (\vp_i - \vp_j)(x_j) }+\theta \abs{\D^\alpha (\vq_i - \vq_j)(x_j) }\\ &\leq \epsilon\abs{x_i-x_j}^{m-\abs{\alpha}}
\end{align}
for $\abs{\alpha} \leq m$ and $0 \leq i,j \leq \ksh$. 

Hence, $H_l(x_0)$ is convex. Since $x_0,\vp_0,$ and $\vq_0$ were arbitrary, the lemma is proved.
\end{proof}

In \cite{BMP03} and \cite{F06}, the analogous version of Lemma \ref{lemma:termination of Glaeser refinement} for affine holding spaces was proven using an argument which relied on the well-definition of the dimension of vector subspaces. Here, we adapt this argument to the case of non-affine holding spaces so another definition of dimension will be required. 

In this section, we will use the notation $B_{\eta}(\vec{P})$ to denote the open ball of radius $\eta$ in $\vec{\P}$ centered $\vec{P}$ with respect to the metric
\begin{equation}\label{eq:preliminary metric}
    d(\vec{P},\vec{P}'):=\max_{|\alpha|\le m}|\da(\vec{P}-\vec{P}')(0)|.
\end{equation}
$\overline{B}_{\eta}(\vec{P})$ will denote the analogous closed ball in $\vec{\P}$.

Let $(H(x))_{x\in E}$ be a holding space with Glaeser refinement $(\tilde{H}(x))_{x\in E}$. Observe that whenever $x_0,x_1,x_2,...\in E$, $\vec{P}_0\in \tilde{H}(x_0)$ and $x_j\to x_0$ as $n\to\infty$, there exist $\vec{P}_j\in H(x_j)$ such that $d(\vec{P}_j,\vec{P}_0)\to0$. This property, rather than the definition of Glaeser refinement, will be key in the proof of Lemma \ref{lemma:termination of Glaeser refinement}.

\begin{definition}\label{def:local dimension}
Given $K\subset \vec{\P}$ convex and $\vec{P}\in \vec{\P}$, we define the \underline{local dimension of $K$ at $\vec{P}$}, denoted $\dim_{\vec{P}}K$, to be the largest integer $k$ such that the following holds:
\begin{itemize}
    \item[\LA{eq:local dimension definition}] There exists $\eta>0$ such that $K\cap B_\eta(\vec{P})\supset W\cap B_\eta(\vec{P})$ for some $k$-dimensional affine subspace $W\ni \vec{P}$.
\end{itemize}
We take $\dim_{\vec{P}}K=-\infty$ in the case where $\vec{P}\notin K$. Thus, the condition $\dim_{\vec{P}}K\neq-\infty$ implies $\vec{P}\in K$.
\end{definition}

For an example illustrating the notion of local dimension, consider the half-ball $K=\{(x,y)\in\R^2:x^2+y^2\le1, y\ge0\}$. In this case $\dim_{(0,1/2)}K=2$ since $B_{1/10}((0,1/2))\subset K$ and $\dim_{(0,0)}K=1$ since any 2-dimensional ball containing $(0,0)$ must intersect the lower half-plane but the line segment from $(-1/2,0)$ to $(1/2,0)$ is contained in $K$. Furthermore, $\dim_{(1,0)}K=\dim_{(0,1)}K=0$ since in either case, any line segment with the given point as its midpoint must intersect $K^c$.

\begin{remark}\label{remark:on local dimension}
\begin{enumerate}
    \item\label{remark:on local dimension 1} If $\vec{P}$ is in the relative interior of $K$, written $\vec{P}\in \text{int} K$, then we may take $K\cap B_\eta(\vec{P})=W\cap B_\eta(\vec{P})$ in \eqref{eq:local dimension definition}.
    \item\label{remark:on local dimension 2} In the case where $K$ is a closed, convex subset of $\vec{\P}$, then it has a well-defined dimension as a set, often defined as the smallest dimension of an affine subspace containing the set $K$, which we denote $\dim K$. When $\vec{P}\in \text{int} K$, $\dim_{\vec{P}} K=\dim K$; in particular, for a vector subspace $V\subset\vec{\P}$, $\dim_{\vec{P}} V=\dim V$ for any $\vec{P}\in V$.
    \item\label{remark:on local dimension 3} Whenever $K\supset L$, we have $\dim_{\vec{P}}K\ge \dim_{\vec{P}}L$.
\end{enumerate}
\end{remark}

Lemma \ref{lemma:termination of Glaeser refinement} will follow shortly from the following result:

\begin{lemma}\label{lemma:dimension terminates}
Suppose $\left(H(x)\right)_{x\in E}$ is a convex bundle. Let $x\in E$, $\vec{P}\in\vec{\P}$, and $k\ge0$ be an integer. If
\begin{equation}\label{eq:inductive version}
    \dim_{\vec{P}}H_{2k+1}(x)\geq D-k,
\end{equation}
then $\dim_{\vec{P}}H_{l}(x)=\dim_{\vec{P}}H_{2k+1}(x)$ for all $l\ge 2k+1$.
\end{lemma}

To see the value of Lemma \ref{lemma:dimension terminates}, let us use it to prove Lemma \ref{lemma:termination of Glaeser refinement}:

\begin{proof}[Proof of Lemma \ref{lemma:termination of Glaeser refinement} via Lemma \ref{lemma:dimension terminates}]

\newcommand{\vp}{\vec{P}}

Let $x\in E$ and $\vp\in H_{2D+1}(x)$. Then, $\dim_{\vec{P}} H_{2D+1}(x)\ge0$ because the zero-dimensional affine space $W=\{\vec{P}\}$ satisfies $B_\eta(\vp)\cap K\supseteq B_\eta(\vp)\cap W$ for any $\eta>0$. 

Taking $k=D$ in \eqref{eq:inductive version}, $\dim_{\vp}H_l(x)=\dim_{\vec{P}}H_{2k+1}(x)\geq0$ for all $l\ge2k+1$. The local dimension is a nonnegative integer if and only if $\vp\in H_l(x)$; therefore, $\vp\in H_l(x)$ for all $l\ge 2D+1\ge 2k+1$.
\end{proof}

Now, it suffices to prove Lemma \ref{lemma:dimension terminates}.

\begin{proof}[Proof of Lemma \ref{lemma:dimension terminates}]

\newcommand{\vp}{\vec{P}}

We prove this result by induction on $k$. By Lemma \ref{lemma:convexity for all}, $H_l(y)$ is convex for all $y\in E$ and $l\ge0$. This convexity will be used throughout the proof.

{\bf Base case ($k=0$):} First, suppose $\dim_{\vp}H_1(x)=D$ for some $x\in E$ and $\vp\in H_1(x)$. In such a case $\dim_{\vp} H(x)=D$ as well by Remark \ref{remark:on local dimension} No. \ref{remark:on local dimension 3}. Pick $\eta$ so $H_{1}(x)\cap B_\eta(\vp)\supset W\cap B_\eta(\vp)$ for some $D$-dimensional affine subspace $W$ running through $\vp$. The only $D$-dimensional affine space in $\vec{\P}$ is $\vec{\P}$ itself, so $B_\eta(\vp)\subset H_1(x)$. 

\begin{claim}\label{claim:for base case}
$B_{\eta/2}(\vp)\subset H(y)$ for all $y$ sufficiently close to $x$.
\end{claim}

\begin{proof}[Proof of Claim \ref{claim:for base case}]

\renewcommand{\vp}{\vec{P}}
\newcommand{\vq}{\vec{Q}}

Suppose the contrary. Then, there is a sequence $x_n\to x$ and $\vp_n\in B_{\eta/2}(\vp)$ such that $\vp_n\notin H(x_n)$.

Consider the set $B_\eta(\vp)\cap H(x_n)$. Since $\vp_n\notin H(x_n)$ and $H(x_n)$ is convex, there exists $\vq_n\in B_\eta(\vp)$ such that $B_{\eta/4}(\vq_n)$ and $H(x_n)\cap B_\eta(\vp)$ are disjoint subsets of $B_\eta(\vp)$. By passing to a subsequence, we may assume there exists $\vq\in B_{\eta}(\vp)$ such that $d(\vq_n,\vq)\to 0$. This contradicts the fact that $\vq\in H_1(x)$, thus proving Claim \ref{claim:for base case}.
\end{proof}

Since $B_{\eta/2}(\vp)\subset H(y)$ for all $y$ sufficiently close to $x$, we see this property also holds for $H_1(y)$ in place of $H(y)$. In fact, let $r>0$; then for $y$ in an open subset of $E$, $H_{l+1}(y)\cap B_r(\vp)$ is determined solely by $H_{l}(z)\cap B_r(\vp)$ for $z$ in that open subset by the definition of Glaeser refinement.

Repeating this logic, we see $B_{\eta/2}(\vp)\subset H_l(y)$ for all $H_l(y)$ with $l\ge 1$ and $y\in E$ sufficiently close to $x$; this includes the case $y=x$. Thus, the local dimension of $H_l(x)$ at $\vp$ remains at $D$ for all $l$, establishing the base case.

{\bf Induction Hypothesis:} Assume
\begin{equation*}
    \dim_{\vp}H_{2j+1}(x)\geq D-j
\end{equation*}
implies $\dim_{\vp}H_{l}(x)=\dim_{\vp}H_{2j+1}(x)$ for all $l\ge 2j+1$ whenever $j\le k$.

Now suppose $\dim_{\vp}H_{2k+3}(x)\geq D-k-1$. If $\dim_{\vp}H_{2k+3}(x)> D-k-1,$ then the conclusion follows from the induction hypothesis. Thus, we may assume
\begin{equation}
    \dim_{\vp}H_{2k+3}(x)=D-k-1.
\end{equation}
Furthermore,
\begin{equation} \dim_{\vp}H_{2k+1}(x)=\dim_{\vp}H_{2k+2}(x)=\dim_{\vp}H_{2k+3}(x)=D-k-1,
\end{equation}
since if $\dim_{\vp}H_{2k+1}(x)>D-k-1$, $\dim_{\vp} H_l(x)$ would have terminated at some number greater than $D-k-1$ starting at some $l<2k+1$ by induction hypothesis. We want to show

\begin{equation}
    \dim_{\vp}H_{l}(x)=D-k-1
\end{equation}
whenever $l\ge 2k+3$.

\begin{claim}\label{claim:for IH}
Let $l=2k+1, 2k+2$. Fix $x\in E$ and $\vp\in H_{l+1}(x)$. Pick $\eta>0$ so that $H_{l+1}(x)\cap B_\eta(P)\supseteq W\cap B_\eta(P)$ for some $(D-k-1)$-dimensional affine subspace $W\ni P$.  For $y$ sufficiently close to $x$, either
\begin{equation}\label{eq:higher dimension}
    \dim [H_l(y)\cap B_\eta(\vp)]>D-k-1
\end{equation}
or
\begin{equation}\label{eq:equals subspace}
H_{l}(y)\cap B_{\eta/2}(\vp)=V_y\cap B_{\eta/2}(\vp)
\end{equation}
for some $(D-k-1)$-dimensional affine subspace $V_y$ of $\vec{\P}$. Furthermore, if $\vp$ is in the relative boundary of $H_{l+1}(x)$, then in particular, \eqref{eq:higher dimension} holds for $y$ sufficiently close to $x$.
\end{claim}

\begin{proof}[Proof of Claim \ref{claim:for IH}]

Since $l=2k+1,2k+2$, we have $\dim_{\vp}H_{l}(x)=\dim_{\vp}H_{l+1}(x)=D-k-1$.

If $\vp\in\text{int} H_{l+1}(x)$, then suppose the contrary. That is, there exists a sequence $x_n\to x$ such that \eqref{eq:higher dimension} and \eqref{eq:equals subspace} fail for $y=x_n$. In particular, since \eqref{eq:higher dimension} does not hold, $H_{l}(x_n)\cap B_\eta(P)$ is contained in a $(D-k-1)$-dimensional affine space $V_n$. Here, we use Remark \ref{remark:on local dimension} No. \ref{remark:on local dimension 1}.

By compactness of the Grassmannian of $(D-k-1)$-planes in $\vec{\P}$ with intersection in $\overline{B}_{\eta}(\vp)$ and the closure of $\overline{B}_\eta(\vp)$, we may assume $V_n$ converges to a subspace $V$ by passing to a subsequence. 

Since $H_{l+1}(x)\cap B_\eta(\vp)=W\cap B_\eta(\vp)$, $V$ must equal $W$. Else, there exists $\vp'\in H_{l+1}(x)$ which is not the limit of $\vp_n\in H_l(x_n)$, contradicting our prior observation on refinements.

\newcommand{\vq}{\vec{Q}}

Since \eqref{eq:equals subspace} fails, there exists $\vec{P}'_n\in V_n\cap B_{\eta/2}(\vec{P})$ such that $\vec{P}'_n\notin H_l(x_n)$. Furthermore, by the convexity of $H_l(x_n)$, there exists $\vq_n\in B_{\eta}(\vp)\cap V_n$ such that $\vq_n\notin H_l(x_n)$ and $B_{\eta/4}(\vq_n)\subset B_\eta(\vp)$ is disjoint from $H_{l}(x_n)$.

By passing to another subsequence, we may assume $\vq_n$ converges to some $\vq\in V=W$. However, this contradicts the fact that $\vq\in H_{l+1}(x)$, establishing our claim for $\vp\in\text{int} H_{l+1}(x)$.

If $\vp$ is in the relative boundary of $H_{l+1}(x)$, then $\dim H_{l+1}(x)>D-k-1$. Taking $\vq\in\text{int}H_{l+1}(x)$ such that $\dim_{\vq}H_{l+1}(x)>D-k-1$, we may repeat the above argument to deduce, at the very least, that for $y$ sufficiently close to $x$, $\dim [H_{l}(y)\cap B_\eta(\vp)]>D-k-1$. This completes the proof of Claim \ref{claim:for IH}.
\end{proof}

Consider $y$ sufficiently close to $x$. There are now distinct cases based on which of \eqref{eq:higher dimension} or \eqref{eq:equals subspace} holds for $l=2k+1$ and $l=2k+2$.

First, there is the case that \eqref{eq:higher dimension} holds for $l=2k+1$ or $l=2k+2$, which means $H_j(y)\cap B_\eta(\vec{P})$ and $H_{l}(y)\cap B_\eta(\vp)$ have equal relative interiors for all $j\ge l$ by induction hypothesis.


Else, \eqref{eq:equals subspace} holds for both $l=2k+1,2k+2$, meaning $H_{2k+1}(y)\cap B_{\eta/2}(\vp)=V_y\cap B_{\eta/2}(\vp)$ for some $(D-k-1)$-dimensional affine subspace $V_y$ of $\vec{\P}$ and $H_{2k+2}(y)\cap B_{\eta/2}(\vp)=V_y'\cap B_{\eta/2}(\vp)$ for some $(D-k-1)$-dimensional affine subspace $V_y'$ of $\vec{\P}$. We see that $V_y=V_y'$ since $H_{2k+2}(y)\subset H_{2k+1}(y)$ and $\dim V_y=\dim V_y'$.

In any case, we see that for $y$ sufficiently close to $x$, $H_{2k+1}(y)\cap B_\eta(\vp)$ and $H_{2k+2}(y)\cap B_\eta(\vp)$ have equal relative interiors. Ideally, one would like to show that these sets are completely equal; however, Lemma \ref{lemma:closure same} below will address this disparity.

\begin{lemma}\label{lemma:closure same}
Let $(H(x))_{x\in E}$ be a bundle and fix $y_0\in E$. Fix $\delta_0>0$. Let $\left(\hat{H}(x)\right)_{x\in E}$ denote the bundle made from taking $(H(x))_{x\in E}$ and replacing $H(y)$ by $\overline{H}(y)$ for all $y\in B(y_0,\delta_0)\setminus\{y_0\}$. Then $H_1(y_0)=\hat{H}_1(y_0)$.

In particular, if two bundles have the same relative interiors of their fibers at every $y\in B(y_0,\delta_0)\setminus\{y_0\}$ and $\vp$ is in both their fibers at $y_0$, then $\vp$ is either in the fiber at $y_0$ of both Glaeser refinements or neither.
\end{lemma}

We recall that given $r>0$, $y$ in a relatively open subset of $E$, $H_{l+1}(y)\cap B_r(\vp)$ is determined solely by $H_{l}(z)\cap B_r(\vp)$ for other $z$ in that relatively open subset. Applying Lemma \ref{lemma:closure same} with $y_0$ ranging over $y$ sufficiently close to $x$, we find
\begin{equation}
    \text{int}H_{l}(y)=\text{int}H_{2k+1}(y)\text{ for }l\ge 2k+1.
\end{equation}

Another application of Lemma \ref{lemma:closure same} shows that $H_l(x)\cap B_{\eta/2}(\vp)$ remains constant in $l$ for $l\ge 2k+1$ as well, completing the proof of Lemma \ref{lemma:dimension terminates}.
\end{proof}

At this point, it merely remains to prove Lemma \ref{lemma:closure same}.

\begin{proof}[Proof of Lemma \ref{lemma:closure same}]

\newcommand{\vp}{\vec{P}}

Since $H(x)\subseteq \hat{H}(x)$ for all $x\in E$, we have $H_1(y_0)\subseteq\hat{H}_1(y_0)$.

Suppose $\vp_0\in\hat{H}_1(y_0)$. Let $\epsilon>0$ and pick $0<\delta<\delta_0$ such that for every $y_1,...,y_{k^\sharp}\in B(y_0,\delta)$, there exists $\vp_j\in\hat{H}(y_j)$ for $j=0,,1,...,k^\sharp$ such that

\begin{equation}\label{eq:another Glaeser stability one}
    |\partial^\alpha(\vp_i-\vp_j)(y_j)|\le (\epsilon/2)|y_i-y_j|^{m-|\alpha|}\text{ for }|\alpha|\le m, 0\le i,j\le k^\sharp.
\end{equation}

Now fix a particular choice of $y_1,...,y_{k^\sharp}$ and make a choice of $\vp_j\in\hat{H}(y_j)$ satisfying \eqref{eq:another Glaeser stability one}. If all the $\vp_j$ lie in $H(y_j)$, then we are done. It is possible to choose $\vp_j\notin H(y_j)$; however, we may always choose $\vp_j'\in H(y_j)$ such that $|\partial^\alpha(\vp_j-\vp_j')(x)|\le \eta_j$ for some $\eta_j>0$ and all $x\in E, 0\le|\alpha|\le m$.

By the triangle inequality, for any $|\alpha|\le m, 0\le i,j\le \ksh$,
\begin{align}
    |\da(\vp_i'-\vp_j')(y_j)|&\le|\da(\vp_i'-\vp_i)(y_j)|+|\da(\vp_i-\vp_j)(y_j)|+|\da(\vp_j-\vp_j')(y_j)|\\
    &\le \eta_i+(\epsilon/2)|y_i-y_j|^{m-|\alpha|}+\eta_j\\
    &\le \eps|y_i-y_j|^{m-|\alpha|}
\end{align}
by taking $\eta_j$ ($0\le j\le \ksh$) sufficiently small relative to $\eps$ and $\min_{i\neq j}|y_i-y_j|$.

Since the choices of $\epsilon$ and $y_j$ were arbitrary, this shows $\vec{P}_0\in H_1(y_0)$.

The second conclusion follows from the fact that if two fibers have the same relative interior, then they have the same closure and by considering the holding space $(H'(x))_{x\in E}$, where $H'(x)=H(x)$ for $x\ne y_0$ and $H(y_0)=\{\vec{P}\}$.
\end{proof}

\section{Reduction to the scalar-valued case}\label{sec:gradient trick}

In this section, we show that the validity of Theorem \ref{thm:heart of the matter} for $C^{m+d}(\Rn,\R)$ implies the same for $C^m(\Rn,\R^d)$, thus reducing all future analysis to scalar-valued functions. The main lemma of the section is the following.

\begin{lemma}\label{lemma:dimension reduction}
Let $m,n,d$ be positive integers. Let $\ksh$ be a sufficiently large constant depending only on $m,n,d$.  If Theorem \ref{thm:heart of the matter} holds for $C^{m+1}(\R^{n+d},\R)$ and shape fields of type $(m+1,n+d,1)$, then it also holds for $C^m(\Rn,\R^d)$ and shape fields of type $(m,n,d)$. 
\end{lemma}

We use a gradient trick inspired by \cite{FL14}.

To begin with, we put more emphasis on the degree of regularity and dimension: For positive integers $s_1, s_2, s_3$, we write $\P^{s_1}(\R^{s_2}, \R^{s_3})$ to denote the vector space of $s_3$-tuples of polynomials with degree no greater than $s_1$ on $\R^{s_2}$. We write $\odot_x^{s_1,s_2}$ to denote the ring multiplication of $s_1$-jets at $x \in \R^{s_2}$, and denote this ring by $\ring_x^{s_1, s_2}$. As before, we use the same notation to denote the action of $\ring_x^{s_1, s_2}$ on the module $\P^{s_1}(\R^{s_2}, \R^{s_3})$.

For the rest of the section, we use $(x,v) = (x_1, \cdots, x_n, v_1, \cdots, v_d)$ to denote a vector in $\R^{n+d}$. We write $\nabla_v$ to denote the operator $(\d_{v_1}, \cdots, \d_{v_d})$. Note that if $P$ is an $(m+1)$-jet (of a real-valued function) on $\R^{n+d}$, then $P(\,\cdot\,,0)$ is an $(m+1)$-jet (of a real-valued function) on $\Rn$, and $\nabla_v\big|_{v=0} P$ is an $m$-jet (of an $\R^d$-valued function) on $\Rn$.

We say a shape field $(\Gamma(x,M))_{x\in E, M \geq 0}$ is \underline{of type $(s_1,s_2,s_3)$} if $\Gamma(x,M) \subset\P^{s_1}(\R^{s_2}, \R^{s_3}) $ for all $x \in E$ and $M \geq 0$.

\newcommand{\Ghat}{\widehat{\Gamma}}
\begin{lemma}\label{lem:gradient trick}
    Let $E \subset \Rn$ be a compact set. Let $(\Gamma(x,M))_{x\in E, M \geq 0}$ be a regular $(C_w,\dm)$-convex shape field of type $(m,n,d)$. For each $x \in E$ and $M \geq 0$, define
    \begin{equation}
        \Ghat((x,0),M) = 
        \set{
        \widehat{P}\in \P^{m+1}(\R^{n+d},\R) : \,
        \begin{matrix*}[l]
            \abs{\da\widehat{P}(x,0)} \leq M \text{ for }\alpha \in \mathbb{N}_0^{n+d}, \abs{\alpha}\leq m+1,\\
            \widehat{P}(\,\cdot\,,0)\equiv 0\text{, and }\nabla_v\big|_{v=0}\widehat{P}\in \Gamma(x,M)
        \end{matrix*}
        }.
        \label{eq.Ghat-def}
    \end{equation}

There exists $C = C(m,n,d,C_w)$ such that    
    \begin{enumerate}
    \item $(\Ghat((x,0),M))_{(x,0)\in E\times \set{0}, M \geq 0}$ is a regular $(C,\dm)$-convex shape field of type $(m+1,n+d,1)$ on $E\times\set{0} \subset \R^{n+d}$.
    \item $C^{-1}\norm{\Gamma}\leq \norm{\widehat{\Gamma}} \leq C\norm{\Gamma}$, with $\norm{\cdot}$ as in Definition \ref{def:norm of a bundle}.
    \end{enumerate}
    
\end{lemma}

As a corollary to part 2 of the above lemma, the Finiteness Lemma (Lemma \ref{lemma:finiteness}) for vector-valued functions immediately follows from the lemma for scalar-valued functions.

\begin{proof}
\newcommand{\Phat}{\widehat{P}}

We begin with the first statement. It is clear that $(\Ghat((x,0),M))_{(x,0)\in E\times \set{0}, M \geq 0}$ is a shape field of type $(m+1,n+d,1)$.

We now prove that $(\Ghat((x,0),M))_{(x,0)\in E\times \set{0}, M \geq 0}$ is regular.

Condition \eqref{regularity 1} is clearly satisfied, thanks to the first constraint in \eqref{eq.Ghat-def}. 

We prove condition \eqref{regularity 2}. Let $\eps > 0$ and $M \geq 0$. Let $\eta = \eta(\eps,M) > 0$ be a small number to be chosen. Let $(x,0), (x',0) \in E\times \set{0}$ with $\abs{(x,0)-(x',0)}\leq \eta$. Let $\Phat \in \Ghat((x,0),M)$ and $\Phat \in \Ghat(x',0) = \bigcup_{M'\geq 0}\Ghat((x',0),M')$, with 
\begin{equation}
    \abs{\da(\Phat - \Phat')(x,0)}\,,\,\abs{\da(\Phat - \Phat')(x',0)} \leq \delta\abs{x-x'}^{m+1-\abs{\alpha}}
    \text{ for }\alpha \in\mathbb{N}_0^{n+d}, \abs{\alpha}\leq m+1.
    \label{eq.trick-01}
\end{equation}

Set $\vec{P} = \nabla_v\big|_{v=0}\Phat$ and $\vec{P}' = \nabla_v\big|_{v=0}\Phat'$. By the definition of $\Ghat$, $\vec{P}\in \G(x,M)$ and $\vec{P}' \in \G(x')$. In view of \eqref{eq.trick-01}, we have
\begin{equation}
    \abs{\d^\beta(\vec{P}-\vec{P}')(x)}\,,\,\abs{\d^\beta(\vec{P}-\vec{P}')(x')} \leq C\eta\abs{x-x'}^{m-\abs{\alpha}}
    \text{ for }\alpha \in \mathbb{N}_0^n, \abs{\alpha}\leq m.
    \label{eq.trick-02}
\end{equation}
Since $(\G(x,M))_{x\in E, M \geq 0}$ is regular, for sufficiently small $\eta$, \eqref{eq.trick-02} implies
\begin{equation}
    \vec{P}' \in \G(x,M+\tilde{\eps}(\eta))\text{, i.e.,} \nabla_v\big|_{v=0}\Phat \in \G(x,M+\tilde{\eps}(\eta)).
    \label{eq.trick-03}
\end{equation}
Here $\tilde{\eps}(\cdot)$ is a decreasing function of $\eta$, depending only on $M$. 

We choose $\eta$ to be sufficiently small (depending only on $\eps$ and $M$), so that \eqref{eq.trick-01}, \eqref{eq.trick-03}, and a similar argument as in Lemma \ref{lemma:application} altogether imply $\Phat' \in \Ghat((x',0),M+\eps)$. Condition \eqref{regularity 2} is satisfied.

For condition \eqref{regularity 3}, let $\eps > 0$, $M \geq 0$, and let $\eta > 0$ be a sufficiently small number to be determined. Let $(x,0) \in E\times\set{0}$, $\Phat \in \Ghat((x,0),M)$, $\Phat'\in \Ghat(x,0)$ with 
\begin{equation}
    \abs{\da(\Phat - \Phat')(x,0)} \leq \eta
    \text{ for }\alpha \in \mathbb{N}_0^{n+d},\,\abs{\alpha}\leq m+1.
    \label{eq.trick-04}
\end{equation}
Set $\vec{P} = \nabla_v\big|_{v=0}\Phat$ and $\vec{P}' = \nabla_v\big|_{v=0}\Phat'$. By the definition of $\Ghat$, $\vec{P}\in \G(x,M)$ and $\vec{P}' \in \G(x')$. In view of \eqref{eq.trick-04}, we have
\begin{equation}
    \abs{\d^\beta(\vec{P}-\vec{P}')(x)} \leq C\eta
    \text{ for }\beta \in \mathbb{N}_0^n,\, \abs{\beta}\leq m.
    \label{eq.trick-05}
\end{equation}
Since $(\G(x,M))_{x\in E, M \geq 0}$ is regular, for sufficiently small $\eta$ (depending only on $\eps$ and $M$), \eqref{eq.trick-04}, \eqref{eq.trick-05}, and a similar argument as in Lemma \ref{lemma:application} imply $\Phat' \in \Ghat((x,0), M + \eps)$. Condition \eqref{regularity 3} is satisfied.

Thus, we have shown that $(\Ghat((x,0),M))_{(x,0)\in E\times\set{0}, M \geq 0}$ is regular. 

We now show that $(\Ghat((x,0),M))_{(x,0)\in E\times\set{0}, M \geq 0}$ is $(C,\dm)$-convex for some $C$ depending only on $m,n,d,C_w$.

\newcommand{\Qhat}{\widehat{Q}}
Let $\delta \in (0,\dm] $, $(x,0)\in E\times\set{0}$, $M \geq 0$, $\Phat_1, \Phat_2, \Qhat_1, \Qhat_2 \in \P^{m+1}(\R^{n+d},\R) $ satisfy
\begin{align}
    &\Phat_1, \Phat_2 \in \Ghat((x,0),M)\label{eq.trick-cd-1}\\
    &\abs{\da(\Phat_1-\Phat_2)(x,0)} \leq M\delta^{m+1-\abs{\alpha}}
    \text{ for }\alpha \in \mathbb{N}_0^{n+d},\,\abs{\alpha}\leq m+1;\label{eq.trick-cd-2}\\
    &\abs{\da \Qhat_i(x,0)}\leq \delta^{-\abs{\alpha}}
    \text{ for }\alpha \in \mathbb{N}_0^{n+d},\,\abs{\alpha}\leq m+1,\, i = 1,2; \text{ and}\label{eq.trick-cd-3}\\
    &\Qhat_1\odot_{(x,0)}^{m+1,n+d}\Qhat_1 + \Qhat_2\odot_{(x,0)}^{m+1,n+d}\Qhat_2 = 1.\label{eq.trick-cd-4}
\end{align}
We want to show that
\begin{equation}
    \Phat = \sum_{i = 1,2}\Qhat_i\odot_{(x,0)}^{m+1,n+d}\Qhat_i\odot_{(x,0)}^{m+1,n+d}\Phat_i \in \Ghat((x,0),CM),\label{eq.trick-cd-5}
\end{equation}
or equivalently, 
\begin{align}
    &\abs{\da\Phat(x,0)} \leq M \text{ for }\alpha \in \mathbb{N}_0^{n+d},\, \abs{\alpha}\leq m+1,\label{eq.trick-cd-6}\\
    &\Phat(\,\cdot\,, 0) \equiv 0\text{, and }\label{eq.trick-cd-7}\\
    & \nabla_v\big|_{v = 0}\Phat \in \Gamma(x,CM).\label{eq.trick-cd-8}
\end{align}

Thanks to \eqref{eq.trick-cd-1}, we have
\begin{align}
    &\abs{\da\Phat_i(x,0)} \leq M \text{ for }\alpha\in\mathbb{N}_0^{n+d},\,\abs{\alpha}\leq m+1,\label{eq.trick-cd-9}\\    
    &\Phat_i(\,\cdot\,, 0)\equiv 0\text{ for }i = 1,2, \text{ and }\label{eq.trick-cd-10}\\
    &\nabla_v\big|_{v=0}\Phat_i \in \G(x,M) \text{ for } i = 1,2. \label{eq.trick-cd-11}.
\end{align}

In view of \eqref{eq.trick-cd-4} and the definition of $\Phat$ in \eqref{eq.trick-cd-5}, we have
\begin{equation}
    \da \Phat(x,0) = \sum_{\beta+\gamma = \alpha} C_{\alpha,\beta,\gamma} \cdot \d^\beta \Qhat_1(x,0) \cdot \d^{\gamma}\Qhat_1(x,0)\cdot  \d^{\alpha-\beta-\gamma}(\Phat_1-\Phat_2)(x,0) 
    .\label{eq.trick-cd-12}
\end{equation}
Using \eqref{eq.trick-cd-2} and \eqref{eq.trick-cd-3} to estimate \eqref{eq.trick-cd-11}, we see that \eqref{eq.trick-cd-6} follows.

Thanks to \eqref{eq.trick-cd-9}, we have
\begin{equation*}
    \Phat(\,\cdot\,,0) = \sum_{i = 1,2}\Qhat_i\big|_{v=0}\odot_{(x,0)}^{m+1,n+d}\Qhat_i\big|_{v=0}\odot_{(x,0)}^{m+1,n+d}\Phat_i\big|_{v = 0} \equiv 0.
\end{equation*}
Thus, \eqref{eq.trick-cd-7} follows.

Let $\pi: \ring_{(x,0)}^{m+1,n+d}\to \ring_{(x,0)}^{m,n+d}$ be the natural projection. Thanks to \eqref{eq.trick-cd-10}, we have
\begin{equation}
    \begin{split}
    \nabla_v\big|_{v=0}\Phat &= \sum_{i = 1,2}
    \pi \Qhat_i\big|_{v=0}\odot_{(x,0)}^{m,n+d}\pi \Phat_i\big|_{v=0}\odot_{(x,0)}^{m,n+d}\nabla_v\big|_{v=0}\Phat_i
    \\
    &\quad\quad\quad+
    2\pi \Qhat_i\big|_{v=0}\odot_{(x,0)}^{m,n+d} \pi\Phat_i\big|_{v=0}\odot_{(x,0)}^{m,n+d}\nabla_v\big|_{v=0}\Qhat_i\\
    &= \sum_{i = 1,2}\pi \Qhat_i\big|_{v=0}\odot_{(x,0)}^{m,n+d}\pi \Qhat_i\big|_{v=0}\odot_{(x,0)}^{m,n+d}\nabla_v\big|_{v=0}\Phat.
    \label{eq.trick-cd-13}
\end{split}
\end{equation}

We write $\da = \d_x^\beta \d_v^\xi$, for $\alpha \in \mathbb{N}_0^{n+d}$, $\beta \in \mathbb{N}_0^n$, $\xi \in \mathbb{N}_{0}^d$. 

If $\abs{\xi} = 1$, i.e., $\d_v = \d_{v_s}$ for some $s \in \set{1, \cdots, d}$, then \eqref{eq.trick-cd-2} implies
\begin{equation*}
    \abs{\da(\Phat_1 - \Phat_2)(x,0)} = \abs{\d_x^\beta(\d_{v_s}\big|_{v=0}\Phat_1 - \d_{v_s}\big|_{v=0}\Phat_2)(x)} \leq M\delta^{m-\abs{\beta}}
    \text{ for }\abs{\beta}\leq m.
\end{equation*}
Therefore,
\begin{equation}
\abs{
\d_x^\beta \left[
\nabla_v\big|_{v=0}(\Phat_1-\Phat_2)
\right] (x)
} \leq CM\delta^{m-\abs{\beta}}
\text{ for }\abs{\beta}\leq m.
\label{eq.trick-cd-14}
\end{equation}

On the other hand, we see from \eqref{eq.trick-cd-3} that
\begin{equation}
    \abs{\d_x^\beta \left[
    \pi \Qhat_i \big|_{v=0}
    \right]  }
    = \abs{
    \abs{\d_x^\beta\Qhat(x,0)} \leq \delta^{-\abs{\beta}}
    \text{ for }\abs{\beta}\leq m.
    }
    \label{eq.trick-cd-15}
\end{equation}

In view of \eqref{eq.trick-cd-13}--\eqref{eq.trick-cd-15} and the $(C _w,\dm)$-convexity of $(\Gamma(x,M))_{x\in E, M \geq 0}$, we can conclude that $\Phat \in \Ghat((x,0),CM)$, so \eqref{eq.trick-cd-5} holds. We have shown that $(\Ghat((x,0),M))_{(x,0)\in E\times \set{0}, M \geq 0}$ is $(C,\dm)$-convex.

This concludes the proof of Lemma \ref{lem:gradient trick}(1).

We turn to the second statement. 

Let $M = 2\norm{\Ghat}$. We will show that $\norm{\G}\leq CM$. 

Let $x_1, \cdots, x_\ksh \in E$ be given. There exist $\Phat_1 \in \Ghat((x_1,0),M), \cdots, \Phat_\ksh \in \Ghat((x_\ksh,0),M)$ such that
\begin{equation}
    \abs{\da(\Phat_i - \Phat_j)(x_i,0)} \leq M\abs{(x_i,0)-(x_j,0)}^{m+1-\abs{\alpha}}
    \text{ for $\alpha\in \mathbb{N}_0^{n+d}$, $\abs{\alpha}\leq m+1$}.
    \label{eq.equi}
\end{equation}
We set
\begin{equation*}
    \vec{P}_i:= \nabla_v\big|_{v=0}\Phat_i \in \P^m(\Rn,\R^d)
    \text{ for } 1 \leq i \leq \ksh.
\end{equation*}
It follows immediately from \eqref{eq.equi} that
\begin{equation*}
    \abs{\d^\beta(\vec{P}_i-\vec{P}_j)(x_i)}\leq CM\abs{x_i-x_j}^{m-\abs{\beta}}
    \text{ for $\beta \in \mathbb{N}_0^{n}$, $\abs{\beta}\leq m$}.
\end{equation*}
Therefore, $\norm{\Gamma}\leq \norm{\Ghat}$.

Now we show the reverse direction. Let $M = 2\norm{\G}$. We will show that $\norm{\Ghat} \leq M$.

Let $x_1, \cdots, x_\ksh \in E$ be given. There exist $\vec{P}_1 \in \G(x_1,M), \cdots, \vec{P}_\ksh \in \G(x_\ksh, M)$ such that
\begin{equation}
    \abs{\d^\beta(\vec{P}_i - \vec{P}_j)(x_i)} \leq M\abs{x_i-x_j}^{m-\abs{\alpha}}\text{ for $\beta \in \mathbb{N}_0^n$, $\abs{\beta}\leq m$}.
    \label{eq.pres-2-1}
\end{equation}
We define
\begin{equation*}
    \widehat{P}_i(\,\cdot\,,v)= v\cdot \vec{P}_i(\,\cdot\,) \in \P^{m+1}(\R^{n+d},\R)
    \text{ for } 1 \leq i \leq \ksh.
    \label{eq.pres-2-2}
\end{equation*}

Observe that for $1 \leq i \leq \ksh$, $\xi \in \mathbb{N}_0^d$,
\begin{equation}
    \d_v^\xi\big|_{v=0}\Phat_i = \begin{cases}
    0 &\text{ if } \abs{\xi} = 0\\
    P_{i,j} \text{ for some $1 \leq j \leq d$} &\text{ if }\abs{\xi}  =1\\
    0 &\text{ if }\abs{\xi} \geq 2
    \end{cases}\,.\label{eq.pres-2-3}
\end{equation}
Therefore, for $\abs{\xi} \neq 1$, we have
\begin{equation}
    \abs{\d_x^\beta \d_v^\xi(\Phat_i - \Phat_j)(x_i,0)} = 0
    \text{ for } \beta \in \mathbb{N}_0^n,\, \xi \in \mathbb{N}_0^d,\, \abs{\beta}+\abs{\xi}\leq m+1.
    \label{eq.pres-2-4}
\end{equation}
Suppose $\abs{\xi} = 1$. Without loss of generality, we may assume $\d_v^\xi = \d_{v_s}$ for some $s \in \set{1,\cdots, d}$. Thanks to \eqref{eq.pres-2-1}, we have
\begin{equation}
    \begin{split}
         \abs{\d_x^\beta\d_{v_s}(\Phat_i - \Phat_j)(x_i,0)}
        = \abs{\d_x^\beta(P_{i,s}-P_{j,s})(x_i)}
        \leq M\abs{x_i-x_j}^{m-\abs{\alpha}}= M\abs{x_i-x_j}^{m+1-\abs{\gamma}}.
    \end{split}
    \label{eq.pres-2-5}
\end{equation}

We see from \eqref{eq.pres-2-4} and \eqref{eq.pres-2-5} that $\norm{\Ghat} \leq M$. This proves the second statement.

\end{proof}

With Lemma \ref{lem:gradient trick} in hand, we are ready to prove Lemma \ref{lemma:dimension reduction}.

\begin{proof}[Proof of Lemma \ref{lemma:dimension reduction}]

\newcommand{\Ehat}{\widehat{E}}
\newcommand{\Rnd}{\R^{n+d}}
\newcommand{\Qhat}{\widehat{Q}_0}

Let $Q_0 \subset \Rn$ be a cube of length $3$ and $E \subset Q_0$ be compact. Let $\Ehat = E\times\set{0}\subset \Rnd$ and $\Qhat = Q_0 \times [0,0+3\delta_{Q_0})^d$. Note that $\Qhat$ is a hypercube of sidelength less than 3. 

Let $(\Gamma(x,M))_{x\in E, M \geq 0}$ be a regular $(C_w,1)$-convex shape field. Write $\Gamma(x) = \bigcup_{M\geq 0}\Gamma(x,M)$. Suppose $\Gamma^*(x) \neq \void$ for all $x \in E$, where $(\Gamma^*(x))_{x\in E}$ is the termination of iterated Glaeser refinements of $(\Gamma(x))_{x\in E}$ guaranteed by Lemma \ref{lemma:dimension terminates}. We write $\Gamma^*(x,M) = \Gamma^*(x)\cap \Gamma(x,M)$ for each $x \in E$ and $M \geq 0$. 

As in Lemma \ref{lem:gradient trick}, we define the following objects.
\begin{equation}
    \Ghat((x,0),M) = \set{
        \widehat{P}\in \P^{m+1}(\R^{n+d},\R) : 
         \,
        \begin{matrix*}[l]
            \abs{\da\widehat{P}(x,0)} \leq M \text{ for }\alpha \in \mathbb{N}_0^{n+d},\, \abs{\alpha}\leq m+1,\\
            \widehat{P}(\,\cdot\,,0)\equiv 0\text{, and }\nabla_v\big|_{v=0}\widehat{P}\in \Gamma(x,M)
        \end{matrix*}
        }.
        \label{eq.closed-0}
\end{equation}
\begin{equation}
    \Ghat^*((x,0),M) = \set{
        \widehat{P}\in \P^{m+1}(\R^{n+d},\R) : 
         \,
        \begin{matrix*}[l]
            \abs{\da\widehat{P}(x,0)} \leq M \text{ for }\alpha \in \mathbb{N}_0^{n+d},\, \abs{\alpha}\leq m+1,\\
            \widehat{P}(\,\cdot\,,0)\equiv 0\text{, and }\nabla_v\big|_{v=0}\widehat{P}\in \Gamma^*(x,M)
        \end{matrix*}
        }.
        \label{eq.red-00}
\end{equation}

Recall that the first part of Theorem \ref{thm:heart of the matter} was established in Lemma \ref{lemma:termination of Glaeser refinement}; thus it suffices to prove the second part. By Taylor's theorem and the definition of Glaeser refinement, we see that if $(\Gamma(x))_{x\in E}$ has a section $F$ satisfying $\|F\|_{C^m(\R^n,\R^d)}\le M$, then $(\Gamma^*(x))_{x\in E}$ is nonempty and $\|\Gamma^*\|\le CM$. Thus, it suffices to suppose $(\Gamma^*(x))_{x\in E}$ is nonempty and determine the existence of a section with the appropriate norm bounds.

Thanks to Lemma \ref{lem:gradient trick} and the assumption that $(\G(x,M))_{x\in E, M \geq 0}$ is a closed regular $(C,1)$-shape field, we see that
\begin{equation}
    \text{$(\Ghat((x,0),M))_{(x,0)\in \Ehat, M \geq 0}$ is a closed regular $(C,1)$-convex shape field.}
    \label{eq.closed-1}
\end{equation}

We write
\begin{equation}
    \Ghat((x,0)) = \bigcup_{M\geq 0}\Ghat((x,0),M)
    \text{ and }
    \Ghat^*((x,0)) = \bigcup_{M\geq 0}\Ghat^*((x,0),M)
    \text{ for }(x,0)\in \Ehat.
    \label{eq.red-000}
\end{equation}

We will show that the hypothesis of Theorem \ref{thm:heart of the matter} is satisfied for $\Ehat$, $m+1$, $n+d$, $\Ghat$, and $\Ghat^*$ in place of $E$, $m$, $n$, $\Gamma$, and $\Gamma^*$. 
In particular, we need to show that
\begin{enumerate}
    \item[\LA{hyp1}] $\Ghat^*((x,0)) \neq \void$ for each $(x,0)\in \Ehat$
    \item[\LA{hyp2}] $\Ghat^*((x,0))$ is its own Glaeser refinement, and
    \item[\LA{hyp3}] $\Ghat^*((x,0))\subset\Ghat((x,0))$ for $x\in E$.
\end{enumerate}

Assuming the above, we see that the shape field $(\Ghat((x,0),M))_{(x,0)\in \Ehat,M\geq 0}$ satisfies the hypotheses of Theorem \ref{thm:heart of the matter} for $C^{m+1}(\R^{n+d},\R)$. While we did not show the iterated Glaeser refinement of $\Ghat$ terminates in $\Ghat^*$, it is clear from definition that is must terminate in a Glaeser stable bundle containing $\Ghat^*$; call this bundle $\Gamma'$. It follows that $\Gamma'((x,0))$ is nonempty for all $(x,0)\in \hat{E}$ from the fact that $\Ghat((x,0))\subset\Gamma'((x,0))$. The appropriate quantitative bounds follow from $\|\Gamma'\|\le\|\Ghat'\|$.

Applying Theorem \ref{thm:heart of the matter}, there exists $G\in C^{m+1}(\R^{n+d},\R)$ with
\begin{equation*}
    \norm{G}_{\dot{C}^{m+1}(\R^{n+d},\R)} \leq \widehat{M}
    \text{ and }
    J_{(x,0)}^{+} \in \Ghat^*((x,0),\widehat{M}) \text{ for all }(x,0) \in \Ehat.
\end{equation*}
Here, $J_{(x,0)}^+$ denotes the $(m+1)$-jet at $(x,0)$, and $\widehat{M} = C\norm{\Gamma'}\le C\norm{\Ghat}$ for some controlled constant $C$.
Thanks to the regularity condition \eqref{regularity 1}, we can further improve the control of the norm
\begin{equation}
    \norm{G}_{C^{m+1}(\R^{n+d},\R)} \leq C\widehat{M}.
    \label{eq.red-11}
\end{equation}

Define
\begin{equation*}
    \vec{F}(x) = \left(
    \d_{v_1}G(x,0), \cdots, \d_{v_d}G(x,0)
    \right).
\end{equation*}
In view of the definition of $\Ghat$ in \eqref{eq.red-00} and \eqref{eq.red-000} and \eqref{eq.red-11}, we have 
\begin{equation}
    \norm{\vec{F}}_{\dot{C}^m(\Rn,\R^d)} \leq C\widehat{M}
    \text{ and }
    J_x\vec{F} \in \G^*(x,C\widehat{M}).
    \label{eq.red-12}
\end{equation}


In view of \eqref{eq.red-12}, we see that Lemma \ref{lemma:dimension reduction} holds.

It now suffices to establish \eqref{hyp1}, \eqref{hyp2}, and \eqref{hyp3}.

We show \eqref{hyp1} as follows. Since we assume that $\Gamma^*(x) \neq \void$ for each $x \in E$, we can pick $\vec{P}^x \in \Gamma^*(x)$ for each $x \in E$. We define
\begin{equation*}
    \widehat{P}^{(x,0)}(y,v)= v\cdot \vec{P}^x(y) \in \P^m(\R^{n+d},\R)
    \text{ for each } (x,0)\in \Ehat.
\end{equation*}
We immediately verify that $\widehat{P}^{(x,0)} \in \Ghat^*((x,0)) $ for each $(x,0) \in \Ehat$. Therefore, $\Ghat^*((x,0))\neq \void$ for each $(x,0) \in \Ehat$, establishing \eqref{hyp1}.

Now, to show \eqref{hyp2}, fix $(x_0,0)\in \Ehat$ and $\widehat{P}_0 \in \Ghat^*((x_0,0))$. We will show that $\widehat{P}_0$ survives the Glaeser refinement procedure. Let $\eps > 0$.

\newcommand{\Phat}{\widehat{P}}

Since $\Phat_0 \in \Ghat^*((x_0,0))$, we can write
\begin{equation}
    \Phat_0(y,v) =  \sum_{1\leq \abs{\gamma}\leq m+1}\frac{1}{\gamma!}v^\gamma P_\gamma(y)
    \text{ for some }
    P_\gamma \in \P^m(\Rn,\R).
    \label{eq.red-1}
\end{equation}
We set
\begin{equation}
    \begin{split}
        \vec{P}_0 &= (P_{0,1}, \cdots, P_{0,d}) = \nabla_v\big|_{v=0}\Phat_0 \\
        &= \left(
        \d_{v_1}\bigg|_{v=0}\left[
        \sum_{1\leq\abs{\gamma}\leq m+1}\frac{1}{\gamma!}v^\gamma P_\gamma(x)
        \right]
        ,\cdots,
        \d_{v_d}\bigg|_{v=0}\left[
        \sum_{1\leq\abs{\gamma}\leq m+1}\frac{1}{\gamma!}v^\gamma P_\gamma(x)
        \right]
        \right).
    \end{split}
    \label{eq.red-2}
\end{equation}

By construction, we have
\begin{equation}
    \vec{P}_0 \in \G^*(x_0).
\end{equation}

Thanks to \eqref{eq.red-1} and \eqref{eq.red-2}, we can write
\begin{equation}
    \Phat_0(y,v) = \sum_{1 \leq j \leq d}v_j P_{0,j}(y) + \sum_{2\leq\abs{\gamma}\leq m+1}\frac{1}{\gamma!}v^\gamma P_\gamma(y).
    \label{eq.red-2-0}
\end{equation}

Since the bundle $(\Gamma^*(x))_{x\in E}$ is Glaeser stable, we know that there exists $\delta > 0$ such that for all $x_1, \cdots, x_{\ksh} \in E \cap B^n(x_0,\delta)$, there exist 
\begin{equation}
    \begin{split}
        \vec{P}_1 &= (P_{1,1}, \cdots, P_{1,d} ) \in \G^*(x_1),
        \\
        &\vdots
        \\
        \vec{P}_{\ksh} &= (P_{\ksh,1},\cdots, P_{\ksh,d})\in \G^*(x_{\ksh}),
    \end{split}
    \label{eq.red-3}
\end{equation}
with
\begin{equation}
    \abs{\da(\vec{P}_i - \vec{P}_j)(x_i)} \leq 
    \eps\abs{x_i - x_j}^{m-\abs{\alpha}}
    \text{ for }\abs{\alpha} \leq m, 0 \leq i,j \leq \ksh.
    \label{eq.red-4}
\end{equation}

For any $(x_1,0), \cdots, (x_\ksh,0) \in \Ehat \cap B^{n+d}((x_0,0),\delta) $, we set
\begin{equation}
    \Phat_k(y,v) = \sum_{j = 1}^d v_j P_{k,j}(y) + \sum_{2\leq \abs{\gamma}\leq m+1}\frac{1}{\gamma!}v^\gamma P_\gamma (y)
    \text{ for } 1 \leq k \leq \ksh.
    \label{eq.red-5}
\end{equation}
Here, $P_\gamma$ are as in \eqref{eq.red-2-0}, and $\vec{P}_1, \cdots, \vec{P}_\ksh$ are as in \eqref{eq.red-3}.

We claim that
\begin{equation}
    \Phat_k \in \Ghat((x_k,0)) \text{ for } 1 \leq k \leq \ksh
    \label{eq.red-6}
\end{equation}
and
\begin{equation}
    \abs{\da(\Phat_i - \Phat_j)(x_i,0)} \leq C\eps\abs{x_i-x_j}^{m+1-\abs{\alpha}}
    \text{ for }\abs{\alpha} \leq m+1, 0 \leq i,j \leq \ksh.
    \label{eq.red-7}
\end{equation}

To verify \eqref{eq.red-6}, we apply $\nabla_v\big|_{v = 0}$ to \eqref{eq.red-5} and see that
\begin{equation}
    \nabla_v\big|_{v=0}\Phat_k = (P_{k,1}, \cdots, P_{k,d}) = \vec{P}_k \in \G^*(x_k)
    \text{ for } 1 \leq k \leq \ksh.
    \label{eq.red-8}
\end{equation}
Therefor, \eqref{eq.red-6} follows from \eqref{eq.red-00} and \eqref{eq.red-8}.

To verify \eqref{eq.red-7}, we write $\d^\alpha = \d_x^\beta \d_v^\xi$. Observe that for $0 \leq k \leq \ksh$, 
\begin{equation}
    \d_v^\xi\big|_{v=0}\Phat_k = \begin{cases}
    0 &\text{ if } \abs{\xi} = 0\\
    P_{k,j} \text{ for some $1 \leq j \leq d$} &\text{ if }\abs{\xi}  =1\\
    P_\xi &\text{ if }\abs{\xi} \geq 2
    \end{cases}\,.
    \label{eq.red-9}
\end{equation}
Thanks to \eqref{eq.red-9}, we see that \eqref{eq.red-7} holds trivially for $\abs{\xi} \neq 1$. Therefore, it suffices to show \eqref{eq.red-7} for the case $\abs{\xi} = 1$. Without loss of generality, we may assume $\d_v^\xi = \d_{v_s}$ for some $s \in \set{1,\cdots, d}$. Then
\begin{equation}
    \begin{split}
        \abs{\da(\Phat_i - \Phat_j)(x_i,0)} &= \abs{\d_x^\beta\d_{v_s}(\Phat_i - \Phat_j)(x_i,0)}
        \\
        &= \abs{\d_x^\beta(P_{i,s}-P_{j,s})(x_i)}
        \\
        &\leq \eps\abs{x_i-x_j}^{m-\abs{\beta}} \quad \text{(thanks to \eqref{eq.red-4})}
        \\
        &= \eps\abs{x_i-x_j}^{m+1-\abs{\alpha}}.
    \end{split}
    \label{eq.red-10}
\end{equation}
We see that \eqref{eq.red-7} follows from \eqref{eq.red-10}. 
We have shown that the bundle $(\Ghat^*((x,0)))_{(x,0)\in \Ehat}$ as defined in \eqref{eq.red-00} and \eqref{eq.red-000} is Glaeser stable, giving us \eqref{hyp2}.

Lastly, we get \eqref{hyp3} by observing that $\Ghat^*((x,0))\subset \Ghat((x,0))$ since $\Gamma^*((x,0))\subset\Gamma((x,0))$ for all $(x,0)\in \hat{E}$.

\end{proof}

\begin{remark}
Thanks to Lemma \ref{lemma:dimension reduction}, to prove Theorem \ref{thm:heart of the matter} for $C^m(\Rn,\R^d)$, it suffices to prove it for the case $d = 1$. 
\end{remark}

For the rest of the paper, we write $C^m(\Rn)$ instead of $C^m(\Rn,\R)$, and $F$, $\P$, $P$ instead of $\vec{F}$, $\vec{\P}$, $\vec{P}$.

\section{Finiteness principle for shape fields}\label{sec:fin princ}

Recall from the previous section that we will be working with real-valued functions for the rest of the paper.

    \begin{definition}\label{def.modulus}
    A function $\omega: [0,1] \to [0,\infty)$ is called a \underline{regular modulus of continuity} if it satisfies the following conditions:
    \begin{enumerate}[label=($\omega$-1)]
        \item $\omega(0) = \lim\limits_{t \downarrow 0}\omega(t) = 0$ and $\omega(1) = 1$;
        \item $\omega(t)$ is increasing on $[0,1]$; and
        \item $\omega(t)/t$ is decreasing on $(0,1]$.
    \end{enumerate}
    \end{definition}
    
    Note that in ($\omega$-2) and ($\omega$-3), we do not demand that $\omega$ be strictly increasing, or that $\omega(t)/t$ be strictly decreasing. 
    
    Let $\omega$ be a regular modulus of continuity. We write $\cmw(\Rn)$ to denote the space of all $C^m$ functions $F$ on $\Rn$ for which the norm
    \begin{equation*}
        \norm{F}_{\cmw(\Rn)}:= \sup_{\substack{x \in \Rn\\\abs{\alpha} \leq m}}\abs{\da F(x)} + \sup_{\substack{\abs{\alpha} = m\\x,x' \in \Rn \\0<\abs{x-x'} \leq 1}}\frac{\abs{\da F(x) - \da F(x')}}{\omega(\abs{x-x'})}
    \end{equation*}
    is finite. 
    
    Similarly, we write $\dot{C}^{m,\omega}(\Rn)$ to denote the space of all $C^m$ functions $F$ on $\Rn$ for which the norm
    \begin{equation*}
        \norm{F}_{\cmw(\Rn)}:=\sup_{\substack{\abs{\alpha} = m\\x,x' \in \Rn \\0<\abs{x-x'} \leq 1}}\frac{\abs{\da F(x) - \da F(x')}}{\omega(\abs{x-x'})}
    \end{equation*}
    is finite. 
    

    
    

    
    
	
	The following lemma was proven in \cite{FIL16}.
	
	\begin{lemma}\label{lemma:shape field convexity for more polynomials}
	Let $(\Gamma(x,M))_{x \in E, M\geq 0}$ be a $(C_w,\delta_{\max})$-convex shape field. Let $\delta \in (0,\delta_{\max})$, $x \in E$, $M\geq 0$, $A', A'' > 0$, $P_1, \cdots, P_k, Q_1, \cdots, Q_k \in \p$. Assume that
	\begin{itemize}
	    \item $P_i \in \Gamma(x,A' M)$ for $i = 1, \cdots, k$;
	    \item $\abs{\D^\alpha(P_i - P_j)(x)} \leq A' M\delta^{m-\abs{\alpha}}$ for $\abs{\alpha} \leq m$, $1 \leq i,j \leq k$;
	    \item $\abs{\D^\alpha Q_i(x)} \leq A''\delta^{-\abs{\alpha}}$ for $\abs{\beta} \leq m$ and $1 \leq i \leq k$;
	    \item $\sum_{i = 1}^k Q_i\odot_x Q_i = 1$.
	\end{itemize}
	Then $\sum_{i = 1}^k Q_i\odot_x Q_i\odot_x P_i \in \Gamma(x,CM)$, with $C = C(A',A'',C_w,m,n,k)$. 
	\end{lemma}
	
	
	
    By adapting the proof of the Finiteness Principle for shape fields from \cite{FIL16} (see also \cite{F05-J,F05-Sh}), we obtain the following result.

	\begin{theorem}[$ \dot{C}^{m,\omega} $-Finiteness Principle for Shape Fields]\label{thm:shape fields fin prin} Let $m,n\in\N$. For sufficiently large $ k^\sharp_{\mathrm{SF}} = k(m,n)$, the following holds. 
	
	Let $E \subset \Rn$ be an arbitrary subset and $\omega$ be a regular modulus of continuity. Let $ \Gamma(x,M)_{x \in E, M\geq 0}$ be a $(C_w,\delta_{\max})$-convex shape field. Let $M_0<\infty$ and $Q_0$ be a cube of sidelength $\delta_{Q_0} \leq \delta_{\max}$. Assume that for each $S \subset E$ with $\#(S) \leq \ksh_{SF}$, there exists $(P^x)_{x \in S}$ such that
	
\begin{equation}
    P^x \in \Gamma(x,M_0)\text{ for }x \in S
\end{equation}
and 

\begin{equation}
    |\da(P^x-P^y)(x)|\le M_0\omega(|x-y|)|x-y|^{m-|\alpha|}\text{ for }|\alpha|\le m, x,y\in S, x\ne y.
\end{equation}

	Then there exists 
	$F \in \cmw(Q_0)$, 
	with
\begin{equation}
    J_x F \in \Gamma(x,C^\sharp M_0)\text{ for } x \in E \cap Q_0
\end{equation}
and
\begin{equation}
    \norm{F}_{\dot{C}^{m,\omega}(Q_0)} \leq C^\sharp M_0.
\end{equation}

Here, $C^\sharp = C^\sharp(m,n,C_w)$.

\end{theorem}

\newcommand{\vp}{\vec{P}}
\newcommand{\brac}[1]{\left(#1\right)}

\section{Other Preliminaries}\label{sec:prelim}
	
	
    \begin{lemma}[Finiteness Lemma]\label{lemma:finiteness}
    Let $E\subset\Rn$ be a compact set. Let $(\Gamma(x,M))_{x\in E, M\geq 0}$ be a regular Glaeser stable shape field. Then there exists a finite constant $A^\sharp$ such that:
	
	Given $x_1,...,x_{k^\sharp}\in E$, there exist $P_i\in \Gamma(x_i,A^\sharp)$ such that
	\begin{equation}
	    |\D^\alpha(P_i-P_j)(x_j)|\le A^\sharp|x_i-x_j|^{m-|\alpha|}\text{ for }|\alpha|\le m, 1\le i,j\le k^\sharp.
	\end{equation}
	
	\end{lemma}

The $A^\sharp$ from Lemma \ref{lemma:finiteness} will be fixed and referenced throughout the remainder of the paper.

	\begin{proof}
	\newcommand{\xn}{x^{(\nu)}}
	\newcommand{\an}{A^{(\nu)}}
	We slightly modify the proof given in \cite{F06}.
	
	Suppose towards a contradiction, that for each $\nu = 1, 2, \cdots$ we can find
	\begin{itemize}
	    \item $\xn_1, \cdots, \xn_\ksh \in E$, and
	    \item $\an > 0$, with $\lim\limits_{\nu\to\infty}\an = \infty$,
	\end{itemize}
	such that for each $\nu$, 
	\begin{itemize}
	    \item[\LA{eq.6.1.contra}] there do {\em not} exist polynomials $P_1, \cdots, P_\ksh \in \p$ such that
	\end{itemize}
\begin{equation}
    P_j\in\Gamma(x_j^{(\nu)},A^{(\nu)})\text{ for }j = 1, \cdots, \ksh\text{; and}
\end{equation}
\begin{equation}
    \abs{\D^\alpha(P_i- P_j)(\xn_j)} \leq \an\abs{\xn_i - \xn_j}^{\ma}\text{ for }\abs{\alpha} \leq m, 0 \leq i,j \leq \ksh.
\end{equation}

	\newcommand{\xft}{x^{(\infty)}}
	\newcommand{\aft}{A^{(\infty)}}
	Since $E$ is compact, by passing to a subsequence, we may assume that
	\begin{equation*}
	    \xn_j \to \xft_j \in E
	    \text{ as } \nu 
	    \to \infty, \text{ for } j = 1, \cdots, \ksh.
	\end{equation*}

	\newcommand{\mmax}{{\mu_{\max}}}
	
	Let $z_1, \cdots, z_\mmax$ be an enumeration of the distinct elements of $\set{\xft_1, \cdots, \xft_\ksh}$. For each $\mu =1, \cdots, \mmax$, let $S(\mu):= \set{j : \xft_j = z_\mu, \, 1 \leq j \leq \ksh}$. Therefore, for sufficiently large $\nu$, we have
	\begin{equation}
	    \abs{\xn_j - z_\mu} \leq \eta_1 
	    \text{ for all } j \in S(\mu),
	    \label{eq.6.1.1}
	\end{equation}
	and
	\begin{equation}
	    \abs{\xn_j - \xn_{j'}} \geq \eta_2
	    \text{ for } j \in S(\mu),\, j' \in S(\mu')\text{ with }\mu \neq \mu',
	    \label{eq.6.1.2}
	\end{equation}
	
	where $\eta_1$ and $\eta_2>0$ are chosen to be independent of $\nu$. For the rest of the proof, we use $A_0$, $A_1$, etc., to denote constants independent of $\nu$. 
	
	We now apply the hypothesis that $(\Gamma(x))_{x\in E}$ is Glaeser-stable. 
	
	\newcommand{\pn}{P^{(\nu)}}
	Fix $\mu \in \set{1, \cdots, \mmax}$. We set $x_0 = z_\mu$, $P_0 \in \Gamma(x_0, A_0)$, and $\epsilon = 1$ in the definition of Glaeser refinement. Let $\nu$ be sufficiently large. We set $x_j := \xn_j$ for $j \in S(\mu)$, $x_j := z_\mu$ for $j \notin S(\mu)$ ($1 \leq j \leq \ksh$). Since $\G(x)$ is its own Glaeser refinement, it follows from \eqref{eq.6.1.1} that there exist 
	\begin{equation}\label{eq:membership in Gamma}
	    P^{(\nu)}_j \in \Gamma(\xn_j)\text{, for }j \in S(\mu)
	\end{equation}
	$\nu$ sufficiently large, such that
	\begin{equation}\label{eq:random label}
	    \abs{\D^\alpha(\pn_i - \pn_j)(\xn_j)} \leq \abs{\xn_i - \xn_j}^{m-\abs{\alpha}}\text{ for }i,j \in S(\mu), \abs{\alpha} \leq m.
	\end{equation}
Furthermore, it follows from \eqref{eq.6.1.1}, \eqref{eq:membership in Gamma}, \eqref{eq:random label}, and \eqref{regularity 2} (upon taking $\eta_1$ sufficiently small) that
\begin{equation}
    P^{(\nu)}_j \in \Gamma(\xn_j,A_1)\text{, for }j \in S(\mu).
\end{equation}
	
	We repeat the argument above for $\mu = 1, \cdots, \mmax$. Thus, for sufficiently large $\nu$, we can find polynomials $\pn_1, \cdots, \pn_\ksh$ that satisfy the following:
	\begin{equation}
	   \pn_j\in\Gamma(\xn_j, A_2 )\text{ for } j = 1, \cdots, \ksh; \text{ and }
	    \label{eq.6.1.4}
	\end{equation}
	\begin{equation}
	    \abs{\D^\alpha(\pn_i - \pn_j)(\xn_j)} \leq \abs{\xn_i - \xn_j}^{\ma} \text{ for } \abs{\alpha} \leq m, \, i,j \in S(\mu), \, \mu = 1, \cdots, \mmax.
	    \label{eq.6.1.5}
	\end{equation}
	
	Observe that by \ref{regularity 1}, \eqref{eq.6.1.4} implies

\begin{equation}\label{eq:boundedness of P_j}
    |\da P_j^{(\nu)}(x_j^{(\nu)}|\le A_2\text{ for }|\alpha|\le m.
\end{equation}
	
	Combining \eqref{eq.6.1.2} and \eqref{eq:boundedness of P_j}, we see that
	\begin{equation}\label{eq:23.7}
	    \abs{\D^\alpha(\pn_i - \pn_j)(\xn_j)} \leq A_3\abs{\xn_i - \xn_j}^{\ma}
	    \text{ for }
	    \abs{\alpha} \leq m,\, i \in S(\mu),\, j \in S(\mu'),\, \mu \neq \mu'.
	\end{equation}
	Together with \eqref{eq.6.1.5} and \eqref{eq:23.7}, we see that
	\begin{equation}
	    \abs{\D^\alpha(\pn_i - \pn_j)(\xn_j)} \leq A_4\abs{\xn_i - \xn_j}^\ma 
	    \text{ for }
	    \abs{\alpha} \leq m,\, 1 \leq i,j \leq \ksh.
	    \label{eq.6.1.6}
	\end{equation}
	
	For sufficiently large $\nu$, we have $\an > A_2$ and $\an > A_4$, with $A_2$ and $A_4$, respectively, as in and \eqref{eq.6.1.4} \eqref{eq.6.1.6}. Then,  \eqref{eq.6.1.4}, and \eqref{eq.6.1.6} together contradict \hyperref[eq.6.1.contra]{($\star$)}. 
	
	The lemma is proved.
	
	\end{proof}

	\begin{lemma}\label{lemma:clustering}\cite{F06}
		\newcommand{\nmax}{{\nu_{\max}}}
		Let $ S \subset \R^n $ with $ 2 \leq \#(S) \leq \ksh $. Then we may partition $ S $ into subsets $ S_1, \cdots, S_{\nmax} $ with the following properties.
		\begin{enumerate}[label = (\alph*)]
			\item $ \#(S_\ell) \leq \#(S)  $ for each $ \nu= 1, \cdots, \nmax $. 
			\item If $ x \in S_\nu $ and $ y \in S_\mu $ with $ \nu \neq \mu $, then $ \abs{x - y} > c(\ksh) \cdot  {\rm diam}(S) $. Here, $ c(\ksh) $ depends only on $ \ksh $.
		\end{enumerate}
	\end{lemma}

\begin{lemma}\label{lemma:linear perturbation}
Suppose $W\subset \R^n$ is an $r$-dimensional subspace, $w\in W$, and $w_0,...,w_r$ form a nondegenerate affine $r$-simplex containing $B(w,\delta)\cap W$ for some $\delta>0$. Then, there exists $\eta>0$ for which the following holds:

Let $W'\subset\R^n$ be another $r$-dimensional subspace. If $w_0',...,w_r'\in W'$ satisfy
\begin{equation}
    |w_j-w_j'|<\eta\text{ for } j = 0,1,\cdots,r,\label{eq:6.3.2}
\end{equation}
and $w'\in W'$ satisfies
\begin{equation}
    |w-w'|<\eta,
    \label{eq:6.3.1}
\end{equation}
then $w'$ is contained in the convex hull of $w_0',...,w_r'$.
\end{lemma}



\begin{proof}
    \newcommand{\dt}{{\tilde{d}}}
    Since $w \in Conv(w_0, w_1, \cdots, w_r)$, there exist $\lambda_0, \lambda_1, \cdots, \lambda_r$ such that $w = \sum_{i = 0}^r \lambda_iw_i$. Let $\xi_1, \cdots, \xi_r$ be an orthonormal basis for $\mathrm{span}(w_0 - w_1, \cdots, w_0 - w_r)$. Then the coefficients $\lambda_0, \lambda_1, \cdots, \lambda_r$ satisfy the following constrained linear system.
    \begin{align}
        &\sum_{i = 0}^r \lambda_i(w_i\cdot\xi_j) = w\cdot \xi_j \text{ for } j = 1, \cdots, r;
        \label{eq:6.3.3}
        \\
        &\sum_{i = 0}^r \lambda_i = 1;
        \label{eq:6.3.4}\\
        &0 \leq \lambda_i \leq 1\text{ for } i = 0, 1, \cdots, r.
        \label{eq:6.3.5}
    \end{align}
    Note that the system \eqref{eq:6.3.3}--\eqref{eq:6.3.4} is nondegenerate, since $w_0, w_1, \cdots, w_r$ form a nondegenerate affine $r$-simplex.
    
    Moreover, since $B(w,\delta)\cap W \subset Conv(w_0,w_1, \cdots, w_r)$, we may replace \eqref{eq:6.3.5} by
    \begin{equation}
    c \leq \lambda_i \leq 1
    \text{ for } i = 0, 1, \cdots, r
        \label{eq:6.3.6}
    \end{equation}
    for some small constant $c>0$. 
    
    If $\eta$ is sufficiently small, \eqref{eq:6.3.2} implies that $w_0', \cdots, w_r'$ also form a nondegenerate $r$-simplex in $W'$. We may find $\lambda_0', \lambda_1', \cdots, \lambda_r'$ such that 
    \begin{align}
        &\sum_{i = 0}^r \lambda_i'(w_i'\cdot\xi_j) = w'\cdot\xi_j
        \text{ for }j = 1, \cdots, r;\label{eq:6.3.7}\\
        &\sum_{i = 0}^r \lambda_i' = 1.\label{eq:6.3.8}
    \end{align}
    
    Note that $w' \in Conv(w_0', \cdots, w_r')$ if and only if $\lambda_0', \lambda_1', \cdots, \lambda_r'$ satisfy the system \eqref{eq:6.3.7} and \eqref{eq:6.3.8} coupled with the constraint
    \begin{equation}
        0 \leq \lambda_i' \leq 1\text{ for } i = 0,1,\cdots, r.
        \label{eq:6.3.9}
    \end{equation}
    
    By choosing $\eta$ to be sufficiently small, \eqref{eq:6.3.1} and \eqref{eq:6.3.2} imply that the coefficients of the linear system \eqref{eq:6.3.3}--\eqref{eq:6.3.4} are sufficiently close to that of \eqref{eq:6.3.7}--\eqref{eq:6.3.8}. This forces $\lambda_i$ to be sufficiently close to to $\lambda_i'$ for $i = 0,1,\cdots,r$. Thanks to \eqref{eq:6.3.6}, we see that $\lambda_0',\lambda_1',\cdots,\lambda_r'$ satisfy the constraint \eqref{eq:6.3.9}. This proves the lemma.
\end{proof}

\begin{lemma}[Helly's Theorem]
Let $(\mathcal{K}_\alpha)_{\alpha\in \mathcal{A}}$ be a family of compact, convex sets in $\R^N$. If any $N+1$ of the $\mathcal{K}_\alpha$ have nonempty intersection, then $\bigcap_{\alpha\in \mathcal{A}}\mathcal{K}_\alpha\neq\emptyset$.
\end{lemma}

For a proof of the above, see \cite{RT-c}, for example.

\begin{lemma}[Whitney Extension Theorem for Finite Sets]\label{lemma:WET}\cite{W34-1}
Let $S\subset\R^n$ be finite, and for each $x\in S$, let $P^x\in\p$. If
\begin{equation}
    |\da P^x(x)|\le M\text{ for }|\alpha|\le m
\end{equation}
and
\begin{equation}
    |\da(P^x-P^y)(x)|\le M|x-y|^{m-|\alpha|}\text{ for }|\alpha|\le m, x,y\in S, x\neq y,
\end{equation}
then there exists $F\in C^m(\R^n)$ such that $\|F\|_{C^m(\R^n)}\le CM$ and $J_x F=P^x$ for all $x\in S$. Here, $C$ depends only on $m,n$.
\end{lemma}

\section{Convex Sets and Strata}\label{sec:convex}
    
    In this section, we write $A,A',$ etc., to denote quantities that are strictly greater than $A^\sharp$, with $A^\sharp$ as in Lemma \ref{lemma:finiteness}.
    
	\begin{definition}
	For $x_0 \in E$, $k \in \mathbb{N}_0$, and $A> 0$, we let $\Gamma(x_0,k,A)$ be the set of $P_0\in\Gamma(x)$ such that for any $x_1,...,x_{k}\in E$, there exist $P_1, \cdots, P_{k}\in\p$, with 
	\begin{itemize}
	    \item $P_i \in \Gamma(x,A)$ for $i = 1, \cdots, k$; and
	    \item $|\D^\alpha(P_i-P_j)(x_j)|\le A|x_i-x_j|^{m-|\alpha|}\text{ for }|\alpha|\leq m, 0\le i,j\leq k$.
	\end{itemize}
	
	\end{definition}
	
Note that by the Finiteness Lemma, there exists $A^\sharp<\infty$ such that $\Gamma(x,k,A^\sharp)$ is nonempty for all $x\in E$ and $k\le\ksh$.
  
	
	\newcommand{\ringb}{\overline{\mathcal{R}}}
Let $\ringb_x$ denote the ring of $(m-1)$-jets of functions at $x$ and let $\pi_x : \ring_x \to \ringb_x$ be the natural projection. We identify $\ringb_x$ with the space $\pbar$ of polynomials with degree no greater than $(m-1)$. For $A>0$, we define
	\begin{equation}
	    \Gb(x,k,A) := \pi_x \G(x,k,A).
	\end{equation}
	We note that $\Gb(x,k,A)$ is convex, since $\pi_x$ is linear.
	
	Similarly, let 
	
	\begin{equation}
	    \Gb(x,A) := \pi_x \G(x,A).
	\end{equation}

\subsection{Strata}
	
	The following three lemmas will assist in defining the dimension of a Glaeser stable bundle.

	\begin{lemma}\label{lemma:same dimension 1}
	Let $A\ge 2A^\sharp$, $x_0\in E$, and $k\in\N$. Then,
	\begin{enumerate}
	    \item $\dim\Gamma(x_0)=\dim \Gamma(x_0,A)=\dim\Gamma(x_0,2A^\sharp)$.
	    \item $\dim \Gamma(x_0,k,A)=\dim\Gamma(x_0,k,2A^\sharp)$.
	\end{enumerate}
	\end{lemma}
	
	\begin{proof}
	Suppose the contrary to the first conclusion, that there exists $A>2A^\sharp$ such that $\dim \Gamma(x_0,A)>\dim\Gamma(x_0,2A^\sharp)=d$ and pick $P\in\Gamma(x_0,A)$ such that $P$ is not contained in the $d$-dimensional hyperplane in $\p$ containing $\Gamma(x_0,2A^\sharp)$.
	
	By Lemma \ref{lemma:finiteness}, $\Gamma(x_0,A^\sharp)$ is nonempty, so take $P'\in\Gamma(x_0,A^\sharp)$. 
	
	Now consider the sequence $P_n=(1/n)P'+(1-1/n)P$. By convexity, $P_n\in\Gamma(x_0,A)$ for all $n$. But by supposition, $P_n\notin \Gamma(x_0,2A^\sharp)$ for all $n$, contradicting \eqref{regularity 3}. Thus, $\dim \Gamma(x_0,A)=\dim\Gamma(x_0,2A^\sharp)$. $\dim\Gamma(x_0)=\dim\Gamma(x_0,2A^\sharp)$ follows by taking unions.
	
	Now suppose the contrary to the second conclusion. Analogous to the previous situation, there exists $A>2A^\sharp$ such that $\dim \Gamma(x_0,k,A)>\dim\Gamma(x_0,k,2A^\sharp)=d'$ and pick $P\in\Gamma(x_0,k,A)$ such that $P$ is not contained in the $d'$-dimensional hyperplane in $\p$ containing $\Gamma(x_0,k,2A^\sharp)$.
	
	Let $P'\in\Gamma(x_0,k,A^\sharp)$ and again consider the sequence $P_n=(1/n)P'+(1-1/n)P$. By our prior reasoning, there exists $N$ such that $n\ge N$ implies $P_n\in\Gamma(x_0,3/2A^\sharp)$.
	
	Now, let $x_1,...,x_{\kb}\in E$. Choose $Q_j\in\Gamma(x_j,1)$ such that 
	\begin{equation}
	    |\d^\alpha(Q_i-Q_j)(x_i)|\le |x_i-x_j|^{m-|\alpha|}\text{ for }|\alpha|\le m, 0\le i,j\le k
	\end{equation}
	with $Q_0=P$. Similarly, choose $Q_j'\in\Gamma(x_j,A)$ such that
	
\begin{equation}
    |\d^\alpha(Q'_i-Q'_j)(x_i)|\le A|x_i-x_j|^{m-|\alpha|}\text{ for }|\alpha|\le m, 0\le i,j\le k
\end{equation}
	with $Q_0'=P'$.
	
	Form sequences $Q_j^{(n)}=(1/n)Q_j'+(1-1/n)Q_j$. By \eqref{regularity 3} (and the independence of $\delta$ of $x$ in this condition), we may take $N$ independent of $x_1,...,x_{k}$ such that $n\ge N$ implies $Q_j^{(n)}\in\Gamma(x_j,2A^\sharp)$. Thus, for $n\in N$, $P'_n\in\Gamma(x_0,k,2A^\sharp)$. This is a contradiction, so $\dim \Gamma(x_0,k,A)=\dim\Gamma(x_0,k,2A^\sharp)$.
	\end{proof}
	
	\begin{lemma}\label{lemma:same dimension 2}
	Let $A\ge 2A^\sharp$, $x_0\in E$, and $k\le \ksh$. Then,
	\begin{equation}
	    \dim \Gamma(x_0,A)=\dim\Gamma(x_0,k,A).
	\end{equation}
	\end{lemma}
	
	\begin{proof}
	Suppose the contrary. Then by Lemma \ref{lemma:same dimension 1}, $\dim\Gamma(x_0,k,A')\le \dim\Gamma(x_0,A)-1$ for all $A'>A$. In particular, there exists $P\in\Gamma(x_0,A)$ such that for all $A'>0$ $P\notin\Gamma(x_0,k,A')$. 
	
	Define the convex bundle $(\Gamma'(x))_{x\in E}$, where
	\begin{equation}
	    \Gamma'(x,M) = \left\{
     \begin{array}{lr}
       \Gamma(x,M) & \text{if } M\geq 0, x\ne x_0\\
       \{P\} & \text{if } x=x_0, M\ge A\\
       \emptyset & \text{if } x=x_0, M<A.
     \end{array}
   \right.
	\end{equation}
Observe that 

By Lemma \ref{lemma:finiteness} and the fact $k\le \ksh$, there exists $\tilde{A}>0$ such that $\Gamma'(x_0,k,\tilde{A})$ is nonempty (in particular, $\Gamma'(x_0,k,\tilde{A})=\{P\}$), where the $\Gamma'(x,k,A)$ are defined analogously to the original $\Gamma(x,k,A)$ but with $\Gamma'(x,M)$ in place of $\Gamma(x,M)$.
	
Since $\Gamma'(x,A)\subset \Gamma(x,A)$ for all $x\in E$ and $A>0$, $\Gamma'(x_0,k,\tilde{A})\subset \Gamma(x_0,k,\tilde{A})$ for all $x\in E,A>0$, thus $P\in\Gamma(x_0,k,\tilde{A})$, a contradiction.
	\end{proof}

Motivated by the prior two lemmas, we make the following definition. Given a graded holding space $(H(x,M))_{x\in E}$, define $W_H(x)$ as the (unique) $\dim H(x)$-dimensional subspace of $\p$ containing $H(x)$.
	
\begin{lemma}\label{lemma:same dimension merged}
Let $A\ge 2A^\sharp, x\in E,$ and $k\le \ksh$. Then,
\begin{equation}\label{eq:all equal dimensions}
    \dim W_{\Gamma}(x)=\dim\G(x)=\dim\G(x,A)=\dim\G(x,k,A)=\dim\G(x,2A^\sharp)=\dim\G(x,k,2A^\sharp).
\end{equation}
and
\begin{equation}\label{eq:all equal dimensions bar}
    \dim \pi_xW_{\Gamma}(x)=\dim\Gbar(x)=\dim\Gbar(x,A)=\dim\Gbar(x,k,A)=\dim\Gbar(x,2A^\sharp)=\dim\Gbar(x,k,2A^\sharp).
\end{equation}

\end{lemma}

\begin{proof}
Lemmas \ref{lemma:same dimension 1} and \ref{lemma:same dimension 2} immediately imply \eqref{eq:all equal dimensions}. Since each of the relevant $\Gamma$'s is contained in $\Gamma(x)$, not only are they all the same dimension, but they are contained in the same $\dim\Gamma(x)$-dimensional hyperplane. Thus, the dimensions of the projections under $\pi_x$ are all equal as well, establishing \eqref{eq:all equal dimensions bar}.
\end{proof}





	\begin{definition}\label{def.stratum}
		Let $ E \subset \R^n $ be a compact set. Let $ (H(x))_{x \in E} $ be a convex bundle over $ E $. For each $ x \in E $, we define the \underline{signature} of $ x $ to be
		\begin{equation*}
		{\rm sig}(x) := \brac{\dim H(x), \dim [\ker \pi_x\cap H(x)]}.
		\end{equation*}
		For given integers $ k_1, k_2 $, we define
		\begin{equation*}
		E(k_1,k_2) := \{x \in E : {\rm sig}(x) = (k_1,k_2) \}.
		\end{equation*}
		We define the stratum
		\begin{equation*}
		E_1 := E(k_1^*, k_2^*)
		\end{equation*}
		where $ k_1^* $ is as small as possible, and $ k_2^* $ is as large as possible for this given $ k_1^* $. We call $ E_1 $ the \underline{lowest stratum}. 
	\end{definition}

	\subsection{Lowest stratum}

	\begin{lemma}\label{lem:low-strat}
		Let $ E \subset \R^n $ and $(H(x))_{x \in E} $ be a convex, Glaeser stable bundle. Let $ E_1 $ be the lowest stratum as in Definition \ref{def.stratum}. Then $ E_1 $ is compact. 
	\end{lemma}

	\begin{proof}
		\newcommand{\tx}{\tilde{x}}
		\newcommand{\tp}{\tilde{P}}
		Let $ x \in E $. Suppose $ \dim H(x) = d $. Let $ P_1, \cdots, P_d $ be the vertices of a nondegenerate affine $ d $-simplex in $\text{int} H(x) $. Thus, small perturbation of $ \{P_1, \cdots, P_d\} $ remain the vertices of a nondegenerate affine $ d $-simplex. 
		
		By the Glaeser stability of $(H(x))_{x\in E}$, for any $ \tx \in E $ sufficiently close to $ x $, we may find $ \tp_1, \cdots, \tp_d \in H(\tx) $, such that they are the vertices of a nondegenerate affine $ d $-simplex. Hence,
		\begin{equation*}
		\dim (H(\tx)) \geq d
		\text{ for any $ \tx $ sufficiently close to $ x $.}
		\end{equation*}
		It follows that
		\begin{equation*}
		\{x \in E : \dim H(x)  < d\}
		\end{equation*}
		is a closed set for any $ d \in \mathbb{N}_0 $. In particular, for $ k_1 = \min_{y \in E}\dim H(y) $, the set
		\begin{equation}
		E(k_1):= \bigcup_{k_2 \in \mathbb{N}_0}E(k_1,k_2)
		\text{ is closed.}
		\label{eq.Ek1closed}
		\end{equation}

		By hypothesis of Theorem \ref{thm:heart of the matter} again, we see that the map
		\begin{equation*}
		x \mapsto W_H(x)
		\end{equation*}
		is continuous from $ E(k_1) $ to $\mathcal{G}$, the Grassmannian of $k_1$-planes in $\p$.
		
		Set $ k_2 = \max_{y \in E(k_1)}\dim(\ker \pi_y \cap W_H(y)) $, noting that $\dim(\ker \pi_y \cap W_H(y))=\dim(\ker \pi_y \cap H(y))$. By Definition \ref{def.stratum}, we see that
		\begin{equation*}
		E_1 = \{x \in E(k_1,k_2) : \dim (\ker \pi_x \cap H(x)) = k_2\}.
		\end{equation*}
		We will show that $ E_1 $ is closed.
		
		Let $ (x_\nu)_{\nu = 1}^\infty $ be a sequence in $ E_1 $ converging to $ x \in \R^n $. By \eqref{eq.Ek1closed}, we know that $ x \in E(k_1) $, and thus, $ W_H(x_\nu) $ converges to $ W_H(x) $ in $ \mathcal{G}$. Passing onto a subsequence, we can further assume that $ \ker \pi_{x_\nu} \cap W_H(x_\nu) $ converges to some (not necessarily unique) $ W \in Conv(\P) $ with $ \dim W = k_2 $. 
		
		Now, $ W \in H(x) $ and $ \pi_x|_{W} \equiv 0 $. Therefore, 
		\begin{equation*}
		\dim (\ker \pi_x \cap W_H(x)) \geq k_2.
		\end{equation*}
		By our choice of $ k_2 $, we must have
		\begin{equation*}
		\dim (\ker \pi_x \cap W_H(x)) = k_2.
		\end{equation*}
		Hence, $ x \in E_1 $ as claimed.
		
		Since $ E $ is compact and $ E_1 \subset E $ is closed, the lemma follows.
		
	\end{proof}
	
		\newcommand{\db}{{\bar{d}}}
	Recall from Definition \ref{def.stratum} that $ E_1 $ denotes the lowest stratum of $ E $. By Lemma \ref{lem:low-strat}, $ E_1 $ is compact. Furthermore, for a convex bundle $(H(x))_{x\in E}$, $ \dim H(x) $ and $ \dim(\ker \pi_x \cap H(x)) $ are constant for all $ x \in E_1 $.
	
	Limiting our attention to the bundle $(\G(x,M))_{x\in E}$ from the hypotheses of Theorem \ref{thm:heart of the matter}, we set
	\begin{equation*}
	\begin{split}
	d &:= \dim \G(x) \text{ for all } x \in E_1, \text{ and }\\
	\db &:= \dim \Gb(x) \text{ for all } x\in E_1.
	\end{split}
	\end{equation*}
	
	
	
	
	\subsection{Glaeser Stability}
	
	
Given a shape field $(\Gamma(x,M))_{x\in E, M\geq 0}, x_0\in E,$ and $A>0$, define
\begin{equation}
    \Gamma^o(x,A):=\bigcup_{M<A}\Gamma(x,M).
\end{equation}
	
	\begin{lemma}\label{lemma:GS for Gamma(x,A)}
	Let $(\G(x,M))_{x \in E}$ be a regular, convex, Glaeser-stable, shape field and fix $A>0$. The bundle $(\Gamma^o(x,A))_{x\in E}$ is Glaeser stable.
	\end{lemma}
	
	

\begin{proof}
Fix $x_0\in E$ and $P_0\in \Gamma^o(x,A)$. Then, there exists $M<A$ such that $P_0\in\Gamma(x,M)$. By \ref{regularity 2}, there exists $\eta>0$ such that for any $x'\in E\cap B(x,\eta)$, if $P'\in \Gamma(x')$ with 

\begin{equation}\label{eq:application of regularity 1}
    |\partial^\alpha(P-P')(x)|, |\partial^\alpha(P-P')(x')|\le \eta |x-x'|^{m-|\alpha|}\text{ for all }|\alpha|\le m,
\end{equation}
then $P'\in \Gamma(x',\frac{M+A}{2})\subset\G^o(x',A)$.

Let $0<\eps<\eta$. By the Glaeser stability of $(\G(x))_{x\in E}$, there exists $0<\delta<\eta$ such that if $x_1,...,x_{\ksh}\in E\cap(x,\delta)$, then there exist $P_1,...,P_{\ksh}\in\p$ such that
\begin{equation}
    P_j\in \G(x_j)\text{ for }j=1,...,\ksh
\end{equation}
and
\begin{equation}\label{eq:another closeness of Pi and Pj}
    |\da(P_i-P_j)(x_i)|< \eta|x_i-x_j|^{m-|\alpha|}\text{ for }|\alpha|\le m, 0\le i,j\le\ksh.
\end{equation}

Since \eqref{eq:another closeness of Pi and Pj} implies \eqref{eq:application of regularity 1}, we have $P_j\in\G^o(x_j,A)$ for $j=1,...,\ksh$, which is precisely what we wanted to show.

\end{proof}

In the sequel, we use $\Gamma(x,A)$ to denote $\G^o(x,A)$, simplifying notation. Note that $\Gamma^o(x)$ may not be a closed subset of $\p$, but since $\G^o(x,A)\subset \G(x,A)\subset \G^o(x,A')$ for $A<A'$, we see
\begin{equation}
    \G(x)=\bigcup_{M\geq 0}\G^o(x,M)
\end{equation}
and $(\G^o(x,M))_{x\in E}$ will share most relevant properties with $(\G(x,M))_{x\in E}$, in particular, $(C,\dmax)$-convexity and regularity.

	\newcommand{\kt}{{\tilde{k}}}

	\newcommand{\namedtheorem}[3][2]{
		\theoremstyle{plain}
		\newtheorem*{#1}{#1}
		\begin{#1}[#3]
			#2
		\end{#1}
	}
	
\section{More on Convex Sets}\label{sec:more convex}

In Section \ref{sec:uniform}, we will establish a uniform choice of $\delta$ in the definition of Glaeser stability which will allow us to choose polynomials from the $\Gamma(x,A)$. The purpose of this section is to show that we may take polynomials from the smaller $\Gamma(x,k,A)$ (where $k$ and $A$ will be chosen appropriately).

\begin{lemma}\label{lemma:5.2 equivalent}
Let $x_0\in E, A>0$, and $\kb\in\N$. Suppose $P_0\in\Gamma(x_0,\kb,A)$ and $P_0'\in\Gamma(x_0,A)$ such that $\pi_{x_0}(P_0-P_0')=0$. Then, there exists $C$, depending solely on $m,n$ such that $P_0'\in\Gamma(x_0,\kb,CA)$.
\end{lemma}

\begin{proof}
Let $x_1,...,x_{\kb}\in E$. By the fact that $P_0\in\Gamma(x_0,\kb,A)$, there exist $P_j\in\Gamma(x,A)$ ($1\le j\le \kb$) such that
\begin{equation}\label{eq:from definition of gamma(x,kb,A)}
    |\da(P_i-P_j)(x_j)|\le A|x_i-x_j|^{m-|\alpha|}\text{ for }|\alpha|\le m, 0\le i,j,\le \kb.
\end{equation}

The conclusion will follow from establishing
\begin{equation}\label{eq:compare at x_0}
    |\da(P_0'-P_j)(x_0)|\le A'|x_0-x_j|^{m-|\alpha|}\text{ for }|\alpha|\le m, 1\le j\le \kb
\end{equation}
and
\begin{equation}\label{eq:compare at x_j}
    |\da(P_0'-P_j)(x_j)|\le A'|x_0-x_j|^{m-|\alpha|}\text{ for }|\alpha|\le m, 1\le j\le \kb.
\end{equation}

Since $P_0,P_0'\in\Gamma(x_0,A)$, by \ref{regularity 1} we have
\begin{equation}\label{eq:use of reg 1}
    |\da P_0(x_0)|,|\da P_0'(x_0)|\le A \text{ for }|\alpha|\le m.
\end{equation}

Therefore, since $\da P_0, \da P_0'$ are constant functions for $|\alpha|=m$, we have
\begin{equation}\label{eq:alpha=m case}
    |\da(P_0-P_0')(y)|\le A_1, y\in\R^n, |\alpha|=m,
\end{equation}
where $A_1$ depends solely on $A, m, n$, and $\diam E$. We are able to ignore the dependence on $\diam E$ in the future, as $E\subset Q_0$, a cube of fixed length.

By \eqref{eq:from definition of gamma(x,kb,A)} and \eqref{eq:alpha=m case}, we see for $1\le j\le \kb$ and $|\alpha|=m$
\begin{align}
    |\da(P_0'-P_j)(x_0)|&\le|\da(P_0'-P_0)(x_0)|+|\da(P_0-P_j)(x_0)|\\
    &\le A_1+A|x_0-x_j|^{m-|\alpha|}\\
    &=(A_1+A)|x_0-x_j|^{m-|\alpha|},
\end{align}
establishing \eqref{eq:compare at x_0} with $A'=A_1+A$. \eqref{eq:compare at x_j} follows similarly.


It remains to check the cases of $|\alpha|\le m-1$. Here, we see that \eqref{eq:compare at x_0} follows immediately since $\pi_{x_0}(P_0-P_0')=0$ implies $\da P_0(x_0)=\da P_0'(x_0)$ for $|\alpha|\le m-1$.

To establish \eqref{eq:compare at x_j}, we observe that for $1\le j\le \kb$ and $|\alpha|\le m-1$, by \eqref{eq:from definition of gamma(x,kb,A)} and \eqref{eq:use of reg 1}
\begin{align}
    |\da(P_0'-P_j)(x_j)|&\le|\da(P_0'-P_0)(x_j)|+|\da(P_0-P_j)(x_j)|\\
    &\le A_2|x_0-x_j|^{m-|\alpha|}+A|x_0-x_j|^{m-|\alpha|},
\end{align}
where $A_2$ depends solely on $A, m, n$. This gives \eqref{eq:compare at x_j} with $A'=A_2+A$.

\end{proof}

\begin{lemma}\label{lemma:5.6 equivalent}
Let $A>0$ and $x\in E_1$. Suppose $ 1 + (D+1)\cdot \kt \leq \kb $. If $ \oP\in\text{int}\Gbar(x,\kb,A) $, then there exists $0<\delta<1$ such that if $x'\in E_1\cap B(x,\delta)$, there exists $ \oP' \in \Gbar(x',\kt,A) $ such that
		\begin{equation}
		\abs{\d^\alpha (\oP - \oP')(x)}, \abs{\d^\alpha (\oP - \oP')(x')} \leq A\abs{x - x'}^{m-\abs{\alpha}}
		\text{ for } \abs{\alpha} \leq m-1.
		\end{equation}
\end{lemma}

\begin{proof}
Choose $\oP_0,...,\oP_{\dbar}\in\Gb(x,\kb,A)$ such that $Conv(\oP_0,...,\oP_{\dbar})$ is a $\db$-dimensional simplex containing $B(\eta,\oP)\cap \pi_x W_{\G}(x)$ for some $\eta>0$. Write $\oP_j=\pi_x P_j$ and $\oP=\pi_x P$ for some $P_j,P\in\Gamma(x,\kb,A), 0\le j\le\db.$
		
By definition of the $\Gamma(x,\kb,A)$, for any $x'\in E_1$, there exists $P_j'\in\Gamma(x',A)$ ($0\le j\le d$) such that
	\begin{equation}
		|\da(P_j-P_j')(x')|\le A|x-x'|^{m-|\alpha|}\text{ for }|\alpha|\le m.
	\end{equation}
		
In particular, since $|\da P_j(x)|\le A$ for $|\alpha|\le m$ by \ref{regularity 1},

\begin{equation}\label{eq:close to vertices}
		|\da(\oP_j-\oP_j')(x')|\le A_1|x-x'|^{m-|\alpha|}\text{ for }|\alpha|\le m-1,
	\end{equation}
where $\oP_j'=\pi_{x'}P_j'$ and $A_1$ is a constant depending solely on $A,m,n$.

Applying Lemma \ref{lemma:linear perturbation} to $\oP_0,...,\oP_{\dbar}$, we see that by taking $|x-x'|$ sufficiently small, there exists $0<\delta<1$ such that if $x'\in E_1\cap B(x,\delta)$ and $\oP'\in \pi_{x'}W_{\G}(x')$ satisfies
		\begin{equation}\label{eq:close to base point}
		    |\da(\oP-\oP')(x)|\le A_1|x-x'|^{m-|\alpha|}\text{ for }|\alpha|\le m-1,
		\end{equation}
	    then $\oP'\in \overline{\Lambda}(x'):=Conv(\oP_0',...,\oP_d')\subset\Gbar(x',A)$. $A_1$ will be chosen later and depend solely on $A,m,n$. Here, we implicitly used Lemma \ref{lemma:same dimension merged} to show $\pi_{x'}W_{\G}$ was $\db$-dimensional, allowing us to apply Lemma \ref{lemma:linear perturbation}.
	    
		
		Fix $x'\in E_1\cap B(x,\delta)$. Given a finite set $S \subset E$, define $S^+:= \set{x,x'} \cup S$, and define
		\newcommand{\K}{\mathcal{K}}
		\begin{equation*}
		    \K(S) := \set{\oP' \in \overline{\Lambda}(x') : \begin{matrix}
		  \text{ There exists } P^y \in \Gamma(y,A)\text{ for all } y \in S^+ \text{ such that}\\
		  P^x = P,\, \pi_{x'}P^{x'} = \oP',\, \text{ and }\\
		  \abs{\da(P^y - P^z)(z)} \leq A\abs{y-z}^{\ma} \text{ for } \abs{\alpha} \leq m,\, y,z \in S^+.
		    \end{matrix}}
		\end{equation*}
		
		Each $\K(S)$ is compact, convex subset of $\pbar$, which has dimension $\le D$. Given $S \subset S'$, we have $\K(S') \subset \K(S)$. Given $S^+ \subset E$ with $\#(S^+) \leq \kb + 1$, we see from $P \in \G(x,\kb,A)$ there exists $\tilde{P}\in\P$ such that $\pi_{x'}\tilde{P}$ satisfies the conditions for membership in $\K(S^+)$, except it is only known so far that $\pi_{x'}\tilde{P}\in\overline{\p}$, rather than $\pi_{x'}\tilde{P}\in\overline{\Lambda}(x')$. However, $|\da(P-P^{x'})(x')|\le A|x-x'|^{m-|\alpha|}$ implies \eqref{eq:close to base point} with $A_1$ chosen accordingly; thus, it is required that  $\pi_{x'}\tilde{P}\in\overline{\Lambda}(x')$ thus $\pi_{x'}\tilde{P}\in \K(S^+)$ and $\K(S^+)$ is nonempty.
		
		Suppose $S_1, \cdots, S_{D+1} \subset E$ with $\#(S_i) \leq \tilde{k}$ for $i = 1, \cdots, D+1$. Then $S := S_1 \cup \cdots \cup S_{D+1}$ satisfies $\#(S) \leq \kb + 1$. Hence, $\K(S) \neq \void$, and $\K(S) \subset \K(S_i)$ for each $i$. Thus, $\bigcap_{i = 1}^{D+1}\K(S_i) \neq \void$. By Helly's theorem, there exists $\oP' \in \K(S)$ for every $S \subset E$ with $\#(S) \leq \tilde{k}$. It follows that $\oP' \in \Gb(x',\tilde{k},A)$. Taking $S =\void$ and using the fact that $\oP' \in \K(S)$, we see that the first conclusion follows.

		

\end{proof}

\begin{lemma}\label{lemma:5.7 equivalent}
Suppose $A>0$ and $1 + (D+1)\tilde{k} \leq \kb$. Given $x\in E_1$, and $\overline{Q}\in\text{int}\Gbar(x,\kb,A)$, there exist $\eps_0, \delta_0>0$ such that for any $x\in E_1\cap B(x,\delta_0)$ and any $\overline{Q}'\in\Gbar(x',A)$, if $|\da(\overline{Q}-\overline{Q}')(x)|\le \eps_0$ for $|\alpha|\le m-1$, then $\overline{Q}'\in\Gbar(x',\tilde{k},A)$.
\end{lemma}

\begin{proof}
Choose $\overline{Q}_0,...,\overline{Q}_{\db}\in\text{int}\Gb(x,\kb,A)$ which form a $\db$-dimensional simplex with $\overline{Q}$ lying in its relative interior. Let $\eta>0$ be as in Lemma \ref{lemma:linear perturbation}.

Applying Lemma \ref{lemma:5.6 equivalent} to each of the $\overline{Q}_j$ and taking $\delta_0$ to be the minimum of the associated $\delta$'s, we see that for any $x'\in E_1\cap B(x,\delta_0)$, there exist $\overline{Q}_0',...,\overline{Q}_{\db}'\in\pbar$ such that
\begin{equation}
    \overline{Q}_j'\in\Gb(x',\kb,A)
\end{equation}
and
\begin{equation}\label{eq:the overline Qj are close}
		\abs{\d^\alpha (\overline{Q}_j - \overline{Q}_j')(x)}, \abs{\d^\alpha (\overline{Q}_j - \overline{Q}_j')(x')} \leq A\abs{x - x'}^{m-\abs{\alpha}}
		\text{ for } \abs{\alpha} \leq m-1.
\end{equation}
		
Taking $\delta_0>0$ small enough, by \eqref{eq:the overline Qj are close}, Lemma \ref{lemma:linear perturbation} applies. (By Lemma \ref{lemma:same dimension merged} and the definition of strata, $\dim\Gb(x',A)=\dim\Gb(x,A)=\db$.) Thus, there exists $\eps_0>0$ such that if $\overline{Q}'\in\Gbar(x',A)$, if $|\da(\overline{Q}-\overline{Q}')(x)|\le \eps_0$ for $|\alpha|\le m-1$, then $\overline{Q}'\in Conv(\overline{Q}_0',...,\overline{Q}_{\db}')\subset\Gb(x',\tilde{k},A)$.
\end{proof}

\begin{lemma}\label{lemma:5.8 equivalent}
Suppose $A,A_1 > 0$ and $1 + (D+1)\tilde{k} \leq \kb$.
Given $x\in E_1$, and $Q\in\text{int}\Gamma(x,\kb,A)$, there exist $\eps, \delta>0$ such that for any $x'\in E_1\cap B(x,\delta)$ and any $Q'\in\Gamma(x',A)$, if
\begin{equation}
    |\da(Q-Q')(x)|\le \eps \text{ for }|\alpha|\le m-1,
    \label{eq.8.4.1}
\end{equation}
and
\begin{equation}
    |\da Q'(x)|\le A\text{ for }|\alpha|=m,
    \label{eq.8.4.2}
\end{equation}
then $Q'\in\Gamma(x',\tilde{k},CA)$, with $C$ depending only on $m,n$.
\end{lemma}

\begin{proof}
\newcommand{\alm}{{\abs{\alpha}\leq m}}

Let $x \in E_1$. Let $\eps_0,\delta_0$ be as in Lemma \ref{lemma:5.7 equivalent} with $A = A_1$. Let $\eps$ and $\delta$ be small numbers to be determined, depending only on the parameters appeared above. 

Fix $Q \in \G(x,\kb,A_1)$, $x' \in E_1 \cap B(x,\delta)$, and $Q' \in \G(x',A)$ satisfying \eqref{eq.8.4.1} and \eqref{eq.8.4.2}. 

Since $Q \in \G(x,\kb,A_1)$, we have $\abs{\da Q(x)} \leq A_1$ for $\abs{\alpha} \leq m$. Therefore,
\eqref{eq.8.4.1} and \eqref{eq.8.4.2} imply that $\abs{\da Q'(x)} \leq A_3$ for $\abs{\alpha}\leq m$. Taking $\delta \leq 1$, we have
\begin{equation}
    \abs{\da Q'(x')} \leq A_4 \text{ for }\alm.
    \label{eq.8.4.3}
\end{equation}

\newcommand{\Qb}{\overline{Q}}
We set $\Qb := \pi_x Q$ and $\Qb' := \pi_{x'}Q'$, immediately seeing that by taking $\delta \leq \delta_0$, we have 
\begin{equation}
    \Qb \in \mathrm{int}\,\Gb(x,\kb,A_1),\,x' \in E_1\cap B(x,\delta_0),\text{ and }\Qb' \in \Gb(x',A).
    \label{eq.8.4.6}
\end{equation}

By Taylor's theorem and \eqref{eq.8.4.3}, we have
\begin{equation}
    \abs{\da(\Qb' - Q')(x)} \leq A_3\abs{x-x'}^\ma \leq A_3\delta^\ma \leq A_3\delta\text{ for} \alm.
    \label{eq:8.4.4}
\end{equation}

Taking $\eps$ and $\delta$ to be sufficiently small, we see from \eqref{eq.8.4.1} and \eqref{eq:8.4.4} that
\begin{equation}
    \abs{\da(\Qb' - \Qb)(x)} \leq \eps_0\text{ for }\abs{\alpha} \leq m-1.
    \label{eq.8.4.5}
\end{equation}

In view of \eqref{eq.8.4.6} and \eqref{eq.8.4.5}, we see that Lemma \ref{lemma:5.7 equivalent} applies. Hence, $\Qb' \in \Gb(x',\tilde{k},A_4)$. This means that there exists some $\tilde{Q} \in \G(x',\tilde{k},A_4)$ such that $\Qb' = \pi_{x'}\tilde{Q}$. Now, since $\pi_{x'}(\tilde{Q}-Q') \equiv 0$, we see from Lemma \ref{lemma:5.2 equivalent} that $Q' \in \G(x',\tilde{k},A_5)$.
\end{proof}







\section{Uniform delta and modulus of continuity}\label{sec:uniform}

The purpose of this section is to show that, given $\epsilon>0$, we may take the $\delta$ in the definition of Glaeser stable to be uniform in $x$ and $P_0$ (with some change in $\ksh$), then use this fact and the results of Section \ref{sec:more convex} to construct a modulus of continuity that will let us apply Theorem \ref{thm:shape fields fin prin}.

Given $(x,P),(x',P')\in E\times\P$, we let
\begin{equation}\label{eq:define distance}
    \dist\left((x,P),(x',P')\right):=|x-x'|+\max_\alpha|\da(P-P')(x)|+\max_\alpha|\da(P-P')(x')|.
\end{equation}

Note that the above turns $E\times\P$ into a metric space. While the above description of $\dist(\cdot,\cdot)$ will be most useful, the metric defined by

\begin{equation}
    \rho((x,P),(x',P')):=|x-x'|+\max_\alpha|\da(P-P')(0)|.
\end{equation}
will be useful for establishing a crucial property of $\dist(\cdot,\cdot)$.

\begin{lemma}\label{lemma:for Heine-Borel}
The metrics $\dist(\cdot,\cdot)$ and $\rho(\cdot,\cdot)$ are equivalent, that is, given compact $E\subset\R^n$, there exists $C>0$ such that
\begin{equation}
    C^{-1}\rho((x,P),(x',P'))\le \dist((x,P),(x',P'))\le C\rho((x,P),(x',P'))
\end{equation}
whenever $x,x'\in E, P,P'\in\P$.

In particular, a set is compact in the $\dist$ metric if and only if it is closed and bounded in the $\dist$ metric.
\end{lemma}

\begin{proof}

Write 
\begin{equation}
    P-P'(y)=\sum_\beta c_\beta x^\beta,
\end{equation}
so that
\begin{equation}
    \rho((x,P),(x',P'))=|x-x'|+\max_\alpha |c_\alpha|
\end{equation}
and
\begin{equation}
    \da(P-P')(y)=\sum_{\substack{|\beta|\le m\\\beta\ge\alpha}}c_\beta\beta! y^{\beta-\alpha}.
\end{equation}

Then, letting $M=\sup_{y\in E}|y|<\infty$,
\begin{align*}
    \dist((x,P),(x',P'))&=|x-x'|+\max_\alpha|\da(P-P')(x)|+\max_\alpha|\da(P-P')(x')|\\
    &=|x-x'|+\max_\alpha\left|\sum_{\substack{|\beta|\le m\\\beta\ge\alpha}}c_\beta\beta! x^{\beta-\alpha}\right|+\max_\alpha\left|\sum_{\substack{|\beta|\le m\\\beta\ge\alpha}}c_\beta\beta! (x')^{\beta-\alpha}\right|\\
    &\le|x-x'|+2DM^m\max_{|\beta|\le m}\beta!\max_{|\beta|\le m}|c_\beta|\\
    &\le \left[2DM^m\max_{|\beta|\le m}\beta!\right]\rho((x,P),(x',P')).
\end{align*}

The reverse direction may be proven analogously.
 
The metric $\rho(\cdot,\cdot)$ is by definition equivalent to a Euclidean metric on $E\times\P$ identified in the natural way with $E\times\R^D$. By the above, so is the metric $\dist(\cdot,\cdot)$; thus, the Heine-Borel theorem applies, which is the form of our conclusion.
\end{proof}

For Lemmas \ref{lemma:8.1} through \ref{lemma:uniform delta}, fix a Glaeser stable, convex bundle $(H(x))_{x\in E}$ with lowest stratum $E_1$.

\begin{lemma}\label{lemma:8.1}
Suppose $1+(D+1)\overline{k}\le k^\sharp$. Let $x\in E,$ $P\in\mathrm{int}H(x)$, and $\epsilon>0$ be given. Then there exists $0<\delta<1$ (depending on $P$) such that for every $x'\in E_1\cap B(x,\delta)$, there exists $P'\in H'(x')$ such that
\begin{equation}
    |\partial^\alpha(P-P')(x)|\le\eps|x-x'|^{m-|\alpha|}
    \label{eq.8.1.1}
\end{equation}
and given $x_j'\in E\cap B(x,\delta)$ ($1\le j\le \overline{k}$), there exist $P_j'\in H(x_j)$ such that
\begin{equation}
    |\partial^\alpha(P_i'-P_j')(x_j')|\le\epsilon|x_i'-x_j'|^{m-|\alpha|}, 0\le i,j,\le \kb.
    \label{eq.8.1.2}
\end{equation}
\end{lemma}

We mimic the proof of Lemma 6.1 in \cite{F06}, the most significant change being the introduction of simplices so the compactness hypothesis of Helly's theorem is satisfied.

\begin{proof}

Recall that the bundle $(H(x))_{x \in E}$ is assumed to be Glaeser stable.

Suppose $x' = x$, then we may take $P'  \equiv P$ and \eqref{eq.8.1.1} follows. \eqref{eq.8.1.2} follows from Glaeser stability.

Suppose $x' \neq x$. Since $P\in\text{int}H(x)$, choose $P_0,...,P_d\in H(x)$ which form a $d$-dimensional simplex containing $B(P,\delta)\cap H(x)$. Let $\eta$ be as in the conclusion of Lemma \ref{lemma:linear perturbation}, applied to $P_0,...,P_d$.

Let $0<\eps'<\min\{\eps,\eta\}$ and by the Glaeser stability of $(H(x))_{x\in E}$ pick $0<\delta<\eta$ such that for $j=0,...,d$, if $x'\in E\cap B(x',\delta)$ there exists $P_j'\in H(x')$ such that
\begin{equation}
    |\da(P_j-P_j')(x)|,|\da(P_j-P_j')(x')|\le\eps'|x-x'|^{m-|\alpha|}\le\eta\text{ for }|\alpha|\le m.
\end{equation}

Thus, if $P'\in H(x')$ such that 
\begin{equation}\label{eq:P P' close with eta}
    |\da(P-P')(x)|,|\da(P-P')(x')|\le\eps|x-x'|^{m-|\alpha|}\le\eta\text{ for }|\alpha|\le m,
\end{equation}
then $P'\in\Lambda(x'):=Conv(P_0',...,P_j')$ (as the definition of lowest stratum implies $\dim H(x')=\dim H(x)$).

\newcommand{\K}{\mathcal{K}}
For any finite set $S \subset E \cap B(x,\delta)$ with $S \supset \set{x,x'}$, define
\begin{equation*}
    \K(S):= \set{P' \in \Lambda(x') : \begin{matrix}
    \text{ There exists a map }y\mapsto P^y\text{ from }S\text{ to }\p\text{ such that}\\
    P^x = P,\, P^{x'} = P',\, P^y \in H(y) \text{ for } y \in S, \text{ and }\\
    \abs{\da(P^y - P^z)(z)} \leq \eps'\abs{y-z}^\ma \text{ for }\abs{\alpha} \leq m,\, y,z \in S.
    \end{matrix}}
\end{equation*}
Each $\K(S)$ is a compact, convex subset of $\P$, which has dimension $D$. 

Suppose we are given $S_1, \cdots, S_{D+1} \subset E \cap B(x,\delta)$ , each containing $x$ and $x'$, with $\#(S_i) \leq \kb + 2$ for each $i$. Then $S:= \bigcup_{i = 1}^D S_i \subset E \cap B(x,\delta)$, with $x, x' \in S$, and $\#(S) \leq 2 + (D+1)\kb \leq 1 + \ksh$. Therefore, by the Glaeser stability of $(H(x))_{x\in E}$, there exists $y\mapsto P^y$, with $P^x = P$, $P^y \in H(y)$ for each $y \in S$, and
\begin{equation}
    \abs{\da(P^y - P^z)(z)} \leq \eps'\abs{y-z}^{\ma}\text{ for }\abs{\alpha} \leq m, y,z \in S.
\end{equation}
\eqref{eq:P P' close with eta} forces $P^{x'}\in\Lambda(x')$.

It is clear that $P^{x'} \in \K(S_i)$ for each $i = 1, \cdots, D+1$. Therefore, $\K(S_1), \cdots, \K(S_{D+1})$ have nonempty intersection. By Helly's theorem, there exists $P' \in \K(S') \subset H(x')$, whenever $S' \subset E\cap B(x,\delta)$, with $x,x' \in S'$ and $\#(S') \leq \kb + 2$. \eqref{eq.8.1.1} and \eqref{eq.8.1.2} then follow from the definition of $\K(S')$.
\end{proof}




\begin{definition}
Let $(H(x))_{x\in E}$ be a Glaeser stable bundle. Given $x_0\in E$, $P_0\in H(x_0)$, $\eps>0$ and $\kb\in\N$, we define $\Delta(x_0,P_0;\eps,\kb)$ to be the supremum over all $\delta>0$ such that the following holds:

If $x_1,...,x_k\in E\cap B(x_0,\delta)$, there exist $P_1,...,P_k\in\p$ such that
\begin{equation}
    P_j\in H(x_j), 1\le j\le \kb
\end{equation}
and
\begin{equation}
    |\d^\alpha(P_i-P_j)(x_j)|\le\eps|x_i-x_j|^{m-|\alpha|}\text{ for }|\alpha|\le m, 0\le i,j\le \kb.
\end{equation}
\end{definition}

Note that since the $x_i$ are taken to be in the \textit{open} ball of radius $\delta$, the above holds with $\delta=\Delta(x_0,P_0;\eps,\kb)$.


\begin{lemma}\label{lemma:delta on lines}
Let $x_0\in E$ and $\kb\le \ksh$. Suppose $P_0^{(0)},...,P_0^{(L)}\in H(x)$. Let $P_0 = \sum_{l=0}^L \lambda_l P_0^{(l)}$ for $\lambda_l \in [0,1]$. Then, $\Delta(x,P_0;\kb,\eps)\ge\min_{0\le l\le L}\Delta(x,P_0^{(l)};\kb,\eps)$.
\end{lemma}

\begin{proof}

Let $\eps > 0$. Set $\delta_0:=\min_{0\le l\le L}\Delta(x,P_0^{(l)};\kb,\eps)$.

Let $x_1,...,x_{\overline{k}}\in E\cap B(x,\delta_0)$. By the definition of Glaeser stability, for each $0\le l\le L$ there exist $P_1^{(l)},...,P_{\kb}^{(l)}\in H(x_j)$, and 
\begin{equation}
    \abs{\da(P_i^{(l)} - P_j^{(l)})(x_j)} \leq \eps\abs{x_i - x_j}^\ma
    \text{ for }
    \abs{\alpha} \leq m,\,
    0 \leq i,j \leq \kb, 0\le l\le L.
    \label{eq.8.2.1}
\end{equation}

For $1 \leq j \leq \kb$ and $0\le l\le L$, we set
\begin{equation*}
    P_j := \sum_{l=0}^L \lambda_l P_j^{(l)}.
\end{equation*}
Since $H(x_j)$ is convex for each $j$, we have $P_j \in H(x_j)$. Furthermore, thanks to \eqref{eq.8.2.1}, we have
\begin{equation*}
    \begin{split}
        \abs{\da(P_i - P_j)(x_j)} &\leq \sum_{0\le l\le L}\lambda_j|\da(P_i^{(l)}-P_j^{(l)})(x_j)|\leq \eps\abs{x_i - x_j}^\ma
    \end{split}
\end{equation*}
for $\abs{\alpha} \leq m$, $0 \leq i,j \leq \kb$.

\end{proof}

\begin{lemma}\label{lemma:open sets for Glaeser}
Suppose $1+(D+1)\overline{k}\le k^\sharp$. Let $x\in E_1$ and $P\in\mathrm{int}H(x)$. There exists $\eps_0>0$ such that when $0<\eps<\eps_0$, there exist $\delta,\eta>0$ such that whenever
$x'\in E_1$, $P'\in W_H(x')$ and
\begin{equation}
    \dist((x,P),(x,P'))<\eta,
\end{equation}
we have $P'\in H(x')$ and
\begin{equation}
    \Delta(x',P';\kb,\eps)\ge \delta.
\end{equation}
\end{lemma}

\begin{proof}
Note that if $d=0$, then $H(x)$ consists of a single polynomial for any $x\in E_1$ and the conclusion follows from Lemma \ref{lemma:8.1}.

Suppose $d\ge 1$ and let $P_0,...,P_d\in\text{int}H(x)$ which form a nondegenerate affine simplex with $P$ in its relative interior. Let $\eta_0>0$ be as in the conclusion of Lemma \ref{lemma:linear perturbation}. Noting that $\eta_0$ is dependent solely on $P$, set $\eps_0=\eta_0$.

For each $P_j$, let $\delta_j>0$ be as in the conclusion of Lemma \ref{lemma:8.1} and take $\delta=\frac{1}{2}\min\{\delta_0,...,\delta_d,1\}$. Now take $\eta=\min\{\delta,\eta_0\}$.

Now, suppose $x'\in E_1\cap B(x,\delta)$, $P'\in H(x')$ and $\dist((x,P),(x',P'))<\eta$. Then,
\begin{equation}\label{eq:x x ' close}
    |x-x'|\le \delta
\end{equation}
and
\begin{equation}\label{eq:inequality with eta}
    |\da(P-P')(x)|,|\da(P-P')(x')|<\eta\text{ for }|\alpha|\le m
\end{equation}

By \eqref{eq:x x ' close} and Lemma \ref{lemma:8.1}, there exist $P_0',...,P_d'\in H(x')$ such that
\begin{equation}
    |\da(P_j-P_j')(x_0)|,|\da(P_j-P_j')(x')|\le\eps|x'-x_0|^{m-|\alpha|}\le \eta_0\text{ for }|\alpha|\le m
\end{equation}
and
\begin{equation}
    \Delta(x',P_j';\kb,\eps)\ge\delta, 0\le j\le d.
\end{equation}

By Lemma \ref{lemma:linear perturbation}, the definition of lowest stratum, and \eqref{eq:inequality with eta}, $P'$ lies in the convex hull of $P_0',...,P_d'$. Thus, by Lemma \ref{lemma:delta on lines},
\begin{equation}
    \Delta(x',P';\kb,\eps)\ge \delta.
\end{equation}

\end{proof}

\begin{lemma}\label{lemma:existence of nearby P}
Let $x\in E_1$ and $P\in\mathrm{int}H(x)$. If $\eps>0$, then there exists $\tilde{\eta}>0$ such that whenever $x'\in E_1\cap B(x,\tilde{\eta})$, there exists $P'\in H(x')$ such that
\begin{equation}
    \dist((x,P),(x',P'))<\eps.
\end{equation}
\end{lemma}

\begin{proof}
By the definition of Gleaser stability, there exists $0<\tilde{\eta}<\min\{1,\eps/2\}$ such that whenever $x'\in E_1\cap B(x,\tilde{\eta})$, there exists $P'\in H(x')$ such that
\begin{equation}\label{eq:blank space}
    |\da(P-P')(x)|,|\da(P-P')(x')|\le\eps/2|x-x'|^{m-|\alpha|}\le\eps/2\text{ for }|\alpha|\le m.
\end{equation}

Combining \eqref{eq:blank space} with the fact $|x-x'|<\tilde{\eta}<\eps/2$ we obtain the desired conclusion.
\end{proof}

\begin{lemma}\label{lemma:uniform delta}
Suppose $1+(D+1)\overline{k}\le k^\sharp$. Let $\eps>0$ and $K_0$ be a compact subset of $\R^n\times\p$ such that whenever $(x,P)\in K_0$, $P\in \mathrm{int} H(x)$. Then, there exists $\delta>0$ such that whenever $(x_0,P_0)\in K_0$ and $x_1,...,x_k\in E_1\cap B(x_0,\delta)$, there exists $P_1,...,P_{\kb}\in H(x)$ such that
\begin{equation}
    |\da(P_i-P_j)(x_j)|\le\eps|x_i-x_j|^{m-|\alpha|}\text{ for }|\alpha|\le m, 0\le i,j,\le \kb.
\end{equation}
\end{lemma}

\begin{proof}
Let $(x,P)\in K_0$. Then, since $P$ lies in the relative interior of $H(x)$, we may apply Lemma \ref{lemma:open sets for Glaeser} to obtain $\eta_x,\delta_x,(\eps_0)_x>0$ such that whenever $(x',P')$ lies in
\begin{equation}
    B_{K_0}((x,P),\eta_x):\{(x',P')\in K_0:\dist((x,P),(x',P'))<\eta_x\},
\end{equation}
and $0<\eps<(\eps_0)_x$ we have
\begin{equation}
    \Delta(x',P';\kb,\eps)\ge \delta_x.
\end{equation}

Thus, $\{B_{K_0}((x,P),\eta_x)\}_{(x,P)\in K}$ forms an open cover of $K_0$. By compactness of $K_0$, we obtain a finite subcover $\{B_{K_0}((x_1,P_1),\eta_{x_1}),...,B_{K_0}((x_l,P_l),\eta_{x_l})\}$. Taking $\delta=\min\{\delta_{x_1},...,\delta_{x_l}\}>0$ and requiring $0<\eps<\min\{(\eps_0)_{x_1},...,(\eps_0)_{x_l}\}$, we see that
\begin{equation}
    \Delta(x',P';\kb,\eps)\ge \delta
\end{equation}
for all $(x',P')\in K_0$, which is equivalent to our desired conclusion.

For $\eps\ge\min\{(\eps_0)_{x_1},...,(\eps_0)_{x_l}\},$ apply the conclusion for $\frac{1}{2}\min\{(\eps_0)_{x_1},...,(\eps_0)_{x_l}\}$ and observe
\begin{equation}
    \Delta(x',P';\kb,\eps)\ge \Delta\left(x',P';\kb,\frac{1}{2}\min\{(\eps_0)_{x_1},...,(\eps_0)_{x_l}\}\right)\ge \delta.
\end{equation}
\end{proof}

We now consider the case where $(H(x))_{x\in E}$ is the shape field $(\Gamma(x,M))_{x\in E}$. Let $x\in E_1$ and $P\in\text{int}\G(x,\ksh,2A^\sharp)$.

Let $\eps_0=(\eps_0)_{(x,P)}$ be as in Lemma \ref{lemma:open sets for Glaeser} applied to the Glaeser stable bundle $(\G(x,2A^\sharp))_{x\in E}$. (See Lemma \ref{lemma:GS for Gamma(x,A)}.) Let $\eta_{(x,P)}$ and $\delta_{(x,P)}$ be the corresponding $\eta$ and $\delta$ from the conclusion of Lemma \ref{lemma:open sets for Glaeser} with $\eps$ taken as $(\eps_0)_{(x,P)}/2$.

By Lemma \ref{lemma:existence of nearby P}, there exists $0<\tilde{\eta}_{(x,P)}<\eta_{(x,P)}$ such that whenever $x'\in E_1\cap B(x,\tilde{\eta})$, there exists $P'\in \G(x',A^\sharp)$ such that
\begin{equation}
    \dist((x,P),(x',P'))<\eta_{(x,P)}/4.
\end{equation}

Furthermore, by Lemma \ref{lemma:5.8 equivalent}, we may take $\eta_{(x,P)}$ and $\tilde{\eta}_{(x,P)}$ small enough that if $x'\in E_1$ and $P\in \Gamma(x')$ satisfies
\begin{equation}
    \dist((x,P),(x',P'))\le\eta_{(x,P)},
\end{equation}
then $P'\in\Gamma(x',\tilde{k},CA^\sharp)$, where $C$ is as in the conclusion of Lemma \ref{lemma:5.8 equivalent} and $\tilde{k}$ satisfying
\begin{equation}\label{eq:k tilde}
    1 + (D+1)\tilde{k} \leq \ksh.
\end{equation}

Consider the collection of balls $B(x,\tilde{\eta}_{(x,P)}/4)\subset\R^n$, with $x$ ranging over $E_1$ and given $x$, $P$ ranging over $\text{int}\Gamma(x,\ksh,2A^\sharp)$. This forms an open cover of $E_1$. $E_1$ is compact, so consider a finite subcover
\begin{equation}
    \left\{B(x_1,\tilde{\eta}_{(x_1,P_1)}/4),...,B(x_J,\tilde{\eta}_{(x_J,P_J)}/4)\right\}.
\end{equation}

Define
\begin{equation}
    K_j=\{(x,P)\in \overline{B}((x_j,P_j),{\eta}_{(x_j,P_j)}/2):x\in E_1\cap\overline{B}(x,\tilde{\eta}/2), P\in \G(x,2A^\sharp)\},
\end{equation}
where $\overline{B}$ is used to refer to the closed ball, and
\begin{equation}
    K=\bigcup_{j=1}^J K_j\subset E_1\times\P.
\end{equation}

Note that Lemma \ref{lemma:5.8 equivalent} implies that
\begin{equation}\label{eq:implication}
    \text{if }(x,P)\in K,\text{ then }P\in \G(x,\tilde{k},CA^\sharp).
\end{equation}
with $C$ depending only on $m,n$.

\begin{lemma}\label{lemma:K is compact}
$K$ is compact under the topology induced by the $\dist$ metric defined in \eqref{eq:define distance}.
\end{lemma}

\begin{proof}
It suffices to show $K_j$ is compact for arbitrary $1\le j\le J$. Let $1\leq j_0\le J$. By Lemma \ref{lemma:for Heine-Borel}, it suffices to show $K_{j_0}$ is closed and bounded; boundedness follows immediately from \ref{regularity 1}.

To show $K_{j_0}$ is closed, let $x\in\R^n,P\in\p$, and $(x_1,P_1),(x_2,P_2),...\in K_{j_0}$ such that
\begin{equation}
    |x_i-x|,\dist((x_i,P_i),(x,P))\to 0\text{ as }i\to\infty.
\end{equation}
It follows that $(x,P)\in \overline{B}((x_{j_0},P_{j_0}),{\eta}_{(x_{j_0},P_{j_0})}/2)$ and $x\in \overline{B}(x_{j_0},\tilde{\eta}_{(x_{j_0},P_{j_0})}/2)$ since the same is true for each $(x_i,P_i)$.

By compactness of $E_1$ established by Lemma \ref{lem:low-strat}, we have $x\in E_1$. Our goal is to show $P\in\G(x,2A^\sharp)$. By Lemma \ref{lemma:open sets for Glaeser} and the definition of $K$, to show $P\in\Gamma(x,2A^\sharp)$, it suffices to show $P\in W_{\G}(x)$.

By the compactness of the Grassmanian of $d$-dimensional hyperplanes of $\p$ having nontrivial intersection with $\{|\da P(y)|\le 2A^\sharp\text{ for }|\alpha|\le m\}$ for some $y\in E$, we may suppose by passing to a subsequence that $W_{\G}(x_i)$ converges to some $d$-dimensional subspace $W\subset\p$.

If $W=W(x)$ we are done, so suppose for the sake of contradiction that $W\neq W_{\G}(x)$. Then there exists $\tilde{P}\in \Gamma(x,2A^\sharp)$ such that $\tilde{P}\notin W_{\G}(x)$, hence there cannot exist a sequence $P_i\in W_{\G}(x_i)$ satisfying $P_i\to \tilde{P}$. This contradicts the Glaeser-stable property of $(\Gamma(x,2A^{\sharp}))_{x\in E}$ established in Lemma \ref{lemma:GS for Gamma(x,A)}. Thus, $W=W_{\G}(x)$.
\end{proof}





\begin{lemma}\label{lemma:keep in K}
Let $\tilde{k}$ be as in \eqref{eq:k tilde}. There exists $\gamma>0$ such that:

For all $x\in E_1$, there exists $P\in \G(x,\tilde{k},CA^\sharp)$ such that $(x,P)\in K$ and if $x'\in E_1,P'\in \G(x',2A^\sharp)$ such that $\dist((x,P),(x',P'))<\gamma$, then $(x',P')\in K$.
\end{lemma}

\begin{proof}
Returning to the notation of the construction of $K$, choose $1\le j\le J$ such that $x\in B(x_j,\tilde{\eta}_{(x_j,P_j)}/4)$. By Lemma \ref{lemma:existence of nearby P} and the construction of $K$, there exists $P\in \G(x,2A^\sharp)$ such that
\begin{equation}
    \dist((x_j,P_j),(x,P))<\eta_{(x_j,P_j)}/4,
\end{equation}
thus $(x,P)\in K$ and by \eqref{eq:implication}, $P\in\G(x,\tilde{k},CA^\sharp)$.

Now suppose $x'\in E_1\cap B(x,\tilde{\eta}_{(x_j,P_j)}/4),P'\in \G(x')$ such that $\dist((x,P),(x',P'))<{\eta}_{(x_j,P_j)}/2$. Then by the triangle inequality, we have
\begin{equation}
    |x_j-x'|<\tilde{\eta}_{(x_j,P_j)}
\end{equation}
and
\begin{equation}
    \dist((x_j,P_j),(x',P'))<\eta_{(x_j,P_j)}/2.
\end{equation}

Thus, by definition, $(x',P')\in K_j$ and the conclusion follows setting $\gamma=\min_{1\le j\le J}\tilde{\eta}_{(x_j,P_j)}/4$.





\end{proof}

\begin{lemma}\label{lemma:modulus of continuity for 2 points}

There exists $0<\delta_0<1$ and a regular modulus of continuity ${\omega}$ for which the following holds:

Given $x,x'\in E_1$ with $|x-x'|\le \delta_0$ and given $(x,P)\in K$, there exists $P'\in\G(x',2A^\sharp)$ such that 

\begin{equation}
    |\d^\alpha(P'-P)(x)|\le{\omega}(|x-x'|)\cdot|x-x'|^{m-|\alpha|}\text{ for }|\alpha|\le m.
\end{equation}
\end{lemma}

\begin{proof}

\renewcommand{\pbar}{\overline{P}}

We mimic the proof of Lemma 6.4 in \cite{F06}.

For $\nu = 0,1,2,\cdots$, we set $\eps_\nu := 2^{-\nu}$. By Lemma \ref{lemma:K is compact}, $K$ is compact and we may apply Lemma \ref{lemma:uniform delta} to successively pick $\delta_0, \delta_1, \delta_2, \cdots$, such that the following hold.
\begin{enumerate}[label=($\delta$-\arabic*)]
    \item $\delta_0 = 1$.
    \item $0 < \delta_{\nu+1} < \frac{1}{2}\delta_\nu$.
    \item If $\nu \geq 1$, then given $x, x' \in E_1$ with $\abs{x - x'} \leq \delta_\nu$, and given $(x,P)\in K$, there exists $P' \in \G(x',2A^\sharp)$ with
    \begin{equation*}
        \abs{\da(P' -P)(x)} \leq \frac{1}{2}\eps_\nu \abs{x' - x}^\ma
        \text{ for } \abs{\alpha} \leq m.
    \end{equation*}
\end{enumerate}

We define a regular modulus of continuity ${\omega}(t)$ on $[0,1]$ by setting
\begin{equation*}
    {\omega}(0) = 0,\,
    {\omega}(\delta_\nu) = \eps_\nu,\,
    \text{and }
    {\omega}(t) \text{ linear on each }[\delta_{\nu+1},\delta_\nu]
    \text{ for }\nu \geq 0.
\end{equation*}

Now suppose $x, x' \in E_1$. 

Suppose $x = x'$, we may simply take $P = P'$. 

Suppose $x \neq x'$. Furthermore, assume that $0 < \abs{x - x'} < \delta_1$. Let $(x,P)\in K$. Pick $\nu \geq 1$ such that $\delta_{\nu+1} < \abs{x - x'} \leq \delta_\nu$. By ($\delta$-3), there exists $P' \in \G(x',2A^\sharp)$ such that
\begin{equation}
    \abs{\da(P' - P)(x)} \leq \frac{1}{2}\eps_\nu\abs{x'-x}^{\ma}
    \text{ for }\abs{\alpha} \leq m.
    \label{eq.10.2.1}
\end{equation}

On the other hand, since $\abs{x'- x} > \delta_{\nu+1}$, we have
\begin{equation}
{\omega}(\abs{x'-x}) \geq {\omega}(\delta_{\nu+1}) = \eps_{\nu+1} = \frac{1}{2}\eps_\nu.
    \label{eq.10.2.2}
\end{equation}

Combining \eqref{eq.10.2.1} and \eqref{eq.10.2.2}, we have
\begin{equation*}
    \abs{\da(P' - P)(x)} \leq {\omega}(\abs{x' - x})\abs{x - x'}^{\ma}
    \text{ for }
    \abs{\alpha} \leq m.
\end{equation*}
\end{proof}

\begin{lemma}[Existence of Modulus of Continuity]\label{lemma:modulus of continuity}

Suppose
\begin{equation}
    \ksh\ge D+2, 1+(D+1)\kt\le k^\sharp
\end{equation}
and let $\omega$ be as in Lemma \ref{lemma:modulus of continuity for 2 points}. Then, given any $\kb\ge1$, there exists a controlled constant $\hat{C}_{\kb}$ and $\delta'>0$ such that the following holds:

Let $x_0\in S\subset E_1$ with $\diam(S)\le \delta'$ and $|S|\le\overline{k}$. Then,
there exists a map $x\to P^x$ from $S$ into $\p$ such that
\begin{equation}\label{eq:something about being in K}
    P^x\in\G(x,\kt,CA^\sharp), (x,P^x)\in K \text{ for each }x\in S; \text{ and}
\end{equation}
\begin{equation}
    |\D^\alpha(P^x-P^y)(y)|\le \hat{C}_{\kb}\omega(|x-y|)|x-y|^{m-|\alpha|}\text{ for }x,y\in S, |x-y|\le \delta', |\alpha|\le m.
\end{equation}
\end{lemma}

\begin{proof}
In what follows, we first construct the mapping $x\mapsto P^x$ via Lemma \ref{lemma:clustering} and repeated applications of Lemma \ref{lemma:modulus of continuity for 2 points} as in \cite{F06}. Then, we will show that each step of the construction is valid in the sense that it produces $P^x$ such that $(x,P^x)\in K$, justifying the uses of Lemma \ref{lemma:modulus of continuity for 2 points} and satisfying \ref{eq:something about being in K}.

The proof is by induction on $\kb$, the base case $\kb=1$ following from simply taking $P^{x_0}=P_0$ as guaranteed by Lemma \ref{lemma:keep in K}.


Now suppose Lemma \ref{lemma:modulus of continuity} holds for all positive integers less than $\kb$. 

First require that $\delta_1<\delta_0$, where $\delta_0$ is as in Lemma \ref{lemma:modulus of continuity for 2 points}.

By Lemma \ref{lemma:clustering}, we may partition $S$ into sets $S_0,...,S_M$ such that
\begin{equation}
    \#(S_{l})\le \kb-1\text{ for each }l, (0\le l\le M), \text{ and}
\end{equation}
\begin{equation}\label{eq:use of clustering}
    \dist(S_{l},S_{l'})>c_{\kb}\cdot\diam(S) \text{ for }l\ne l'.
\end{equation}
Without loss of generality, take $x_0=x_0\in S_0$ and suppose each $S_l$ is nonempty, fixing $x_{l}\in S_{l}$ ($l\ge1$).


Again referring to Lemma \ref{lemma:keep in K}, choose $\gamma>0$ and $(x_0,P_0)\in K$ such that if $x'\in E_1,P'\in \G(x',2A^\sharp)$ such that $\dist((x_0,P_0),(x',P'))<\gamma$, then $(x',P')\in K$.

By Lemma \ref{lemma:modulus of continuity for 2 points}, there exist $P_{l_1}\in\G(x_{l_1},2A^\sharp)$ for $1\le l_1\le M_1$ such that
\begin{equation}
    |\d^\alpha(P_0-P_{l}(x_0)|<\omega(|x_0-x_{l}|)|x_0-x_{l}|^{m-|\alpha|}\text{ for }|\alpha|\le m, 1\le l\le M.
\end{equation}

Letting $\delta=\diam(S)$, we have in particular that
\begin{equation}
    |\d^\alpha(P_0-P_{l}(x_0)|<\omega(|x_0-x_{l}|)\delta^{m-|\alpha|}\text{ for }|\alpha|\le m, 1\le l\le M.
\end{equation}

For each $0\le l\le M$, we apply the induction hypothesis to each $x_l,S_l$ in place of $x_0,S$. Thus we obtain a map $x\mapsto P^x$ such that
\begin{equation}
    P^{x_l}=P_l,
\end{equation}
\begin{equation}
    (x,P^x)\in K,
\end{equation}
and
\begin{equation}\label{eq:sprouted P close}
    |\d^\alpha(P^x-P^y)(y)|\le \hat{C}_{\kb-1}\omega(|x-y|)|x-y|^{m-|\alpha|}\text{ for }|\alpha|\le m
\end{equation}
whenever $x,y\in S_l$ (same $l$).

It remains to show that \eqref{eq:sprouted P close} holds in the case where $x\in S_l, y\in S_{l'}$ ($l\ne l'$).

From \eqref{eq:sprouted P close} in the case already established, we have
\begin{equation}\label{eq:IH for x}
    |\d^\alpha(P^x-P_l)(x)|\le \hat{C}_{\kb-1}\omega(\delta)\delta^{m-|\alpha|}\text{ for }|\alpha|\le m
\end{equation}
and
\begin{equation}\label{eq:IH for y}
    |\d^\alpha(P^y-P_{l'})(y)|\le \hat{C}_{\kb-1}\omega(\delta)\delta^{m-|\alpha|}\text{ for }|\alpha|\le m.
\end{equation}

By \eqref{eq:IH for x} and the fact $|x-y|\le \delta$, we have
\begin{equation}\label{eq:comp Px Pl}
    |\d^\alpha(P^x-P_l)(y)|\le C''\hat{C}_{\kb-1}\omega(\delta)\delta^{m-|\alpha|} \text{ for }|\alpha|\le m
\end{equation}
and by similar reasoning
\begin{equation}\label{eq:comp Pl Pl'}
    |\d^\alpha(P_l-P_{l'})(y)|\le C''\omega(\delta)\delta^{m-|\alpha|} \text{ for }|\alpha|\le m.
\end{equation}

Summing \eqref{eq:IH for y}, \eqref{eq:comp Px Pl}, and \eqref{eq:comp Pl Pl'} gives
\begin{equation}
    |\d^\alpha(P^x-P^y)(y)|\le C'''\hat{C}_{\kb-1}\omega(\delta)\delta^{m-|\alpha|}\text{ for }|\alpha|\le m.
\end{equation}

Observing that \eqref{eq:use of clustering} implies $|x-y|\ge c_{\kb}\delta$ and therefore $\omega(|x-y|)\ge\omega(c_{\kb}\delta)\ge c_{\kb}\omega(\delta)$, we have
\begin{equation}
    |\d^\alpha(P^x-P^y)(y)|\le \tilde{C}\hat{C}_{\kb-1}\omega(|x-y|)|x-y|^{m-|\alpha|}\text{ for }|\alpha|\le m.
\end{equation}

In the above, we have implicitly described a recursive method for determining $P^x$ from repeated applications of Lemma \ref{lemma:modulus of continuity for 2 points} giving us $P^x\in\G(x,2A^\sharp)$ for $x\in S$. Therefore, we must check that at each step we have produced $(x,P^x)\in K$ in order to satisfy the hypotheses of Lemma \ref{lemma:modulus of continuity for 2 points}.

At each stage, we have the inequality
\begin{equation}\label{eq:end of modulus proof}
    |\d^\alpha(P^x-P^{x_0})(x_0)|<C_{\kb}\omega(|x-x_0|)|x-x_0|^{m-|\alpha|}\text{ for }|\alpha|\le m.
\end{equation}

By Lemma \ref{lemma:keep in K}, we may establish that $(x,P^x)\in K$ by choosing $\delta'$ small enough such that \eqref{eq:end of modulus proof} and $|x-x_0|<\delta'$ implies $\dist((x,P^x),(x_0,P_0))<\gamma$. This is allowed since the constant $C_{\kb}$ is dependent only on $\kb$. Since $(x,P^x)\in K$, $P^x\in \G(x,\tilde{k},CA^\sharp)$, so our map has the desired properties.
\end{proof}

\section{Main Proof}\label{sec:main proof}






In this section, we complete the proof of Theorem \ref{thm:heart of the matter} in the case $d=1$.

Let $\Gamma'(x,M)$ now refer to the unrefined shape field in the hypothesis of Theorem \ref{thm:heart of the matter}; we will use the hypothesis that $\Gamma'(x,M)$ is closed for fixed $x,M$.

Given $x_0\in E_1,\kb\in\N, A>0$, define $\Gamma'(x_0,\kb,A)\subset\Gamma'(x_0,A)$ to be the closure of the set of $P_0\in\Gamma'(x_0,A)$ such that for all $x_1,...,x_{\kb}\in E$, there exist $P_j\in\Gamma(x_j,A)$ (\textit{not} $\Gamma'(x,A)$) such that
\begin{equation}\label{eq:Pi close again}
    |\d^\alpha(P_i-P_j)(x_j)|\le A|x_i-x_j|^{m-|\alpha|}\text{ for }|\alpha|\le m, 0\le i,j\le \kb.
\end{equation}

By definition, $(\Gamma'(x,\kb,M))_{x\in E_1,M\geq 0}$ is closed, is convex, and satisfies $\Gamma(x,\kb,M)\subset\Gamma'(x,\kb,M)$ for all $M>1, x\in E_1,\kb\in\N$.

First, we show that a property similar to that described in \eqref{eq:Pi close again} applies to all $P_0\in \G'(x_0,\kb,M)$ (even after the closure is taken). Then, we will show that the shape fields finiteness principle (Theorem \ref{thm:shape fields fin prin}) applies to the $\G'(x_0,\kb,M)$.

\begin{lemma}\label{lemma:9.1}
Let $X_0\in E_1, M\geq 0$. If $P_0\in\Gamma'(x_0,\overline{k},M)$ and $x_1,...,x_{\kb}$, then there exist $P_j\in\Gamma(x_j,M)\subset\Gamma(x_j,2M)$ such that
\begin{equation}
    |\d^\alpha(P_i-P_j)(x_j)|\le 2M|x_i-x_j|^{m-|\alpha|}\text{ for }|\alpha|\le m, 0\le i,j\le \kb.
\end{equation}
\end{lemma}

\begin{proof}
Fix $x_0\in E_1$ and $M\geq 0$. 

Let $P_0\in\Gamma'(x_0,\overline{k},M)$ and $x_1,...,x_{\kb}\in E_1$. If there do not exist $P_j\in\Gamma(x_j,M)$ such that
\begin{equation}
    |\d^\alpha(P_i-P_j)(x_j)|\le M|x_i-x_j|^{m-|\alpha|}\text{ for }|\alpha|\le m, 0\le i,j\le \kb,
\end{equation}
then there exist $P^{(1)},P^{(2)},...\in \Gamma'(x,\kb,M)$ such that $P^{(\nu)}\to P_0$ and $P_j^{(\nu)}\in\Gamma(x_j,M)$ ($\nu=1,2,3...$) such that
\begin{equation}\label{eq:just need random name}
    |\d^\alpha(P^{(\nu)}_i-P^{(\nu)}_j)(x_j)|\le M|x_i-x_j|^{m-|\alpha|}\text{ for }|\alpha|\le m, 0\le i,j\le \kb.
\end{equation}

By taking $\nu$ sufficiently large and applying the triangle inequality, we find $N\in\N$ such that
\begin{equation}\label{eq:just need random name 2}
    |\d^\alpha(P_0-P^{(N)}_j)(x_j)|\le 2M|x_0-x_j|^{m-|\alpha|}\text{ for }|\alpha|\le m, 0\le j\le \kb.
\end{equation}

Combining \eqref{eq:just need random name} and \eqref{eq:just need random name 2} gives the desired conclusion.

\end{proof}

\begin{lemma}\label{lemma:9.2}
$(\Gamma'(x,\kb,M))_{x\in E_1,M\geq 0}$ is $(C,1)$-Whitney convex, where $C$ depends only on $C_w, m, n$.
\end{lemma}
\begin{proof}
Let $0 < \delta\leq 1$, $M\geq 0$, and $x_0\in E_1$. Choose $P^{(0)}_1, P^{(0)}_2\in\G'(x_0,\kb,M)$ which satisfy
\begin{equation}\label{eq:C-delta 2'}
    \abs{\da (P^{(0)}_1 - P^{(0)}_2)(x_0)} \leq M\delta^{m-\abs{\alpha}} \text{ for } \abs{\alpha} \leq m;
\end{equation}	
\begin{equation}\label{eq:C-delta 3'}
    \abs{\da Q_i(x_0)} \leq \delta^{-\abs{\alpha}} \text{ for } \abs{\alpha} \leq m, i = 1,2; \text{ and}
\end{equation}	
\begin{equation}\label{eq:C-delta 4'}
    Q_1 \odot_{x_0} Q_1 + Q_2 \odot_{x_0} Q_2 = 1
\end{equation}	

We want to show that
\begin{equation}
    P^{(0)}_0:=\sum_{l = 1,2}Q_l\odot_{x_0} Q_l\odot_{x_0} P^{(0)}_l \in \G'(x_0,\kb,CM).
    \label{eq.6.4.1}
\end{equation}

By \eqref{eq:C-delta 3'} and \eqref{eq:C-delta 4'}, there exists $\theta_1,\theta_2=1-\theta_1\in C^m(\R^n)$ such that
\begin{equation}
    0\le \theta_l(x)\le 1, x\in \R^n, l=1,2;
\end{equation}
\begin{equation}
    J_{x_0}\theta_l=Q_l^2, l=1,2;
\end{equation}
\begin{equation}\label{eq:derivatives bound on theta_l}
    |\da\theta_l(x)|\le C\delta^{-|\alpha|}\text{ for }|\alpha|\le m, x\in\R^n, l=1,2\text{; and}
\end{equation}
\begin{equation}
    \text{supp}(\theta_1)\subset B(x_0,\delta).
\end{equation}

Let $x_1,...,x_{\kb}\in E$. Since $P^{(0)}_1,P^{(0)}_2\in\G'(x_0,\kb,A)$, for $l=1,2$ there exist $P_l^{(i)}\in\Gamma(x_j,2M)$ ($1\le j\le\kb$) such that
\begin{equation}\label{eq:another thing with P_l close}
    |\d^\alpha(P^{(i)}_l-P^{(j)}_l)(x_j)|\le 2M|x_i-x_j|^{m-|\alpha|}\text{ for }|\alpha|\le m, 0\le i,j\le \kb.
\end{equation}

Thus by the classical Whitney extension theorem for finite sets (Lemma \ref{lemma:WET}), for $l=1,2$ there exist $F_l\in C^m(\R^n)$ such that
\begin{equation}\label{eq:norm bound on F_l}
    \|F_l\|_{C^m(\R^n)}\le CM, l=1,2\text{; and}
\end{equation}
\begin{equation}
    J_{x_i}F_l=P_l^{(i)}, 0\le i\le \kb, l=1,2.
\end{equation}

Let $F=\theta_1F_1+\theta_2F_2$, so that \eqref{eq:C-delta 2'}, \eqref{eq:derivatives bound on theta_l}, and \eqref{eq:norm bound on F_l} imply
\begin{equation}\label{eq:norm bound on F}
    \|F\|_{C^m(\R^n}\le CM.
\end{equation}

Now for $1\le i \le \kb$ define
\begin{equation}
    P_0^{(i)}:=J_{x_i}(F)=J_{x_i}(\theta_1F_1+\theta_2F_2)=J_{x_i}\theta_1\odot_{x_i}P_1^{(i)}+J_{x_i}\theta_2\odot_{x_i}P_2^{(i)}.
\end{equation}

If $x_i\notin B(x_0,\delta)$, then $J_{x_i}\theta_1=0$, so $P_0^{(i)}=P_2^{(i)}\in\Gamma(x_i,2M)$. If $x_i\in B(x_0,\delta)$, then \eqref{eq:C-delta 2'} and \eqref{eq:another thing with P_l close} imply
\begin{equation}
    |\da(P_1^{(i)}-P_2^{(i)})(x_i)|\le CM\delta^{m-|\alpha|}\text{ for }|\alpha|\le m,
\end{equation}
so $P_0^{(i)}\in\G(x_i,CM)$ by the $(C_w,1)$ convexity of $(\G(x,M))_{M\geq 0,x\in E}$.

By \eqref{eq:norm bound on F}, we see
\begin{equation}
    |\d^\alpha(P^{(i)}_0-P^{(j)}_0)(x_j)|\le CM|x_i-x_j|^{m-|\alpha|}\text{ for }|\alpha|\le m, 0\le i,j\le \kb,
\end{equation}
establishing \eqref{eq.6.4.1}.

\end{proof}

Recall that our proof of Theorem \ref{thm:heart of the matter} is by induction on the number of strata. The following lemma will address the base case.

\begin{lemma}\label{lemma:Step 1 complete}
Let $A^\sharp$ be as in Lemma \ref{lemma:finiteness} and $\tilde{k}\ge \ksh_{SF}$ satisfy \eqref{eq:k tilde}. Then, there exists $\tilde{F}\in \dot{C}^{m}(\R^n,\R^d)$ such that
\begin{equation}\label{eq:first quantitative bound}
    \|\tilde{F}\|_{\dot{C}^m(\R^n,\R^d)}\le CA^\sharp
\end{equation}
and
\begin{equation}\label{eq:control over correction}
    J_x(\tilde{F})\in \Gamma(x,\tilde{k},CA^\sharp) \text{ for all }x\in E_1,
\end{equation}
where $C$ depends solely on $m,n$.
\end{lemma}

\begin{proof}

By Lemma \ref{lemma:modulus of continuity}, there exists a modulus of continuity $\omega$ such that the following holds. If $S\subset E_1$ and $|S|\le\kt$, then there exists a map $x\mapsto P^x$ satisfying

\begin{equation}
    P^x\in \Gamma(x,\kt,CA^\sharp) \text{ for each }x\in S; \text{ and}
\end{equation}
\begin{equation}
    |\D^\alpha(P^x-P^y)(y)|\le C\omega(|x-y|)|x-y|^{m-|\alpha|}\text{ for }x,y\in S, |x-y|\le 1, |\alpha|\le m.
\end{equation}

In particular, by the classical Whitney extension theorem for finite sets (Lemma \ref{lemma:WET}) there exists $F^S$ such that

\begin{equation}
    \|F^S\|_{\dot{C}^{m,\omega}(\R^n,\R^d)}\le CA^\sharp
\end{equation}
and
\begin{equation}
    J_xF^S\in \Gamma(x,\kb,CA^\sharp)\text{ for all }x\in S.
\end{equation}

By Lemmas \ref{lemma:9.1} and \ref{lemma:9.2}  and the fact that $\kt$ has been chosen to be greater than $k^\sharp_{SF}$, by Theorem \ref{thm:shape fields fin prin}, there exists $\tilde{F}\in \dot{C}^{m,\omega}(\R^n,\R^d)$ such that
\begin{equation}\label{eq:bound 1 from SFFP}
    \|\tilde{F}\|_{\dot{C}^{m,\omega}(\R^n,\R^d)}\le C'A^\sharp
\end{equation}
and
\begin{equation}\label{eq:bound 2 from SFFP}
    J_x(\tilde{F})\in \Gamma'(x,\kt,C'A^\sharp) \text{ for all }x\in E_1.
\end{equation}


The first conclusion \eqref{eq:first quantitative bound} follows from \eqref{eq:bound 1 from SFFP}.

To obtain \eqref{eq:control over correction}, we observe that if for $x_0\in E_1$, $J_{x_0}\tilde{F}\in \Gamma'(x_0,M)$, then $J_{x_0}\tilde{F}\in\Gamma(x_0,M)$ since $F\in C^m(\R^n)$ and $(\Gamma(x,M))_{x\in E}$ is obtained through multiple Glaeser refinements of $(\Gamma'(x,M))_{x\in E}$. Thus, \eqref{eq:control over correction} follows from \eqref{eq:bound 2 from SFFP} and our careful definition of the $\Gamma'(x,\kb,A)$.

\end{proof}

Now we move on to the induction step. Suppose we are given a compact set $E$ with a Glaeser stable bundle $(\G(x,M))_{x\in E,M\geq 0}$ which has $\Lambda$ associated strata, and that Theorem \ref{thm:heart of the matter} is known to hold in the case of $\Lambda-1$ strata with $\ksh=\ksh_{\text{old}}$.

Given a cube $Q$, let $rQ$ denote the $r$ times concentric dilation of $Q$. $Q^*$ will be used to denote $1.01 Q$.

By Lemma \ref{lem:low-strat}, $E_1$ is compact; thus $\R^n\setminus E_1$ is open and admits a Whitney decomposition as in \cite{F06} and \cite{W34-1}. That is, there exist closed cubes $Q_\nu$ such that

\begin{subequations}\label{eq:Whitney decomposition}
\begin{equation}
    \delta_\nu:=diam(Q_\nu)\le1.
\end{equation}
\begin{equation}
    \R^n\setminus E_1=\bigcup_\nu Q_\nu.
\end{equation}
\begin{equation}
    3Q_\nu\cap E_1=\emptyset.
\end{equation}
\begin{equation}\label{eq:distance from boundary}
    \text{If }\delta_\nu<1,\text{ then there exists }x_0^{(\nu)}\in E_1\text{ such that }\dist(x_0^{(\nu)},Q_\nu)<C\delta_\nu.
\end{equation}
\begin{equation}
    \text{If the boundaries of }Q_\mu\text{ and }Q_\nu\text{ intersect, then }c\delta_\mu<\delta_\nu<C\delta_\mu.
\end{equation}
\end{subequations}


Furthermore, there exists a Whitney partition of unity $\{\theta_\nu\}_\nu$ subordinate to $\{Q_\nu\}_\nu$ satisfying

\begin{subequations}\label{eq:Whitney POU}
\begin{equation}
    1=\sum_\nu\theta_\nu\text{ on }\R^n\setminus E_1.
\end{equation}
\begin{equation}
    \supp(\theta_\nu)\subset Q^*_\nu.
\end{equation}
\begin{equation}
    |\D^\alpha\theta_\nu|\le C\delta_\nu^{-|\alpha|}\text{ on }\R^n\text{ for }|\alpha|\le m+1.
\end{equation}
\end{subequations}

Our current goal is to show that we may apply the induction hypothesis on each of the $E\cap Q_\nu$, but only after subtracting off the function $\tilde{F}$ from Lemma \ref{lemma:Step 1 complete}.

To apply the induction hypothesis, first we must establish Glaeser stability for the desired bundle. To do so, consider the bundle $(H(x))_{x\in E}$ such that
\begin{equation}
    H(x)=\{J_x\tilde{F}\}\text{ if }x\in E_1
\end{equation}
and
\begin{equation}
    H(x)=\Gamma(x,C_0A^\sharp)\text{ if }x\notin E_1,
\end{equation}
where $C=C_0$ is as in the conclusion of Lemma \ref{lemma:Step 1 complete}.

Since $E_1$ is compact and $J_x\tilde{F}\in \Gamma(x,C_0A^\sharp)$ for $x\in E_1$, we see that the Glaeser stability of $(\Gamma(x,C_0A^\sharp))_{x\in E}$ implies the Glaeser stability of $(H(x))_{x\in E}$. Furthermore, $\{(x_0,J_{x_0}\tilde{F}):x_0\in E_1\}$ is compact. Thus, by Lemma \ref{lemma:uniform delta}, we may choose $\delta$ to be uniform in $x_0$ in the following.

For any $\epsilon>0$, there exists $\delta>0$ for which the following holds: If $x_0\in E_1$ and $x_1,...,x_{\kt}\in E\cap B(x_0,\delta)$, then there exist
\begin{subequations}
\begin{equation}
    P_i\in \Gamma(x_i,C_0A^\sharp)
\end{equation}
such that
\begin{equation}
    P_0=J_{x_0}\tilde{F}
\end{equation}
and
\begin{equation}
    |\D^\alpha(P_i-P_j)(x_j)|\le \epsilon|x_i-x_j|^{m-|\alpha|}\text{ for }|\alpha|\le m, 0\le i,j,\le \kt.
\end{equation}
\end{subequations}

Subtracting off $J_{x_i}\tilde{F}$ from each of the above polynomials, we obtain the following:

If $x_0\in E_1$ and $x_1,...,x_{\kt}\in E\cap B(x_0,\delta)$, then there exist
\begin{subequations}
\begin{equation}\label{eq:for Glaeser stable 1}
    P_i\in -J_{x_i}\tilde{F}+\Gamma(x_i,C_0A^\sharp)
\end{equation}
such that
\begin{equation}
    P_0=0
\end{equation}
and
\begin{equation}\label{eq:for Glaeser stable 2}
    |\D^\alpha(P_i-P_j)(x_j)|\le \epsilon|x_i-x_j|^{m-|\alpha|}\text{ for }|\alpha|\le m, 0\le i,j,\le \kt.
\end{equation}
\end{subequations}

In particular, by\eqref{eq:distance from boundary}, for any $\epsilon>0$ there exists $\delta>0$ such that if $x_i\in E\cap Q^*_\nu$ and $\delta_\nu<\delta$, then
\begin{equation}\label{eq:epsilon control of polynomials}
    |D^\alpha(P_i)(x_i)|\le \epsilon\delta_\nu^{m-|\alpha|}.
\end{equation}

Second, we must determine bounds for the norm of the appropriate holding space. By \eqref{eq:control over correction}, for any $x_0\in E_1$ and $x_1,...,x_{\kt}\in E$, there exist

\begin{subequations}
\begin{equation}
P_i\in \Gamma(x_i,C_0A^\sharp)
\end{equation}
such that
\begin{equation}
P_0=J_{x_0}\tilde{F}
\end{equation}
and
\begin{equation}
    |\D^\alpha(P_i-P_j)(x_j)|\le C_0A^\sharp|x_i-x_j|^{m-|\alpha|}\text{ for }|\alpha|\le m, 0\le i,j,\le \kt.
\end{equation}
\end{subequations}

Again subtracting off $J_{x_i}\tilde{F}$ from each of the above polynomials, we obtain the following:

If $x_0\in E_1$ and $x_1,...,x_{\kt}\in E$, there exist
\begin{subequations}
\begin{equation}\label{eq:in Gamma C}
    P_i\in -J_{x_i}\tilde{F}+\Gamma(x_i,C_0A^\sharp)
\end{equation}
such that
\begin{equation}
    P_0=0
\end{equation}
and
\begin{equation}\label{eq:homog C control of polynomials}
|\D^\alpha(P_i-P_j)(x_j)|\le C_0A^\sharp|x_i-x_j|^{m-|\alpha|}\text{ for }|\alpha|\le m, 0\le i,j,\le \kt.
\end{equation}
\end{subequations}

In particular, if $x_i\in Q_\nu^*$ and $\delta_\nu<1$, then by \eqref{eq:distance from boundary}
\begin{equation}\label{eq:C control of polynomials}
    |\D^\alpha(P_i)(x_i)|\le C_0A^\sharp\delta_\nu^{m-|\alpha|}\text{ for }|\alpha|\le m
\end{equation}
when $x_1,...,x_{\kt}$ are taken from $E\cap Q_\nu$.

Observe that for each scale $\delta_\nu<1$ there are only finitely many $Q_\nu$ such that $\delta_\nu=\diam(Q_\nu)$. Thus, by \eqref{eq:for Glaeser stable 1}, \eqref{eq:for Glaeser stable 2}, \eqref{eq:epsilon control of polynomials}, \eqref{eq:in Gamma C}, \eqref{eq:homog C control of polynomials}, and \eqref{eq:C control of polynomials}, one may produce a function $A:(0,1]\to(0,C_0A^\sharp]$ such that $\lim_{t\downarrow0}A(t)=0$ and:

For any $Q_\nu$ satisfying $\delta_\nu<1$ and $x_1,...,x_{\kt}\in E\cap Q^*_\nu$, there exist
\begin{subequations}\label{eqs:norm of new SF}
\begin{equation}\label{eq:being in Gamma}
    P_i\in -J_{x_i}\tilde{F}+\Gamma(x_i,C_0A^\sharp)
\end{equation}
such that 
\begin{equation}\label{eq:control of polynomials norm}
    |\D^\alpha(P_i)(x_i)|\le A(\delta_\nu)\delta_\nu^{m-|\alpha|}\text{ for }|\alpha|\le m
\end{equation}
and
\begin{equation}
    |\D^\alpha(P_i-P_j)(x_j)|\le A(\delta_\nu)|x_i-x_j|^{m-|\alpha|}\text{ for }|\alpha|\le m, 0\le i,j,\le \kt.
\end{equation}
\end{subequations}



We will apply the following rescaled version of the induction hypothesis on each $E\cap Q_\nu^*$.

Before its formal statement, we define a shape field $(H(x,M))_{x\in E}$ to be \underline{$\delta$-regular} for $\delta>0$ if it satisfies \eqref{regularity 2} and \eqref{regularity 3}, but
\begin{equation}
    H(x,M)\subset\{P\in\p:|\da P(x)|\le \delta^{m-|\alpha|}M\text{ for }|\alpha|\le m\}
\end{equation}
in place of \eqref{regularity 1}. (The usual notion of being regular is equivalent to being 1-regular.)

\begin{lemma}\label{lemma:rescaled induction hypothsis}
Let $\tilde{\delta}>0$, $Q_0\subset\R^n$ be a cube of length $3\tilde{\delta}$ and $E\subset Q_0$ be compact. Suppose that for each $x\in E$ we are given a $(C_w,\tilde{\delta})$ convex, $\tilde{\delta}$-regular, closed shape field $(\Gamma(x,M))_{x\in E,M\geq 0}$ such that
\begin{enumerate}
    \item The Glaeser refinement of $(\Gamma(x))_{x\in E}$ terminates in a bundle $(\Gamma^*(x))_{x\in E}$ such that $\Gamma^*(x)$ is nonempty for all $x\in E$ and
    \item There exists $A>0$, such that given $x_1,...,x_{k^\sharp_\text{old}}\in E$, there exists polynomials $P_j\in\Gamma^*(x,A)$ satisfying
    \begin{equation}
        |\d^\alpha(P_i-P_j)(x_j)|\le A|x_i-x_j|^{m-|\alpha|}\text{ for }|\alpha|\le m, 1\le i,j,\le k^\sharp_\text{old}.
    \end{equation}
\end{enumerate}

Assume also that $E$ has fewer than $\lambda$ strata. Then there exists $F\in C^m(\R^n)$ such that
\begin{enumerate}
    \item $\|F\|_{\dot{C}^m}\le CA$ and
    \item $J_xF\in\Gamma(x,CA)$ for all $x\in E$,
\end{enumerate}
where $C$ depends solely on $m,n$.

\end{lemma}

Given $x\in E, M\geq 0$, define
\begin{equation}
    \tilde{\Gamma}(x,M):=(-J_x\tilde{F}+\G'(x,C_0A^\sharp M/A(\delta_\nu)))\cap\{P\in\p:|\da P(x)|\le \delta^{m-|\alpha|}M\text{ for }|\alpha|\le m\},
\end{equation}
where again, $(\G'(x,M))_{x\in E}$ refers to the unrefined bundle mentioned in Theorem \ref{thm:main theorem}.

Then, it is clear from the relevant definitions that $(\tilde{\G}(x,M))_{x\in E\cap Q_\nu^*, M\geq 0}$ is a $(C_w,\tilde{\delta})$ convex, $\tilde{\delta}$-regular, closed shape field.

Furthermore, the Glaeser refinement of $(\tilde{\G}(x))_{x\in E}$ terminates in $(\G^*(x))_{x\in E}$, the Glaeser stable bundle obtained through refinement of $(\G(x))_{x\in E}$. This is because
\begin{equation}
    \bigcup_{M\geq 0}\{P\in\p:|\da P(x)|\le \delta^{m-|\alpha|}M\text{ for }|\alpha|\le m\}=\p,
\end{equation}
so $\G(x)=\tilde{\Gamma}(x)$ for all $x\in E$. As a further corollary, we see that $\dim\G(x)=\dim\tilde{G}(x)$ and $\dim\pi_x\G(x)=\dim\pi_x\tilde{\Gamma}(x)$ for all $x\in E$, so both bundles have the same number of strata on $E\cap Q_\nu$.

Lastly, \eqref{eqs:norm of new SF} says that given $x_1,...,x_{k^\sharp_\text{old}}\in E$, there exist polynomials $P_j\in\Gamma^*(x,A(\delta_\nu))$ satisfying
    \begin{equation}
        |\d^\alpha(P_i-P_j)(x_j)|\le A(\delta_\nu)|x_i-x_j|^{m-|\alpha|}\text{ for }|\alpha|\le m, 1\le i,j,\le k^\sharp_\text{old}.
    \end{equation}

Thus, the hypotheses of Lemma \ref{lemma:rescaled induction hypothsis} are fully satisfied. We conclude that there exists $F_\nu\in C^m(\R^n)$ such that
\begin{equation}\label{eq:controlled homogeneous norm induction step}
    \|F_\nu\|_{\dot{C}^m(\R^n)}\le CA(\delta_\nu)
\end{equation}
and
\begin{equation}\label{eq:which SF it's in}
    J_x F_\nu\in \tilde{\Gamma}(x,CA(\delta_\nu)), x\in E\cap Q_\nu^*.
\end{equation}

By definition of $\tilde{\Gamma}(x,M)$, \eqref{eq:which SF it's in} implies

\begin{equation}
    J_x F_\nu\in-J_{x_i}\tilde{F}+\Gamma(x,CA^\sharp)
\end{equation}
and
\begin{equation}\label{eq:control on jets induction step 1}
    |\D^\alpha(J_xF_\nu)(x)|\le CA(\delta_\nu)^{m-|\alpha|}\text{ for }|\alpha|\le m, x\in E\cap Q_\nu^*.
\end{equation}

Combining \eqref{eq:controlled homogeneous norm induction step} with \eqref{eq:control on jets induction step 1} gives

\begin{equation}\label{eq:control on jets induction step 2}
    |\D^\alpha(J_xF_\nu)(x)|\le CA(\delta_\nu)^{m-|\alpha|}\text{ for }|\alpha|\le m, x\in Q_\nu^*.
\end{equation}

For $\delta>0$, define

\begin{equation}
    F^{[\delta]}(x)=\sum_{\delta_{\nu}>\delta}\theta_\nu(x)F_\nu(x).
\end{equation}

Following \cite{F06}, one may easily show, using the properties of the Whitney decomposition and associated partition of unity, along with
\eqref{eq:control on jets induction step 2}
that $F^{[\delta]}$ converges in $C^m$ to a function $F^{[0]}$ satisfying
\begin{equation}
    \|F^{[0]}\|_{C^m}\le CA^\sharp
\end{equation}
and
\begin{equation}
    J_x F^{[0]}=0, x\in E_1.
\end{equation}

For $x\in E\setminus E_1$, $x\in \supp(\theta_\nu)$ for finitely many $\nu$; thus, there exists a $\delta(x)>0$ such that $F^{[0]}\equiv F^{[\delta(x)]}$ in a neighborhood of $x$. As a result,
\begin{align*}
    J_xF^{[0]}&=J_xF^{[\delta(x)]}\\
    &=\sum_{supp(\theta_\nu)\ni x} J_x\theta_{\nu}\odot J_xF_\nu\\
    &\in -J_x\tilde{F} + \Gamma(x,CA^\sharp)
\end{align*}
by the $(C_w,\tilde{\delta})$-convexity of the shape field, specifically Lemma \ref{lemma:shape field convexity for more polynomials}.

We now set $F=F^{[0]}+\tilde{F}$ and see that $F$ is the desired section, as
\begin{equation}
    J_x(F)\in \Gamma(x,CA^\sharp)\text{ for all }x\in E
\end{equation}
and
\begin{equation}
    \|F\|_{\dot{C}^m}\le CA^\sharp.
\end{equation}

\section{Improvement of $k^{\sharp}$}\label{sec:improvement}

\begin{theorem}\label{thm:k sharp improvement}
Theorem \ref{thm:heart of the matter} holds with $k^\sharp=2^{\dim\vp}$ in \eqref{GR} and \eqref{def:norm of a bundle}.
\end{theorem}

Given a holding space $H(x)$, we will now let $\tilde{H}(x)$ denote its Glaeser refinement when $k^\sharp$ is taken to be the value in \eqref{GR} initially used to prove Theorem \ref{thm:main theorem}. Call this value $k_1$. We will use $H'(x)$ to denote the Glaeser refinement when $k^\sharp$ is taken to be $k_0:=2^{\dim\vp}$. Theorem \ref{thm:k sharp improvement} is easily seen to be implied by the following proposition.

\begin{proposition}\label{prop:same refinement}
Let $H(x)$ be a holding space. Then $\tilde{H}(x)=H'(x)$.
\end{proposition}

The containment $\tilde{H}(x)\subset H'(x)$ follows trivially from the definitions, as increasing $k^\sharp$ increases the number of $y_j$ near $x$ for which we must find polynomials in $H(y_j)$. Thus, the essence of the proof will be showing that $\tilde{H}(x)\supset H'(x)$.

A key to proving Proposition \ref{prop:same refinement} is the following result proven in \cite{US2020}.

\begin{theorem}\label{thm:US k sharp}
Let $S\subset \R^n$ be a finite set of diameter at most 1. For each $x\in S$, let $\vec{G}(x)\subset\vec{\p}$ be convex. Suppose that for every subset $S' \subset S$ with $\abs{S} \leq 2^{\dim \vec{\p}}$, there exists $F^{S'}\in C^m(\R^n,\R^d)$ such that $\|F^{S'}\|_{C^m(\R^n,\R^d)}\le 1$ and $J_x F^{S'}\in \vec{G}(x)$ for all $x\in S'$.
		
Then, there exists $F\in C^m(\R^n,\R^d)$ such that $\|F\|_{C^m(\R^n,\R^d)}\le \gamma$ and $J_x F\in \vec{G}(x)$ for all $x\in S$.
		
Here, $\gamma$ depends only on $m,n,d$, and $|S|$.


\end{theorem}


A scalar-valued version of Theorem \ref{thm:US k sharp} was originally proven in \cite{Shv08}. However, we state the vector-valued version for full generality, as this is particularly useful for the selection problem.

\begin{proof}[Proof of Proposition \ref{prop:same refinement}]
Pick $y_0\in E$ and let $P_0\in H'(y_0)$. Replacing each $H(y)$ with $H(y)-P_0$, we may assume $P_0$ is the zero polynomial.

Fix $\epsilon>0$ and let $M\geq 0$ be a constant depending only on $m,n,d$ and to be determined later. Choose $0<\delta<1/2$ so that for any $y_1,...,y_{k_0}\in B(y_0,\delta)\cap E$ there exist $P_1,...,P_{k_0}\in\vp$ such that
\begin{equation}
    P_j\in H(y_j)
\end{equation}
and
\begin{equation}
    |\D^\alpha(P_i-P_j)(y_j)|\le (\epsilon/M)|y_i-y_j|^{m-|\alpha|}\text{ for }|\alpha|\le m, 0\le i,j\le k_0.
\end{equation}

Let $y_1,...,y_{k_1}\in B(x,\delta)\cap E$ and write $S=\{y_0,y_1,...,y_{k_1}\}$. For $y_j\in S$, define $\vec{G}(y_j)$ to be $H(y_j)$ if $j\ge1$ and $\{P\}$ if $j=0$. Since $P\in H'(y_0)$, for any $S'\subset S$ satisfying $|S'|=k_0+1$ and $y_0\in S$, there exist polynomials $(P_z)_{z\in S'}$ such that
\begin{equation}\label{eq:z z'}
    |\D^\alpha(P_z-P_{z'})(z)|\le (\epsilon/M)|z-z'|^{m-\alpha}\text{ for }|\alpha|\le m, z,z'\in S'.
\end{equation}
In particular, taking $z'=y_0$ in \eqref{eq:z z'}, we see that
\begin{equation}
    |\D^\alpha P_z(z)|\le (\epsilon/M) \delta^{m-|\alpha|}\le(\epsilon/M).
\end{equation}

Therefore, by the classical Whitney extension theorem for finite sets (Lemma \ref{lemma:WET}), there exists $F_{S'}$ such that
\begin{equation}
    J_{y_j}F=P_j\in \vec{G}(y_j)\text{ for }y_j\in S'
\end{equation}
and
\begin{equation}
    \|F_{S'}\|_{C^m}\leq C(\epsilon/M),
\end{equation}
where $C$ depends only on $m,n,d$.

By adding $y_0$ if necessary, we see that for \textit{any} $S'\subset S$ with $|S'|=k_0+1$ there exists $F_{S'}$ such that
\begin{equation}
    J_{y_j}F_{S'}=P_j\in \vec{G}(y_j)\text{ for all }y_j\in S'
\end{equation}
and
\begin{equation}
    \|F_{S'}\|_{C^m}\leq C(\epsilon/M).
\end{equation}

Taking $\gamma=\gamma(m,n,d,k_1)$ as in Theorem \ref{thm:US k sharp}, there exists $F$ such that
\begin{equation}
    J_{y_j}F\in \vec{G}(y_j)\subset H(y_j)
\end{equation}
and
\begin{equation}
    \|F\|_{C^m(\R^n)}\le\gamma C(\epsilon/M).
\end{equation}

Setting $Q_j=J_{y_j}F$ and $Q_0=P_0$, we see that for all $0\le i,j\le k_1$, there exist $Q_1,...,Q_{k_1}\in\vp$ with
\begin{equation}
    Q_j\in H(y_j), 1\le j\le k_1
\end{equation}
and
\begin{equation}
    |\D^\alpha(Q_i-Q_j)(y_j)|\le \gamma C(\epsilon/M)|y_i-y_j|^{m-|\alpha|}\text{ for }|\alpha|\le m.
\end{equation}
Picking $M=C\gamma$, we have
\begin{equation}
    |\D^\alpha(Q_i-Q_j)(y_j)|\le\epsilon|y_i-y_j|^{m-|\alpha|}\text{ for }|\alpha|\le m, 0\le i,j,\le k_1.
\end{equation}

Since $y_1,...,y_{k_1}\in B(x,\delta)\cap E$ were arbitrary, $Q_0=P_0\in \tilde{H}(y_0)$. Thus, $\tilde{H}(x)\supset H'(x)$.

\end{proof}

\bibliographystyle{plain}
\bibliography{Whitney-bib}

\begin{thebibliography}{10}

\bibitem{BMP03}
Edward Bierstone, Pierre~D. Milman, and Wiesław Pawłucki.
\newblock Differentiable functions defined in closed sets. {A} problem of
  {W}hitney.
\newblock {\em Invent. Math.}, 151(2):329--352, 2003.

\bibitem{EH18}
Neil Epstein and Melvin Hochster.
\newblock Continuous closure, axes closure, and natural closure.
\newblock {\em Trans. Amer. Math. Soc.}, 370(5):3315--3362, 2018.

\bibitem{F05-J}
Charles Fefferman.
\newblock A generalized sharp {W}hitney theorem for jets.
\newblock {\em Rev. Mat. Iberoam.}, 21(2):577--688, 2005.

\bibitem{F05-L}
Charles Fefferman.
\newblock Interpolation and extrapolation of smooth functions by linear
  operators.
\newblock {\em Rev. Mat. Iberoam.}, 21(1):313--348, 2005.

\bibitem{F05-Sh}
Charles Fefferman.
\newblock A sharp form of {W}hitney's extension theorem.
\newblock {\em Ann. of Math. (2)}, 161(1):509--577, 2005.

\bibitem{F06}
Charles Fefferman.
\newblock {W}hitney's extension problem for {$ {C^m} $}.
\newblock {\em Ann. of Math. (2)}, 164(1):313--359, 2006.

\bibitem{F09-Data-3}
Charles Fefferman.
\newblock Fitting a {$C^m$}-smooth function to data {III}.
\newblock {\em Ann. of Math. (2)}, 170(1):427--441, 2009.

\bibitem{FI20-book}
Charles Fefferman and Arie Israel.
\newblock {\em Fitting Smooth Functions to Data}.
\newblock {CBMS} Regional Conference Series in Mathematics. American
  Mathematical Society, 2020.

\bibitem{FIL16}
Charles Fefferman, Arie Israel, and Garving~K. Luli.
\newblock Finiteness principles for smooth selections.
\newblock {\em Geom. Funct. Anal.}, 26(2):422--477, 2016.

\bibitem{FIL16+}
Charles Fefferman, Arie Israel, and Garving~K. Luli.
\newblock Interpolation of data by smooth non-negative functions.
\newblock {\em Rev. Mat. Iberoam.}, 33(1):305—324, 2016.

\bibitem{fefferman2021c2}
Charles Fefferman, Fushuai Jiang, and Garving~K. Luli.
\newblock ${C^2}$ interpolation with range restriction.
\newblock {\em https://arxiv.org/abs/2107.08272}, 2021.

\bibitem{FK09-Data-1}
Charles Fefferman and Bo'az Klartag.
\newblock Fitting a {$C^m$}-smooth function to data. {I}.
\newblock {\em Ann. of Math. (2)}, 169(1):315--346, 2009.

\bibitem{FK09-Data-2}
Charles Fefferman and Bo'az Klartag.
\newblock Fitting a {$C^m$}-smooth function to data. {II}.
\newblock {\em Rev. Mat. Iberoam.}, 25(1):49--273, 2009.

\bibitem{FK13}
Charles Fefferman and J\'{a}nos Koll\'{a}r.
\newblock Continuous solutions of linear equations.
\newblock In {\em From {F}ourier analysis and number theory to {R}adon
  transforms and geometry}, volume~28 of {\em Dev. Math.}, pages 233--282.
  Springer, New York, 2013.

\bibitem{FL14}
Charles Fefferman and Garving~K. Luli.
\newblock The {B}renner-{H}ochster-{K}ollár and {W}hitney problems for
  vector-valued functions and jets.
\newblock {\em Rev. Mat. Iberoam.}, 30(3):875--892, 2014.

\bibitem{FShv18}
Charles Fefferman and Pavel Shvartsman.
\newblock Sharp finiteness principles for {L}ipschitz selections.
\newblock {\em Geom. Funct. Anal.}, 28:1641--1705, 2018.

\bibitem{fefferman2021cm}
Charles~L. Fefferman and Garving~K. Luli.
\newblock ${C^m}$ semialgebraic sections over the plane.
\newblock {\em https://arxiv.org/abs/2101.06334}, 2021.

\bibitem{G58}
Georges Glaeser.
\newblock Étude de quelques algèbres tayloriennes.
\newblock {\em J. Analyse Math.}, 6:1--124, 1958.

\bibitem{Black19}
Fushuai {Jiang}.
\newblock {Nonnegative Whitney Extension Problem for $C^1(\mathbb{R}^n)$}.
\newblock {\em arXiv e-prints}, page arXiv:1912.06327, December 2019.

\bibitem{JL20-Ext}
Fushuai Jiang and Garving~K. Luli.
\newblock ${C^2(\mathbb{R}^2)}$ nonnegative interpolation by bounded-depth
  operators.
\newblock {\em Advances in Math.}, 375:107391, 2020.

\bibitem{JL20}
Fushuai Jiang and Garving~K. Luli.
\newblock Nonnegative $ {C}^2(\mathbb{R}^2)$ interpolation.
\newblock {\em Advances in Math.}, 375:107364, 2020.

\bibitem{JL20-Alg}
Fushuai Jiang and Garving~K. Luli.
\newblock Algorithms for nonnegative ${C^2(\mathbb{R}^2)}$ interpolation.
\newblock {\em https://arxiv.org/abs/2102.05777}, 2021.

\bibitem{US2020}
Fushuai Jiang, Garving~K. Luli, and Kevin O'Neill.
\newblock On the shape fields finiteness principle.
\newblock {\em Int. Mat. Res. Not.}, 2021.

\bibitem{RT-c}
R.~Tyrrell Rockafellar.
\newblock {\em Convex Analysis}.
\newblock Princeton Landmarks in Mathematics. Princeton University Press,
  Princeton, NJ, 1997, reprint of the 1970 original, princeton paperbacks
  edition, 2015.

\bibitem{Shv84}
Pavel Shvartsman.
\newblock Lipschitz sections of set-valued mappings and traces of functions
  from the {Z}ygmund class on an arbitrary compactum.
\newblock {\em Dokl. Akad. Nauk SSSR}, 276(3):559--562, 1984.

\bibitem{Shv02}
Pavel Shvartsman.
\newblock Lipschitz selections of set-valued mappings and {H}elly's theorem.
\newblock {\em J. Geom. Anal.}, 12(2):289--324, 2002.

\bibitem{Shv04}
Pavel Shvartsman.
\newblock Barycentric selectors and a {S}teiner-type point of a convex body in
  a {B}anach space.
\newblock {\em J. Funct. Anal.}, 210(1):1--42, 2004.

\bibitem{Shv08}
Pavel Shvartsman.
\newblock The {W}hitney extension problem and {L}ipschitz selections of
  set-valued mappings in jet-spaces.
\newblock {\em Trans. Amer. Math. Soc.}, 360(10):5529--5550, 2008.

\bibitem{Shv21-alg}
Pavel Shvartsman.
\newblock The core of a 2-dimensional set-valued mapping. {E}xistence criteria
  and efficient algorithms for lipschitz selections of low dimensional
  set-valued mappings.
\newblock {\em https://arxiv.org/abs/2010.04540}, 2021.

\bibitem{Shv21-core}
Pavel Shvartsman.
\newblock On the core of a low dimensional set-valued mapping.
\newblock {\em https://arxiv.org/abs/2102.07609}, 2021.

\bibitem{Shv01}
Pavel Shvartsman.
\newblock On {L}ipschitz selections of affine-set valued mappings.
\newblock {\em Geom. Funct. Anal.}, 11(4):2001, 840-868.

\bibitem{W34-1}
Hassler {W}hitney.
\newblock Analytic extensions of differentiable functions defined in closed
  sets.
\newblock {\em Trans. Amer. Math. Soc.}, 36(1):63--89, 1934.

\bibitem{W34-2}
Hassler {W}hitney.
\newblock Differentiable functions defined in closed sets. {I}.
\newblock {\em Trans. Amer. Math. Soc.}, 36(2):369--387, 1934.

\bibitem{W34-3}
Hassler {W}hitney.
\newblock Functions differentiable on the boundaries of regions.
\newblock {\em Ann. of Math. (2)}, 35(3):482--485, 1934.

\end{thebibliography}





\end{document}